%% file: main_arXiv.tex
\RequirePackage{fix-cm}

\documentclass[smallextended]{svjour3_arXiv}       % onecolumn (second format)

\smartqed  % flush right qed marks, e.g. at end of proof

\usepackage[a4paper,margin=1in]{geometry}    % wider, cleaner layout for arXiv
\usepackage{graphicx}

\makeatletter

\def\makeheadbox{}%

\newcommand{\TitleFont}{\fontsize{24}{30}\bfseries\selectfont}  % title
\newcommand{\SubtitleFont}{\fontsize{13}{16}\bfseries\selectfont} % subtitle
\newcommand{\AuthorFont}{\fontsize{11}{13}\selectfont}          % authors+affils
\newcommand{\DateFont}{\normalsize}                             % date

\def\@maketitle{%
  \newpage
  \normalfont
  \begin{center}
    {\TitleFont \@title \par}%
    \vskip 2em%
    \if!\@subtitle!\else
      {\SubtitleFont \@subtitle \par}%
      \vskip 1.5em%
    \fi
    {\AuthorFont
      \setbox0=\vbox{%
        \setcounter{auth}{1}\def\and{\stepcounter{auth}}%
        \hfuzz=2\textwidth\def\thanks##1{}\@author}%
      \setcounter{footnote}{0}%
      \global\value{inst}=\value{auth}%
      \setcounter{auth}{1}%
      \def\and{\unskip\nobreak\enskip{\boldmath$ \cdot $}\enskip\ignorespaces}%
      \@author \par
    }%
    \vskip 3.5em%
  \end{center}
  \noindent{\DateFont \@date}\par\vskip 2em%
}

\makeatother

\input{sections/00_header.tex}

\begin{document}

\title{\Large\centering Global Solutions to Non-Convex Functional\\ Constrained Problems with Hidden Convexity}

\titlerunning{\normalsize Global Solutions to Non-Convex Functional Constrained Problems} % if too long for running head

\author{\large \centering Ilyas
 Fatkhullin\textsuperscript{1,2,3} \and Niao He\textsuperscript{2} \and
 Guanghui Lan\textsuperscript{3} \and Florian Wolf\textsuperscript{1,4} }

% if too long for running head
\authorrunning{I. Fatkhullin, et al.}

\institute{
 Ilyas Fatkhullin:~\url{ilyas.fatkhullin@ai.ethz.ch}, Niao He:~ \url{niao.he@inf.ethz.ch},\\
 Guanghui Lan:~\url{george.lan@isye.gatech.edu},
 Florian Wolf:~\url{fwolf@caltech.edu}\at
 \and \textsuperscript{1}Department of Computer Science, ETH
 Zurich, Switzerland \textsuperscript{2}ETH AI Center, ETH Zurich,
 Switzerland. \textsuperscript{3} Department of Industrial and
 Systems Engineering, Georgia Institute of Technology, Atlanta, GA.
 \textsuperscript{4}The
 Computing \&
 Mathematical Sciences Department, California Institute of Technology,
 Pasadena, CA.}


\maketitle

\vspace{-1cm}

\begin{abstract}
 \input{./sections/00_abstract.tex}

 \keywords{constrained optimization \and non-convex optimization \and
  hidden convexity \and proximal point method \and bundle-level}
 % \PACS{PACS code1 \and PACS code2 \and more} CODES Nonconvex programming
 % Nonlinear programming Large-scale problems Methods of nonlinear
 % programming type
 \subclass{90C26 \and 90C30 \and 90C06} % \and 49M37}
\end{abstract}

\setcounter{tocdepth}{2} % show also subsubsection
% \tableofcontents
\input{sections/01_intro.tex}
\input{sections/02_hiddenConvexClass.tex}
\input{sections/03_PPM_analysis.tex}

\input{sections/07_numericalExamples.tex}

\input{sections/08_conclusion.tex}

\section*{Acknowledgment}
This work is supported by the Prof.\ Dr.-Ing.\ Erich Müller-Stiftung,
the ETH AI Center Doctoral Fellowship, the Swiss National Centre of Competence in Research
(NCCR) Automation under grant agreement No.\ 51NF40\_180545, the Swiss National Science Foundation, the Kortschak Scholars Program, and the Office of Naval Research under grant N00014-24-1-2654.
\printbibliography
% end of the main part

\newpage
% start with technical proofs
\appendix
\addtocontents{toc}{\protect\setcounter{tocdepth}{1}}
\input{sections/10_appendix.tex}
\end{document}

%% file: sections/00_header.tex
\input{./sections/00_header_files/00_preamble.tex}
\input{./sections/00_header_files/01_add_packages_here.tex}

\bibliography{./hidden_convex_refs_zotero, ./hidden_convex_refs_add_bib_here.bib}
\input{./sections/00_header_files/02_add_notation_here.tex}

\newtheorem{assumption}[theorem]{Assumption}

%% file: sections/00_header_files/00_preamble.tex
\usepackage[utf8]{inputenc}
\usepackage[T1]{fontenc}
\usepackage[ngerman,english]{babel}

\usepackage{comment}
\usepackage{enumerate}
\usepackage[shortlabels]{enumitem}

\usepackage{amsmath}
\usepackage{mathtools}
\usepackage{amsfonts}
\usepackage{mathrsfs}
\usepackage{amssymb}
\usepackage{dsfont} % for mathbb 1
\usepackage[theorems,skins]{tcolorbox} % add for colored boxes

\usepackage{etoolbox}
\renewcommand{\qed}{\hfill $\Box$}
\AtEndEnvironment{proof}{\qed}

\usepackage{lipsum}

\usepackage{algorithm}
\usepackage{algpseudocode}

\usepackage{aligned-overset}
\usepackage{nicefrac}

\usepackage{todonotes}

\usepackage{booktabs}
\usepackage[table]{xcolor}

\definecolor{darkblue}{rgb}{0,0,.5}
\definecolor{darkorange}{rgb}{0.8, 0.4, 0}
\definecolor{grayETH}{rgb}{0.33, 0.33, 0.33}
\definecolor{blueETH}{rgb}{0.129, 0.361, 0.686}
\definecolor{purpleETH}{rgb}{0.506757, 0.111486, 0.381757}
\definecolor{greenETH}{rgb}{0.398406, 0.454183, 0.14741}
\definecolor{bronzeETH}{rgb}{0.5568627, 0.40392, 0.07450980}

\definecolor{redETH}{rgb}{0.597173, 0.226148, 0.176678}
\definecolor{darkgreen}{HTML}{02862a}

\newcommand{\gray}[1]{\textcolor{grayETH}{#1}}

\newcommand{\orange}[1]{\textcolor{darkorange}{#1}}

\newcommand{\highlight}[1]{\orange{#1}}

\definecolor{linen}{HTML}{FAF0E6} % define linen color as mycolor

\usepackage[
 colorlinks=true,
 linkcolor=darkblue, urlcolor =darkblue, citecolor=darkblue,
 raiselinks=true, bookmarks=true, bookmarksopenlevel=1, bookmarksopen=true,
 bookmarksnumbered=true, hyperindex=true, plainpages=false,
 pdfpagelabels=true, pdfstartview=FitH, pdfstartpage=1,
 pdfpagelayout=OneColumn ]{hyperref} \usepackage{csquotes}

\newcommand{\Assref}[1]{Assumption~\ref{#1}}

%% file: sections/00_header_files/01_add_packages_here.tex
\usepackage{cleveref}
\usepackage{subcaption}
\usepackage{bbm}
\usepackage{tcolorbox}

\usepackage[
 hyperref,
 backend=biber,
 style=alphabetic, citestyle=alphabetic, maxalphanames=3, minalphanames=3,
 giveninits=true,
 backref, backrefstyle=none ]{biblatex}
\DeclareSourcemap{
 \maps[datatype=bibtex]{
  \map{
   \step[fieldset=doi,null]
   \step[fieldset=annotation,null]
   \step[fieldset=urldate,null]
   \step[fieldset=file,null]
   \step[fieldset=keywords,null]
   \step[fieldset=primaryclass,null]
   \step[fieldset=isbn,null]
   \step[fieldset=url,null]
   \step[fieldset=editor,null]
   \step[fieldset=issn,null]
   \step[fieldsource=doi, match=\regexp{.+}, final=true]
   \step[fieldset=eprint,null]
   \step[fieldset=archiveprefix,null]
  }
 }
}
\newcommand{\link}[1]{\href{#1}{\texttt{#1}}}
\newcommand{\githubLink}{\link{https://github.com/Flo-Wo/HiddenConvexityCode}}

\newcommand{\nfr}{\nicefrac}
\newcommand{\fr}{\frac}

\def\RR{{\mathbb{R}}}
\def\ZZ{{\mathbb{Z}}}
\def\PP{{\mathbb{P}}}
\def\NN{{\mathbb{N}}}

\def\SS{{\mathcal{S}}}
\def\AA{{\mathcal{A}}}
\def\VV{{\mathcal{V}}}
\def\WW{{\mathcal{W}}}

\DeclareMathOperator{\OOTildeSym}{\widetilde{\mathcal{O}}}
\DeclareMathOperator{\OOSym}{\mathcal{O}}
\newcommand{\OOTilde}[1]{{\OOTildeSym\!\left(#1\right)}}
\newcommand{\OO}[1]{{\OOSym\!\left(#1\right)}}
\newcommand{\Prob}[1]{{\mathbb{P}\!\left(#1\right)}}

\newcommand{\card}[1]{\left|#1\right|}

\def\epsilon{{\varepsilon}}
\newcommand{\Exp}[1]{{ \mathbb{E}}\!\left[#1\right]}

\newcommand{\Expu}[2]{{\mathbb{E}}_{#1}\!\left[#2\right]}

\newcommand{\innerEmpty}{\,\cdot\,}

\DeclarePairedDelimiter{\paren}{(}{)}
\DeclarePairedDelimiter{\brac}{[}{]} % brackets
\DeclarePairedDelimiter{\cbrac}{\{}{\}} % curly braces
\DeclarePairedDelimiter{\abs}{\lvert}{\rvert}
\DeclarePairedDelimiter{\norm}{\lVert}{\rVert}
\DeclarePairedDelimiterX{\IP}[2]{\langle}{\rangle}{#1, #2} % inner product

\DeclarePairedDelimiter{\plus}{[}{]_{+}}

\newcommand{\eqdef}{\coloneqq}
\newcommand{\define}{\coloneqq}

\DeclareMathOperator*{\argmin}{arg\,min}

\DeclareMathOperator{\proj}{proj}

%% file: sections/00_header_files/02_add_notation_here.tex
\def\uu{{u}}
\def\xx{{x}}

\def\xxBar{\bar{\xx}}
\def\xxHat{\hat{\xx}}

\def\xxOpt{\xx^{*}}
\def\yySlater{\bar{\yy}}
\def\uuOpt{\uu^{*}}

\def\uuHat{\hat{\uu}}
\def\xxOptGlobal{\xx^{*}_\mathrm{uncon}}
\def\uuOptGlobal{\uu^{*}_\mathrm{uncon}}
\def\xxZero{\xx^{(0)}}
\def\xxK{\xx^{(k)}}
\def\xxT{\xx^{(t)}}

\def\xxTnext{\xx^{(t+1)}}
\def\xxN{\xx^{(N)}}

\def\xxBarK{\xxBar^{(k)}}
\def\xxKnext{\xx^{(k+1)}}
\def\xxNnext{\xx^{(N+1)}}

\def\xxHatKnext{\xxHat^{(k+1)}}

\def\offset{{b}}

\def\rhoHat{{\hat{\rho}}}

\def\lambdaHat{{\widehat{\lambda}}}
\def\lambdaOpt{{\lambda^{*}}}

\def\phi{{\varphi}}

\def\yy{{y}}

\def\zz{{z}}
\def\zzT{{\zz}^{(t)}}
\def\zzTnext{{\zz}^{(t+1)}}

\def\zzTilde{\tilde{\zz}}
\def\zzZero{\zz^{(0)}}

\def\zzTUnder{\underline{\zz}^{(t)}}
\def\zzTminusOne{\zz^{(t-1)}}
\def\zzTminusOneUnder{\underline{\zz}^{(t-1)}}
\def\zzTminusTwo{\zz^{(t-2)}}
\def\zzTildeT{\zzTilde^{(t)}}
\def\zzTnext{\zz^{(t+1)}}
\def\vv{{v}}

\def\ss{{s}}

\def\setN{{[N]}}
\def\setT{{[T]}}

\newcommand{\phiK}[1]{{\phi^{(k)}_{#1}}}
\newcommand{\phiKshift}[1]{{\phi^{(k)}_{#1, \mathrm{sh}}}}

\def\XX{{\mathcal{X}}}

\def\UU{{\mathcal{U}}}
\def\DDU{{\mathcal{D}_\UU}}
\def\DDX{{\mathcal{D}_\XX}}

\def\DDUsq{{\mathcal{D}^2_\UU}}

\newcommand{\rb}[1]{\left(#1\right)}
\newcommand{\cb}[1]{\left\{#1\right\}}
\def\Fopt{{F_1^*}}

\def\LB{{\mathsf{LB}}}
\def\UB{{\mathsf{UB}}}
\newcommand{\FLB}[1]{{F_{#1}^{\LB}}}
\newcommand{\FUB}[1]{{F_{#1}^{\UB}}}

\def\feas{{\mathcal{F}}}
\def\infeas{{\mathcal{I}}}
\def\Tinner{{T_{\mathrm{in}}}}

\newcommand{\TinnerUp}[1]{T_{\mathrm{in}}^{\mathrm{#1}}}
\def\Ttotal{{T_{\mathrm{tot}}}}
\newcommand{\TtotalUp}[1]{T_{\mathrm{tot}}^{\mathrm{#1}}}
\def\nonSmoothLabel{{\mathrm{n}\text{-}sm}}
\def\smoothLabel{{sm}}
\def\smoothSlaterLabel{{sm-Slat}}

\def\Oracle{{\mathcal{A}}}

\def\epsin{{\epsilon_{\mathrm{in}}}}

\DeclareMathOperator{\IPPM}{IPPM}

\def\indicator{{\mathbbm{1}}}

\def\aa{{a}}
\def\ss{{s}}

\newcommand*\dd{\: \mathrm{d}}

%% file: sections/00_abstract.tex
Constrained non-convex optimization is fundamentally challenging, as global
solutions are generally intractable and constraint qualifications may not
hold. However, in many applications, including safe policy optimization in
control and reinforcement learning, such problems possess hidden convexity,
meaning they can be reformulated as convex programs via a nonlinear
invertible transformation. Typically such transformations are implicit or
unknown, making the direct link with the convex program impossible. On the
other hand, (sub-)gradients with respect to the original variables are
often accessible or can be easily estimated, which motivates algorithms
that operate directly in the original (non-convex) problem space using
standard (sub\text{-})gradient oracles. In this work, we develop the first
algorithms to provably solve such non-convex problems to global minima.
First, using a modified inexact proximal point method, we establish global
last-iterate convergence guarantees with
$\widetilde{\mathcal{O}}(\varepsilon^{-3})$ oracle complexity in non-smooth
setting. For smooth problems, we propose a new bundle-level type method
based on linearly constrained quadratic subproblems, improving the oracle
complexity to $\widetilde{\mathcal{O}}(\varepsilon^{-1})$. Surprisingly,
despite non-convexity, our methodology does not require any constraint
qualifications, can handle hidden convex equality constraints, and achieves
complexities matching those for solving unconstrained hidden convex
optimization.

%% file: sections/01_intro.tex
\section{Introduction}
\label{sec:Introduction}
Non-convex constrained optimization problems (with possibly non-smooth
objectives and constraints) arise frequently in many modern applications. In this work, we study a sub-class of non-convex problems of the form
\begin{align}
 \begin{aligned}
  \min_{\xx \in \XX} \;F_1(\xx), \; \qquad
  \text{s.t. }\; F_2(\xx) \leq 0,
 \end{aligned}
 \label{eq:IntroEquation}
\end{align}
where $\XX \subset \RR^d$ is a closed convex set, and $F_1, F_2$ are possibly \emph{non-convex} with respect to variable $\xx$.
\footnote{We assume $F_2$ is scalar-valued, but our results are extendable to the vector-valued case.} A central running assumption in this work is
that problem \eqref{eq:IntroEquation} admits a \emph{convex reformulation}
\begin{align}
 \begin{aligned}
  \min_{\uu \in \UU} \;H_1(\uu)
  , \; \qquad
  \text{s.t. }\;  H_2(\uu) \leq 0
 \end{aligned}
 \label{eq:IntroEquationConvex}
\end{align}
$$
 \textit{via variable change:} \quad  u = c(x),
$$
where $H_1, H_2$ are convex functions defined over a
closed convex set $\UU \subset \RR^d$, and $c: \XX \to \UU$ is an invertible map
(with $c^{-1}$ denoting its inverse).\footnote{We show how the invertibility assumption can be relaxed in \Cref{sec:HiddenConvexProblemClass}.}

This property, often referred to as \emph{hidden convexity}, appears in
diverse applications, including policy optimization in optimal control
\cite{sunLearningOptimalControllers2021,andersonSystemLevelSynthesis2019,fatkhullinOptimizingStaticLinear2021,
 zhengBenignNonconvexLandscapes2023, watanabeRevisitingStrongDuality2025}
and reinforcement learning
\cite{zhangVariationalPolicyGradient2020,barakatReinforcementLearningGeneral2023a,
 yingPolicybasedPrimalDualMethods2024}, variational inference
\cite{wu2024understanding,sunNaturalGradientVI2025}, generative models
\cite{kobyzevNormalizingFlowsIntroduction2021}, supply chain and revenue
management \cite{fengSupplyDemandFunctions2018,
 chenEfficientAlgorithmsClass2025}, geometric programming
\cite{xiaSurveyHiddenConvex2020}, neural network training
\cite{bachBreakingCurseDimensionality2017,
 wangHiddenConvexOptimization2022}, and non-monotone games
\cite{vlatakis-gkaragkounisSolvingMinMaxOptimization2021,
 mladenovicGeneralizedNaturalGradient2022,
 sakosExploitingHiddenStructures2023, dorazioSolvingHiddenMonotone2025}, to
name just a few. \Cref{fig:IntroMinExample} provides illustrative examples
of constrained hidden convex problems, which we will cover in more detail
in \Cref{subsec:ExampleNonsmoothConstrainedLeastSquares}.

\begin{figure}
 \centering
 \begin{subfigure}{0.49\textwidth}
  \centering
  \includegraphics[height=.83\linewidth, keepaspectratio]{
   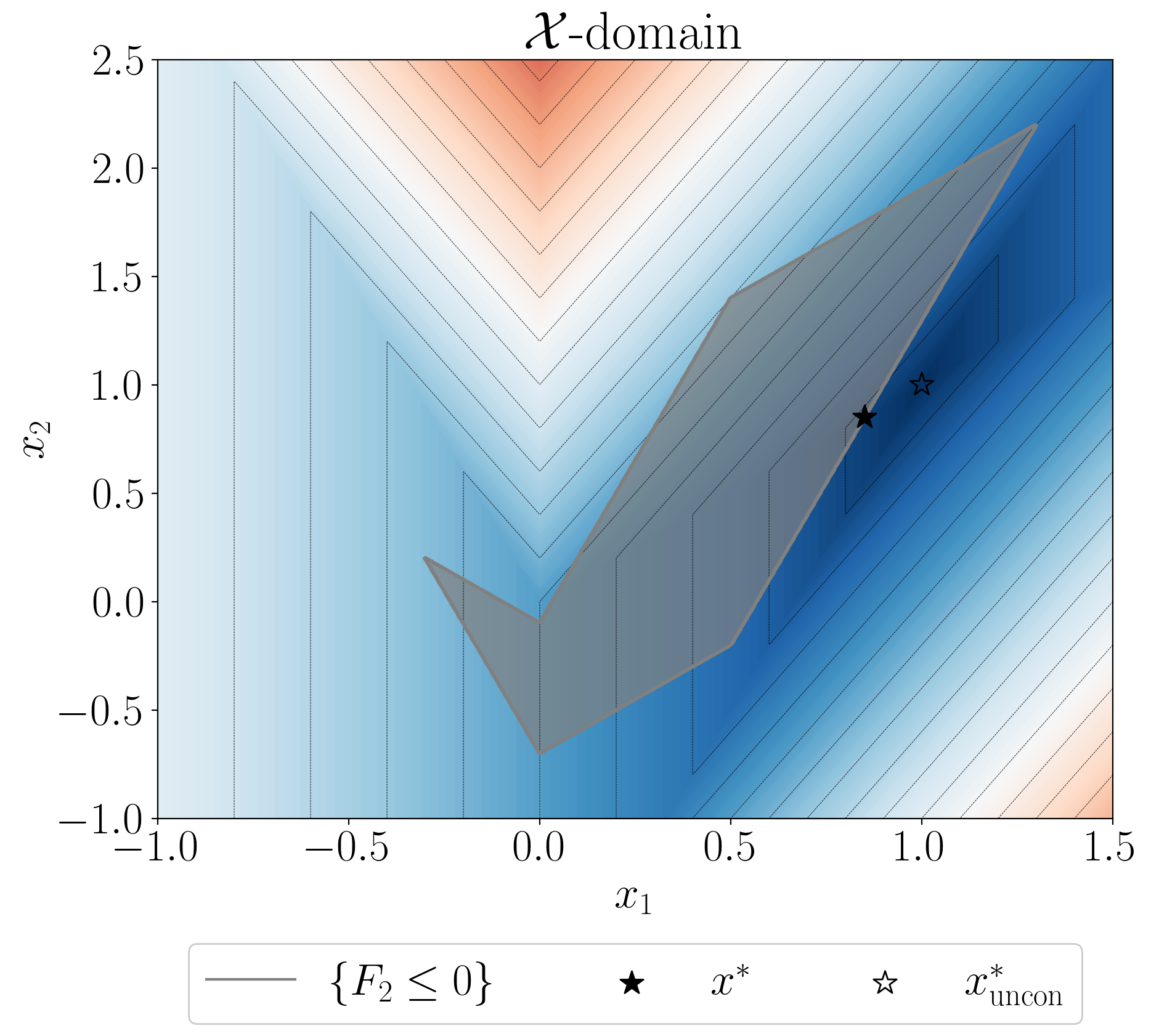
  }
  \caption{\textit{Non-convex} formulation for \eqref{eq:ToyExampleLeastSquares}.}
  \label{fig:IntroMinExampleNonSmoothX}
 \end{subfigure}
 \hfill
 \begin{subfigure}{0.49\textwidth}
  \centering
  \includegraphics[height=.83\linewidth,keepaspectratio]{
   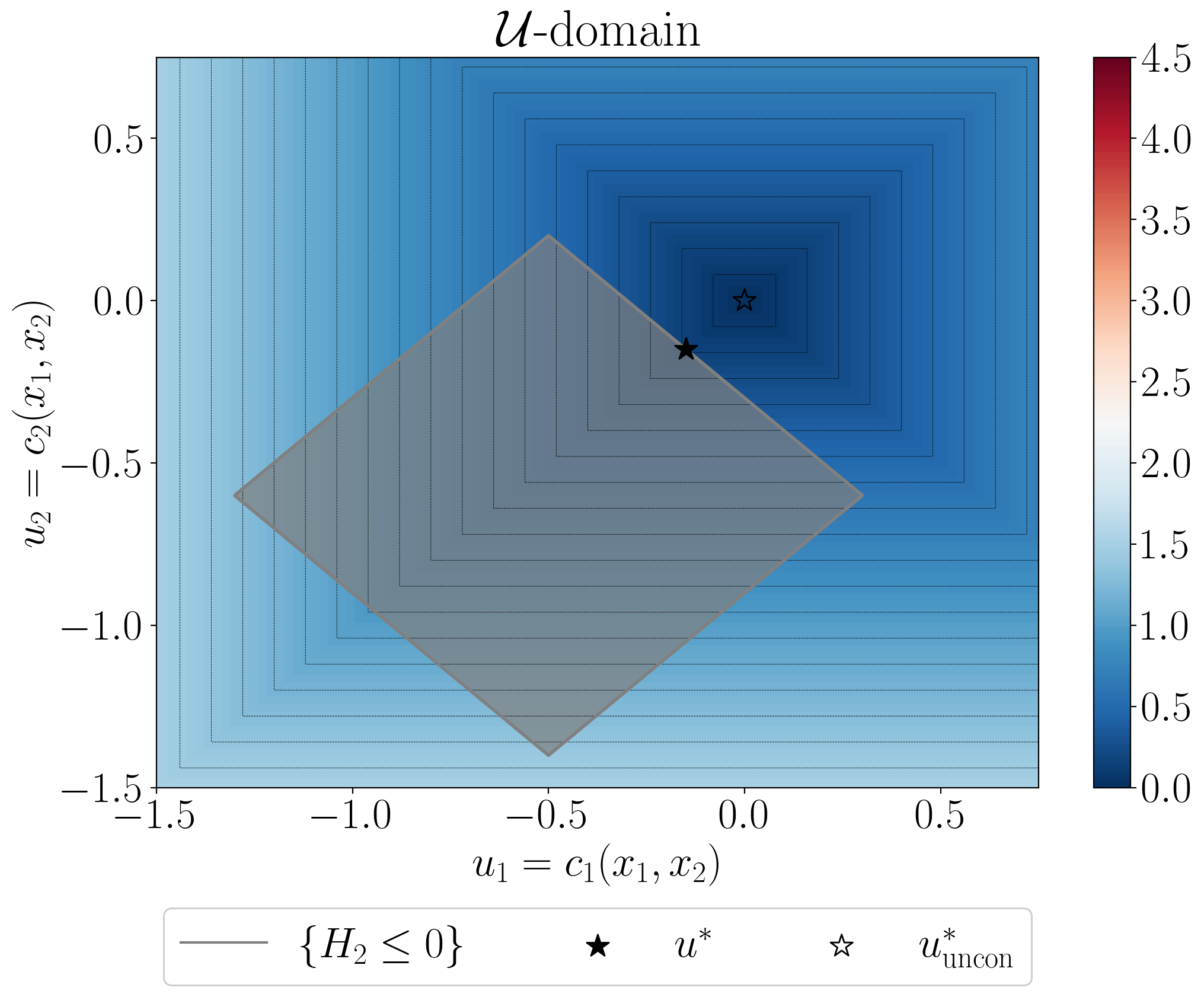
  }
  \caption{\textit{Convex} reformulation for \eqref{eq:ToyExampleLeastSquares}.}
  \label{fig:IntroMinExampleNonSmoothU}
 \end{subfigure}
 \begin{subfigure}{0.49\textwidth}
  \centering
  \includegraphics[height=.83\linewidth,keepaspectratio]{
   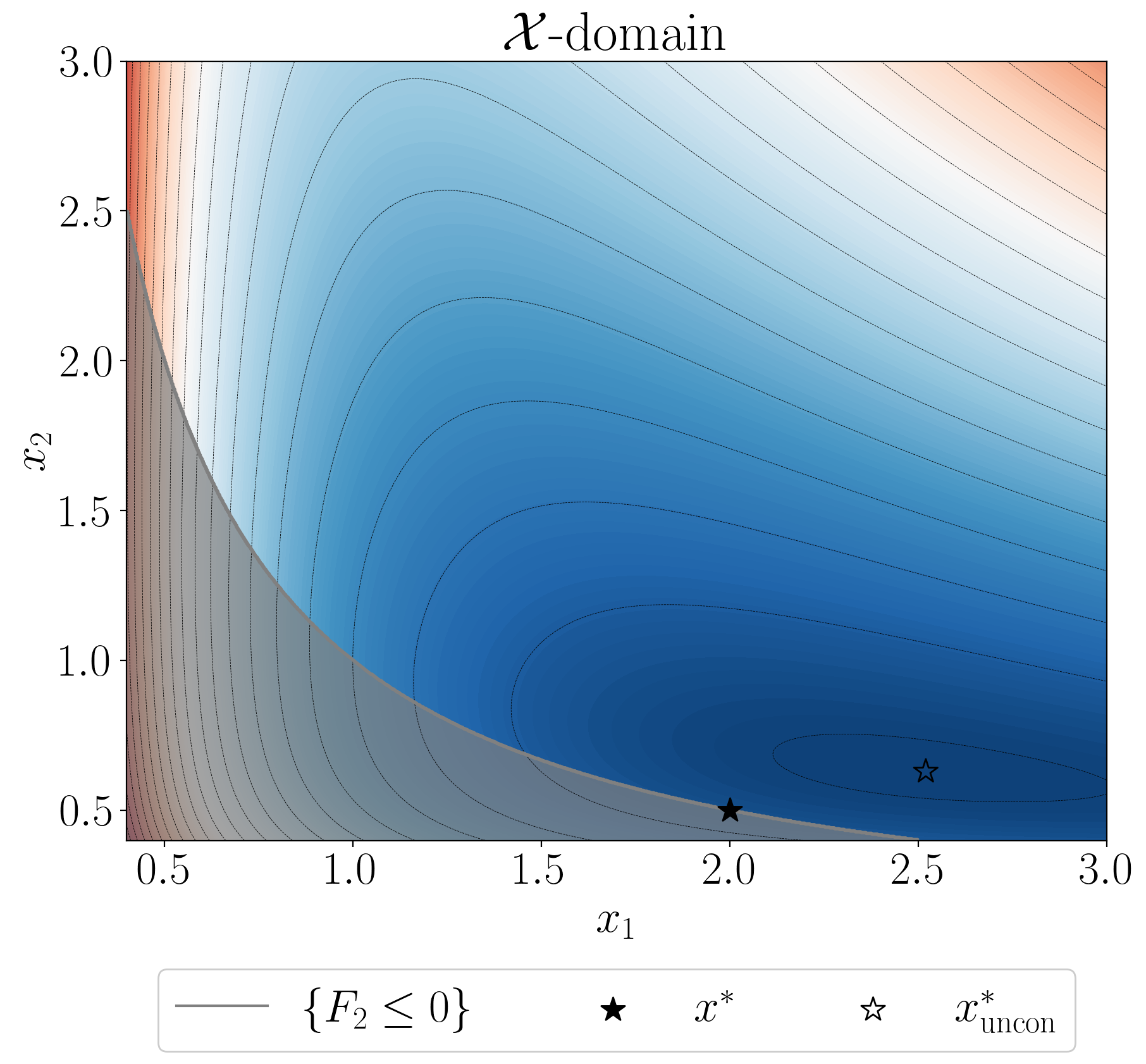
  }
  \caption{\textit{Non-convex} formulation for \eqref{eq:ToyExampleGeometricProgramming}.}
  \label{fig:IntroMinExampleSmoothX}
 \end{subfigure}
 \begin{subfigure}{0.49\textwidth}
  \centering
  \includegraphics[height=.83\linewidth,keepaspectratio]{
   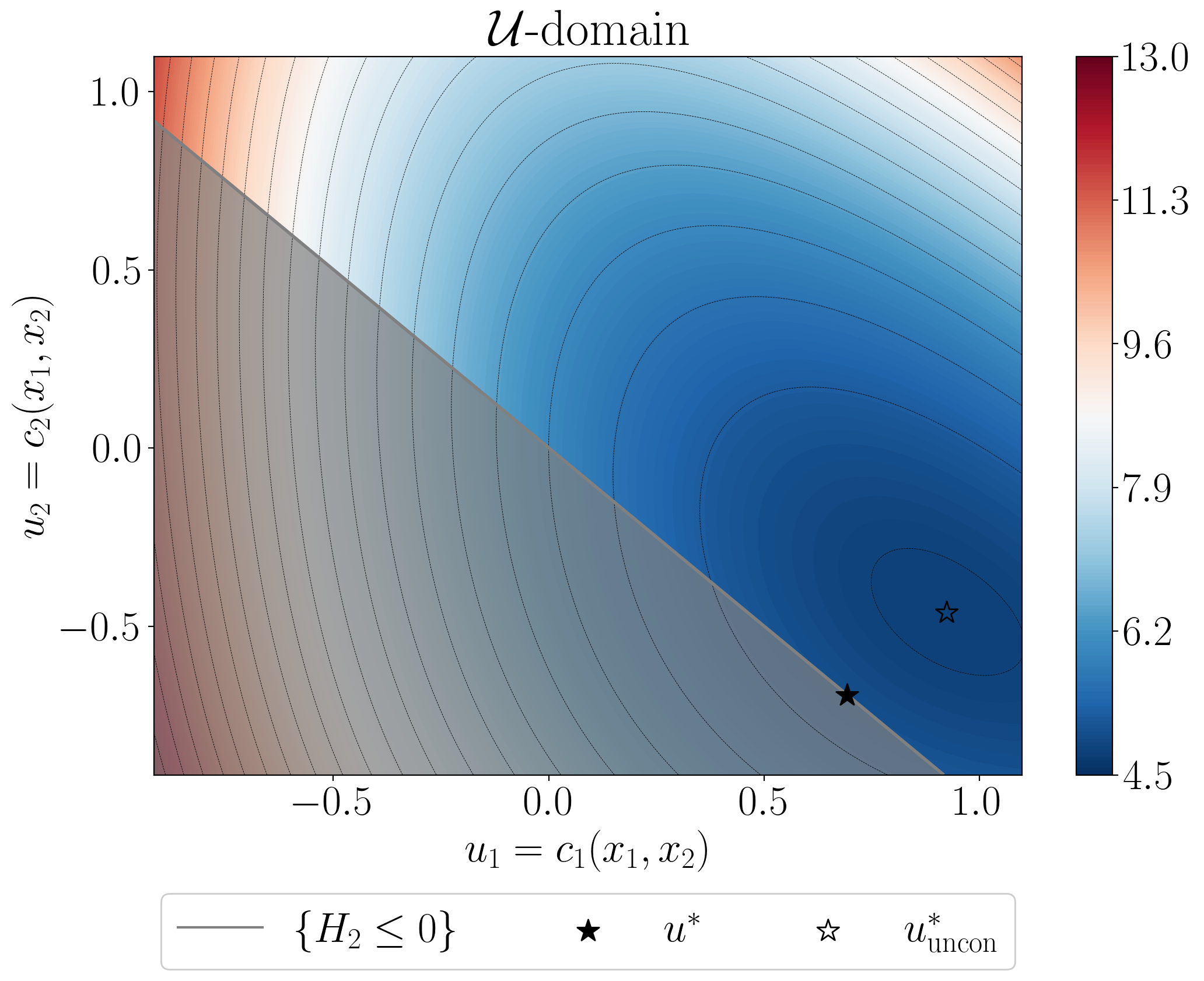
  }
  \caption{\textit{Convex} reformulation for \eqref{eq:ToyExampleGeometricProgramming}.}
  \label{fig:IntroMinExampleSmoothU}
 \end{subfigure}
 \caption[Nesterov's Non-Smooth Chebychev-Rosenbrock Function]{Illustrative examples of non-smooth (top) and smooth (bottom) hidden convex problems:
  (a) and (b) -- \textit{Constrained Non-Linear Least Squares}; (c) and (d) -- \textit{Constrained Geometric Programming}, see \eqref{eq:ToyExampleLeastSquares} and \eqref{eq:ToyExampleGeometricProgramming} in~\Cref{subsec:ExampleNonsmoothConstrainedLeastSquares} for details. The plots illustrate in color the \textit{level sets} of the non-convex formulation (left) and the convex formulation (right). The \gray{feasible sets} are shown as gray regions: $\{F_2 \leq 0\}$ in the $\XX$-domain, and $\{H_2 \leq 0\}$ in the $\UU$-domain. We use the {notation} $\xxOptGlobal$ to denote the optimum of $\min_\xx F_1(\xx)$ without constraints and $\xxOpt$ the minimizer under constraints; analogously, $\uuOptGlobal$ and $\uuOpt$ denote their counterparts in $\UU$. } \label{fig:IntroMinExample}
\end{figure}

Although the reformulation in \eqref{eq:IntroEquationConvex} is convex, in
practice, the transformation $c(\innerEmpty{})$ is often either difficult
to compute or entirely unknown
\cite{watanabeRevisitingStrongDuality2025,yingPolicybasedPrimalDualMethods2024}.
Consequently, solving the convex reformulation and recovering the solution
to \eqref{eq:IntroEquation} is generally impossible. On the other hand,
directly solving the non-convex problem \eqref{eq:IntroEquation} using the
standard (sub-)gradient methods is a well-established practice, which may
already lead to approximate global solutions. A recent work
\cite{fatkhullinStochasticOptimizationHidden2024} justifies the use of
(sub-)gradient methods applied to the hidden convex problem
\eqref{eq:IntroEquation} in the special case without functional constraint
($F_2(\innerEmpty{}) \equiv 0$). They establish global convergence of
(sub-)gradient methods in the function value, i.e., $F_1(\xxN) - F_1^* \leq
 \varepsilon$, in smooth and non-smooth setups.

However, the presence of functional constraints significantly complicates
the analysis. First, extending the standard schemes from convex
optimization (such as primal-dual, switching (sub-)gradient, bundle level)
seems challenging, and we are not aware of satisfactory analysis under
hidden convexity in the literature. Second, the direct application of
convergence guarantees for general non-convex (or weakly convex)
constrained optimization to our problem \eqref{eq:IntroEquation} yields
rather weak guarantees of reaching merely an approximate Karush-Kuhn-Tucker
(KKT) point under strong constraint qualification (CQ)
assumptions~\cite{maQuadraticallyRegularizedSubgradient2020,
 boobStochasticFirstorderMethods2023, jiaFirstOrderMethodsNonsmooth2025}.

\textbf{Contributions.} To address these challenges, we explore two distinct approaches based on (i) the
proximal point framework and (ii) a bundle-level idea.
In both cases, we design algorithms suitable for our hidden convex problems using only the oracle access to
(sub-)gradients and function values of $F_1,$ $F_2,$ which allows us to establish the first guarantees to find an
$(\varepsilon, \varepsilon)$-approximate global minima of \eqref{eq:IntroEquation}, i.e.,
$F_1(\xxN) - F_1^* \leq \varepsilon$ and $F_2(\xxN) \leq \varepsilon$.

\begin{enumerate}[start=1,label={(\bfseries C\arabic*)}]
 \item \textbf{Constrained Proximal Point Approach.}
       We propose a modified Inexact Proximal Point Method (IPPM) with shifted constraints under hidden convexity.
       The key technical novelty of our convergence analysis lies in showing that our IPPM subproblems always satisfy
       Slater's condition without imposing any CQs on the original problem. In our analysis, a Slater point
       is explicitly constructed using the variable transformation $c(x).$ Using this key insight, we apply existing
       algorithms for strongly convex constrained optimization to approximate each IPPM subproblem using the standard
       (sub-)gradient oracle access. In particular, in the non-smooth setting, we use the \textit{Switching Sub-Gradient}
       (SwSG) method as an inner solver to derive $\OOTilde{\varepsilon^{-3}}$ oracle complexity without any CQs.
       In the smooth setting, we use the \textit{Accelerated Constrained Gradient Descent} (ACGD) method, achieving
       $\OOTilde{\varepsilon^{-2}}$ (or $\OOTilde{\theta^{-1} \varepsilon^{-1}}$) oracle complexity without any CQs
       (or with $\theta$-Slater's condition). We refer to \Cref{tab:OverviewRatesTableNonSmooth,tab:OverviewRatesTableSmooth}
       for a summary.
 \item \textbf{Bundle-level Approach.}
       To improve the complexity in the smooth case when Slater's condition may not hold,
       we propose a new Shifted Bundle-level method. The main algorithmic novelty is a carefully selected shift of the
       linear constraints at each subproblem, which allows convergence of the algorithm even without convexity in the $\XX$
       space. When the optimal value of the constrained problem $\Fopt$ is available, we design a
       \textit{Shifted Star Bundle-level} algorithm (S-StarBL), which attains $\OOTilde{\varepsilon^{-1}}$ oracle complexity
       without any CQs. In the setting when $\Fopt$ is unknown, we invent an \textit{adaptive line-search}
       procedure involving an exact penalty as the convergence criterion. The resulting scheme (referred to as S-BL+AdaLS)
       achieves the $\OOTilde{\lambda^* \varepsilon^{-1}}$ complexity under strong duality, where $\lambda^*$ is the optimal
       dual variable. We refer to \Cref{tab:OverviewRatesTableSmooth} for a summary.
\end{enumerate}

Both approaches modify classical convex optimization methods to handle
hidden convexity using shifts in the constraints, and both are the first to
achieve global optimality guarantees for such problems. Our IPPM approach
relies on regularization, ensuring feasibility through an explicit Slater
point construction and achieving reliable convergence under minimal
assumptions. In contrast, the Bundle-level method adopts a cutting-plane
strategy with shifted constraints, attaining a faster
$\OOTilde{\varepsilon^{-1}}$ complexity for smooth problems without
Slater's condition, requiring only the weaker strong duality assumption. We
also validate our methods on a non-smooth constrained non-linear least
squares, and on a smooth geometric programming problem
in~\Cref{chapter:Experiments}.\footnote{The source code for the numerical
 experiments and illustrations is publicly available under \githubLink{}.}

\input{tables/allRates_table1/allRates_table1.tex}

\subsection{Related Work}\label{subsec:related_work}

Motivated by applications in reinforcement learning and revenue management,
recent works have explored the concept of \textit{hidden convexity}
\cite{zhangVariationalPolicyGradient2020,zhangConvergenceSampleEfficiency2021,chenNetworkRevenueManagement2022,
 ghaiNonconvexOnlineLearning2022,fatkhullinStochasticOptimizationHidden2024,
 chenEfficientAlgorithmsClass2025}. This concept was studied in its most
general form in \cite{fatkhullinStochasticOptimizationHidden2024} within
the same framework as ours and more recently in
\cite{bhaskaraDescentMisalignedGradients2024} in the context of misaligned
gradients, under additional assumptions on $c(\innerEmpty{})$. Notably,
\cite{chancelierConditionalInfimumHidden2021} develop properties of the
conditional infimum to detect hidden convexity. In \textit{convex
 reinforcement learning}, a related line of research focuses on policy
gradient (PG) methods \cite{zhangVariationalPolicyGradient2020,
 zhangConvergenceSampleEfficiency2021,
 barakatReinforcementLearningGeneral2023a}, with more recent work
\cite{barakatScalableGeneralUtility2025}, aiming to explicitly approximate
the transformation function $c(\innerEmpty{})$ for large state-action
spaces. However, it is unclear how to extend these methods to the setting
with hidden convex constraints. Beyond reinforcement learning,
\cite{chenNetworkRevenueManagement2022, chenEfficientAlgorithmsClass2025}
study global convergence of gradient-based methods for hidden convex
objectives in \textit{revenue management} over a simple constraint $\XX$,
and \cite{ghaiNonconvexOnlineLearning2022} studied a non-convex online
learning scenario structurally related, with stronger assumptions on the
reparametrization map $c(\innerEmpty{})$, cf. Assumption 1, 2, and 3
therein.

\paragraph{Constrained Optimization.}
There exists a variety of literature on problems of type
\eqref{eq:IntroEquation} for $F_1, F_2$ \textit{(strongly) convex} to find
a global $(\varepsilon, \varepsilon)$-optimal point
\cite{nesterovLecturesConvexOptimization2018}. Important approaches can be
classified into penalty-based
\cite{auslenderPenaltyproximalMethodsConvex1987,
 bertsekasNonlinearProgramming1999, aybatFirstOrderAugmentedLagrangian2012,
 lanIterationcomplexityFirstorderPenalty2013,
 lanIterationcomplexityFirstorderAugmented2016,
 patrascuAdaptiveInexactFast2017, xuAcceleratedFirstOrderPrimalDual2017,
 necoaraComplexityFirstorderInexact2019,
 tran-dinhProximalAlternatingPenalty2019, liuNonergodicConvergenceRate2019,
 xuIterationComplexityInexact2021, zhuOptimalLowerUpper2023}, primal-dual
\cite{nemirovskiProxMethodRateConvergence2004,
 nesterovDualExtrapolationIts2007,
 chambolleFirstOrderPrimalDualAlgorithm2011,
 tran-dinhPrimalDualAlgorithmicFramework2015,
 chambolleErgodicConvergenceRates2016,
 tran-dinhSmoothPrimalDualOptimization2018,
 xuPrimalDualStochasticGradient2020, hamedaniPrimalDualAlgorithmLine2021,
 zhuNewPrimalDualAlgorithms2022, zhangSolvingConvexSmooth2022,
 thekumparampilLiftedPrimalDualMethod2022a,
 khalafiAcceleratedPrimalDualMethods2023,
 boobStochasticFirstorderMethods2023, boobOptimalPrimalDualAlgorithm2024,
 linFasterAcceleratedFirstorder2024}, primal only, e.g. switching
sub-gradient methods
\cite{polyakGeneralMethodSolving1967,lanAlgorithmsStochasticOptimization2020,
 xuCRPONewApproach2021,islamov2025safeef}, level-set and bundle methods
\cite{lemarechalNewVariantsBundle1995,ben2005non,
 zhangPrimalMethodsVariational2025,linLevelSetMethodConvex2018,
 linLevelSetMethodsFiniteSum2018, aravkinLevelsetMethodsConvex2019}. When a
feasible starting point $\xxZero \in \cbrac{F_2 \leq 0}$, is available, one
can consider (log\mbox{-})barrier methods, see e.g.,
\cite{nesterovBarrierSubgradientMethod2011}.

Recently, there has been growing interest in the design and analysis of
algorithms for constrained \textit{non-convex} optimization problems,
aiming to find approximate KKT points. Contrary to the convex case, the
results for finding approximate KKT points are often limited to asymptotic
guarantees \cite{birginGlobalMinimizationUsing2010,
 curtisAdaptiveAugmentedLagrangian2016, sahinInexactAugmentedLagrangian2019,
 wangGlobalConvergenceADMM2019, marchiPenaltyBarrierFramework2024}. The
Proximal Point framework appears beneficial to establish non-asymptotic
guarantees in some restricted settings. Relevant results have been
established under linear constraints
\cite{hajinezhadPerturbedProximalPrimal2019,
 zhangProximalAlternatingDirection2020, zhangGlobalDualError2022,
 meloProximalAugmentedLagrangian2024}, equality constraints
\cite{xieComplexityProximalAugmented2021,
 liRateimprovedInexactAugmented2021}, convex inequality constraints
\cite{liAugmentedLagrangianBased2021, kongIterationComplexityProximal2023}.
The Augmented Lagrangian methods are frequently used in the above mentioned
works; we refer to \cite{dengAugmentedLagrangianMethods2025} for a more
general overview, which covers both convex and non-convex setups.
Non-asymptotic guarantees for finding a KKT point under equality
constraints are available, e.g., in
\cite{cartisCorrigendumComplexityFinding2017}. The proximal point framework
has been used to establish non-asymptotic rates for equality constraints
\cite{kongComplexityQuadraticPenalty2019} and more recently in
\cite{linComplexityInexactProximalpoint2022} for equality and inequality
constraints.
% %%%%%%%%%%%%%%%%%%%%%%%%% NON—SMOOTH
In the {non-smooth weakly convex} setting, key contributions for finding
approximate KKT points include double-loop proximal point method
(PPM)-based approaches, using either primal methods
\cite{maQuadraticallyRegularizedSubgradient2020,jiaFirstOrderMethodsNonsmooth2025}
or primal-dual methods
\cite{boobStochasticFirstorderMethods2023,boobOptimalPrimalDualAlgorithm2024}
for solving the inner problem. Similarly,
\cite{boobLevelConstrainedFirst2025} apply a PPM-based technique for
non-convex constrained optimization problems that can be decomposed into a
sum of an $L$-smooth and non-smooth function. Huang et al.
\cite{huangOracleComplexitySingleLoop2023} proposed a single-loop switching
sub-gradient method that employs a sophisticated switching step-size rule
inspired by \cite{davisSubgradientMethodsSharp2018} to address non-convex
constraints effectively.

In the special case of convex constrained Markov Decision Processes, the
\textit{hidden convex} structure has been studied. To the best of our
knowledge, the work by Ying et al.
\cite{yingPolicybasedPrimalDualMethods2024} is the only one that proposes
solution methods for problems of type \eqref{eq:IntroExampleCvx-CMDP} using
a momentum-based variance-reduced Primal-Dual Policy Gradient scheme.
However, this work has a serious limitation due to their weak satisfactions
of the constraint. The convergence guarantee is given in the form
$\frac{1}{N}\sum_{k=0}^{N-1} F_1(\xxK) \leq \epsilon$,
$\plus{\frac{1}{N}\sum_{k=0}^{N-1} F_2(\xxK)} \leq \epsilon$, where
$(\xxK)_{k}$ are the iterates of the algorithm. Such notion of constraint
satisfaction is weak because it permits cancellations of constraint
violations and does not guarantee an $(\varepsilon, \varepsilon)$-globally
optimal solution within the sequence $\paren{\xxK}_{k}$. In fact, even the
construction of approximate globally optimal point by averaging or sampling
from $\{\xxK\}_{k\geq1}$ is not possible in general. Thus, the design of
globally convergent methods for solving constrained hidden convex problems
remains an open problem.

\paragraph{Constraint Qualifications (CQs).}
Addressing non-convex constrained optimization problems typically requires
imposing very demanding constraint qualification conditions (CQs) to ensure
convergence to approximate KKT points
\cite{maQuadraticallyRegularizedSubgradient2020,
 boobStochasticFirstorderMethods2023, huangOracleComplexitySingleLoop2023,
 boobLevelConstrainedFirst2025, jiaFirstOrderMethodsNonsmooth2025}. The only
exception is the work of Jia and Grimmer
\cite{jiaFirstOrderMethodsNonsmooth2025}, which establishes convergence to
approximate Fritz-John stationarity
\cite{mangasarianFritzJohnNecessary1967} without assuming a CQ. However,
such results do not translate to a global minima even under convexity.
Notably, our algorithms converge to global minima under hidden convexity
and do not require any CQ conditions. This is achieved by an appropriate
modification of existing methods for convex optimization (e.g., inexact
PPM, bundle-level) and a careful use of the hidden convexity structure.

\subsection{Notations and Standard Assumptions}

In the following, we briefly revisit some basic notation. Throughout this
work, we define $\setN \define \{0, \ldots, N\}$ for $N\in \NN.$ We denote
with $\IP{\innerEmpty{}}{\innerEmpty{}},$ the inner product in $\RR^d$
along with its induced Euclidean norm $\norm{\innerEmpty} =
 \norm{\innerEmpty{}}_2,$ where $d \in \NN$ indicates the ambient dimension.
We call the map $c : \XX \to \UU$ invertible if there exists a map $c^{-1}:
 \UU \to \XX$, called inverse, with $c^{-1}(c(\xx)) = \xx$ and
$c(c^{-1}(\uu)) = \uu$ for all $\xx \in \XX$ and $\uu \in \UU$
respectively. The set $\UU\subset\RR^d$ is called convex if for all $\uu,
 \vv \in \UU$, and $\alpha \in [0,1]$ we have $(1-\alpha) \uu + \alpha \vv
 \in \UU$, and with $\DDU \define \sup_{\uu, \vv \in \UU} \norm{\uu - \vv}$
we denote its diameter. To deal with properties of the objective and the
constraints simultaneously, we use $i = 1,2$ to simplify notation. If for a
function $H_i:\UU \to \RR$ there exists $\mu_{H_i} \geq 0$ such that for
all $\uu, \vv \in \UU$ and $\alpha \in [0,1]$ it holds $H_i((1-\alpha)\uu +
 \alpha \vv) \leq (1-\alpha)H_i(\uu) + \alpha H_i(\vv) - \frac{(1-\alpha)
  \alpha \mu_{H_i}}{2} \norm{\uu - \vv}^2$ we call $H_i$ strongly convex if
$\mu_{H_i} > 0$ and convex if $\mu_{H_i} = 0$ on $\UU$ respectively. With
$\RR_{>0}^d$ we denote the set $\cbrac{\xx \in \RR^d \;\vert\; \xx_i > 0,
  \; i=1,\ldots, d}$. For a convex set $\XX\subset \RR^d$ and a point
$\yy\in\RR^d$ we denote with $\Pi_{\XX}(\yy) \define \min_{\xx \in \XX}
 \norm{\yy - \xx}$ its projection onto $\XX$; by $\delta_{\XX}$ we denote
the indicator function, i.e., $\delta_{\XX}(\xx)=0$ if $\xx \in \XX$ and
$\delta_{\XX}(\xx)=\infty$ otherwise. The relative interior of $\XX$ is
denoted by $\text{relint}(\XX).$

A function $F: \XX \to \RR$ is called $\rho$-weakly convex ($\rho$-WC) if
for any fixed $\yy \in \XX$, the function $F_{\rho}(\xx, \yy) \define
 F(\xx) + \frac{\rho}{2} \norm{\xx - \yy}^2$ is convex in $\xx \in \XX$. The
(Fréchet) sub-differential of $F$ in $\xx \in \XX$ is $\partial F (\xx)
 \define \{g \in \RR^d \vert F(\yy) \geq F(\xx) + \IP{g}{\yy - \xx} +
 o(\norm{\yy - \xx}), \forall \yy \in \RR^d\}$, with its elements $g \in
 \partial F(\xx)$ being called sub-gradients of $F$ at $\xx \in \XX$. We
refer to \cite{davisProximallyGuidedStochastic2019} for equivalent
definitions of the sub-differential set of $\rho$-WC functions. A
differentiable function $F: \XX \to \RR$ is $L$-smooth on $\XX \subset
 \RR^d$ if its gradient is $L$-Lipschitz continuous on the set $\XX$, i.e.,
it holds $\norm{\nabla F(\xx) - \nabla F(\yy)} \leq L \norm{\xx - \yy}$ for
all $\xx, \yy \in \XX$. For a constraint $F_2: \XX \to \RR$ and a budget $b
 \in \RR$, we use the following short notation $\cbrac{F_2 \leq b} \define
 \cbrac{\xx \in \XX \;\vert\; F_2(\xx) \leq b}$ to denote the corresponding
feasible set, usually $b = 0$. We use $\OO{\innerEmpty{}}$ notation to hide
all dependencies except for the final accuracy $\epsilon$ and
$\OOTilde{\innerEmpty{}}$ to hide additional logarithmic terms in
$1/\epsilon$.

The following standard assumptions will be used throughout the paper. More
specific assumptions related to the problem structure will be introduced in
the subsequent section.

\begin{samepage}
\begin{assumption}\label{ass:AssumptionExactPPM}
 We assume:\begin{enumerate}
  \item The functions $F_1, F_2$ are $\rho$-weakly convex.
        \label{ass:AssumptionExactPPM-ITEM1}
  \item The functions $F_1, F_2$ are continuous and satisfy for all $\xx \in \XX$
        that $\partial F_1(x) \neq \emptyset$, $\partial F_2(x) \neq \emptyset$ on
        $\XX$, and the norms of the sub-gradients are uniformly bounded by
        $\norm{g_{1}} \leq G_{F_1}$, $\norm{g_{2}} \leq G_{F_2}$ for all $\xx \in
         \XX$, $g_{1} \in \partial F_1(\xx)$ and $g_{2} \in \partial F_2(\xx)$
        respectively. We define $G \define \max\{G_{F_1}, G_{F_2}\}$.
        \label{ass:AssumptionExactPPM-ITEM2}
  \item The domain $\UU$ has bounded diameter $\DDU > 0$.
        \label{ass:AssumptionExactPPM-ITEM3}
 \end{enumerate}
\end{assumption}
\end{samepage}

%% file: tables/allRates_table1/allRates_table1.tex
\begin{table}[H]
 \renewcommand{\arraystretch}{1} % Uniformly increases row spacing
 \centering
 \resizebox{0.8\textwidth}{!}
 {
  % %%%%% start of table
  \begin{tabular}{lcc}
   \toprule
   \parbox{1.5cm}{\centering Setting}
    & Method
    & \parbox{2.8cm}{\centering Complexity}
   \\ \midrule
   \parbox{1.5cm}{\centering $F_2(\innerEmpty{}) \equiv 0$}
    & \parbox{3.7cm}{\centering Sub-Gradient Method (SM)}
    & \parbox{2.8cm}{\centering $\varepsilon^{-6}$        \\ \cite{fatkhullinStochasticOptimizationHidden2024}}
   \\
   \parbox{1.5cm}{\centering $F_2(\innerEmpty{}) \equiv 0$}
    & \parbox{2.7cm}{\centering IPPM+SM}
    & \parbox{2.8cm}{\centering $\varepsilon^{-3}$        \\ \Cref{corr:UnconstrainedHiddenConvex}}
   \\
   \midrule
   \parbox{1.5cm}{\centering $F_2(\innerEmpty{}) \not\equiv 0$}
    & \parbox{3.7cm}{\centering IPPM+SwSG
   \\
    (\Cref{algo:PPM,algo:SwitchingSubgradient})}
    & \parbox{2.8cm}{\centering $\varepsilon^{-3}$        \\ \Cref{cor:FinalRateNonSmoothIPPM}}
   \\
   \bottomrule
  \end{tabular}
  % %%%%% end of table
 }
 \caption[Summary in non-smooth setup]{Summary of total \textit{sub-gradient} and \textit{function evaluation} complexities
  for sub-gradient methods under \textbf{hidden convexity} in the \textbf{non-smooth} setup. The ``Setting'' column distinguishes between unconstrained ($F_2(\innerEmpty{}) \equiv 0$) and constrained ($F_2(\innerEmpty{}) \not\equiv 0$) problems. The third column reports the number of sub-gradient
  (and function, when $F_2(\innerEmpty{}) \not\equiv 0$) evaluations in $\OOTilde{\innerEmpty{}}$ notation to find a
  point $\xxN$ such that
  $F_1(\xxN) - \Fopt \leq \varepsilon$, $F_2(\xxN)\leq \varepsilon$. The complexity of SM in
  \cite{fatkhullinStochasticOptimizationHidden2024} is stated in terms of
  the Moreau envelope and suffers a loss in
  complexity when translated to the original objective in the non-smooth setting, see the discussion after Corollary 1 therein. ``IPPM'' stands for \textit{Inexact Proximal Point Method}, and ``SwSG``
  for \textit{Switching Sub-Gradient}.}
 \label{tab:OverviewRatesTableNonSmooth}
\end{table}

\begin{table}[H]
 \renewcommand{\arraystretch}{3} % Uniform row spacing
 \centering
 \resizebox{\textwidth}{!}
 {
  % %%%%% start of table
  \begin{tabular}{lccc}
   \toprule
   \parbox{1.3cm}{\centering Setting}
    & \parbox{3.2cm}{\centering Method}
    & \parbox{2.8cm}{\centering Complexity}
    & \parbox{3.0cm}{\centering Comment}
   \\ \midrule
   \parbox{1.3cm}{\centering $F_2(\innerEmpty{}) \equiv 0$}
    & \parbox{3.0cm}{\centering Gradient Descent or Heavy-ball momentum}
    & \parbox{2.8cm}{\centering $\varepsilon^{-1}$                                        \\ \cite{zhangVariationalPolicyGradient2020,fatkhullinStochasticOptimizationHidden2024}}
    & --
   \\
   \midrule
   \parbox{1.3cm}{\centering $F_2(\innerEmpty{}) \not\equiv 0$}
    & \parbox{3.0cm}{\centering IPPM+ACGD                                                 \\
    (\Cref{algo:PPM,algo:ACGD})}
    & \parbox{2.8cm}{\centering $\varepsilon^{-2}$ or $\theta^{-1} \varepsilon^{-1}$      \\ \Cref{cor:FinalRateSmoothIPPM_ACGD,cor:FinalRateSmoothIPPM_ACGD_Slater}}
    & $\theta$ is Slater gap
   \\
   % \parbox{1.5cm}{\centering $F_2(\cdot) \not\equiv 0$} &
   % \parbox{2.7cm}{\centering PP + ACGD} & \parbox{2.8cm}{\centering
   % $\theta^{-1} \varepsilon^{-1}$ \\
   % \Cref{cor:FinalRateSmoothIPPM_ACGD_Slater}} & $\theta$-Slater \\
   \parbox{1.3cm}{\centering $F_2(\innerEmpty{}) \not\equiv 0$}
    & \parbox{3.2cm}{\centering S-StarBL                                                  \\
    (\Cref{algo:PolyakMinorant_LB})}
    & \parbox{2.8cm}{\centering $\varepsilon^{-1}$                                        \\ \Cref{thm:MainResultHCPolyakMinorant}}
    & $\Fopt$ is known
   \\
   \parbox{1.3cm}{\centering $F_2(\innerEmpty{}) \not\equiv 0$}
    & \parbox{3.0cm}{\centering S-BL+AdaLS                                                \\
    (\Cref{algo:PolyakMinorant_LB,algo:PolyakMinorant_LB_outer_line_search})}
    & \parbox{2.8cm}{\centering $\lambda\, \varepsilon^{-1}$                              \\ \Cref{thm:MainResultHCPolyakMinorantUnknownOptimalValue}}
    & \parbox{3.5cm}{\centering $F_1(x) - \Fopt + \lambda \plus{F_2(x)} \leq \varepsilon$ \\ strong duality with $\lambda \geq \lambda^*$}
   \\
   \bottomrule
  \end{tabular}
  % %%%%% end of table
 }
 \caption[Summary in smooth setup]{Summary of total \textit{gradient} and \textit{function evaluation} complexities
  for gradient methods under \textbf{hidden convexity} in the \textbf{smooth} setup. The ``Setting'' column distinguishes between unconstrained ($F_2(\innerEmpty{}) \equiv 0$) and constrained ($F_2(\innerEmpty{}) \not\equiv 0$) problems. The third column reports the number of gradient
  (and function, when $F_2(\innerEmpty{}) \not\equiv 0$) evaluations in $\OOTilde{\innerEmpty{}}$ notation to find a
  point $\xxN$ such that
  $F_1(\xxN) - \Fopt \leq \varepsilon$, $F_2(\xxN)\leq \varepsilon$.
  ``IPPM'' stands for \textit{inexact proximal point method},
  ``ACGD'' refers to
  \textit{Accelerated Constrained Gradient Descent}, ``S-StarBL``
  refers to the
  \textit{Shifted Star Bundle-Level} method
  and ``S-BL+AdaLS`` to the \textit{Shifted Bundle-Level} method
  with the \textit{Adaptive Line-Search} for an optimal primal value $\Fopt.$}
 \label{tab:OverviewRatesTableSmooth}
\end{table}

%% file: sections/02_hiddenConvexClass.tex
\section{Hidden Convex Problem Class}\label{sec:HiddenConvexProblemClass}
The existence of a convex reformulation \eqref{eq:IntroEquationConvex} for
the problem \eqref{eq:IntroEquation} motivates its representation as a
compositional optimization problem in the form:
\begin{align}
 \min_{\xx \in \XX} \;F_1(\xx) \define H_1(c(\xx)), \quad
 \text{s.t. } F_2(\xx) \define H_2(c(\xx)) \leq 0,
 \tag{HC}
 \label{eq:MainProblem}
\end{align}
under a \emph{consistent transformation function}, $c(\innerEmpty{})$, for both the objective
and the constraint. Now we introduce the central definition of this work, which postulates
the existence of an (unknown)
transformation $c(\innerEmpty{})$ and its properties.

\begin{definition}[Hidden Convexity\footnote{\Cref{def:HC} can be extended to hidden \emph{strong} convexity,
   cf. \Cref{def:HC-strong} in the Appendix,
   for which we will show additional results in~\Cref{appendix:InexactPPMunderHiddenStrongConvexity}, but for the sake of
   clarity, we focus on hidden convexity throughout this work.}]
 \label{def:HC}
 The above problem \eqref{eq:MainProblem} is called \emph{hidden convex}
 with modulus $\mu_c > 0$, if its components satisfy the following underlying conditions.
 \begin{enumerate}
  \item The domain $\UU = c(\XX)$ is convex, the functions $H_1, H_2 : \UU \to \RR$
        are convex, i.e. satisfy for $i = 1, 2$ and for all $u, v \in \UU$ and any
        $\lambda \in [0,1]$
        \begin{align}
         H_{i}((1-\lambda)\uu + \lambda \vv) \leq (1-\lambda)H_i(\uu)
         + \lambda H_i(\vv).
         % - \frac{(1-\lambda)\lambda \mu_H}{2} \norm{\uu - \vv}^2.
         \tag{HC-1}
         \label{eq:HCHCDefinitionConvexityInequality}
        \end{align}
        Additionally, we assume \eqref{eq:MainProblem}
        admits a solution $\uu^* \in \UU$ with its corresponding objective function
        value $\Fopt \eqdef H_1(\uu^*) = F_1(c^{-1}(\uu^*))$.
  \item The map $c : \XX \to \UU$ is invertible and there exists a $\mu_c > 0$ such
        that for all $\xx, \yy \in \XX$ it holds
        \begin{align}
         \norm{c(\xx) - c(\yy)} \geq \mu_c \norm{\xx - \yy} \tag{HC-2}.
         \label{eq:HCDefinitionNormCfunctionInequality}
        \end{align}
 \end{enumerate}
\end{definition}

Note that the condition \eqref{eq:HCDefinitionNormCfunctionInequality}
along with Assumption~\ref{ass:AssumptionExactPPM}
(\Cref{ass:AssumptionExactPPM-ITEM3}) imply that the domain $\XX$ has
bounded diameter $\DDX \leq \frac{1}{\mu_c} \DDU $. We refer to
\cite{fatkhullinStochasticOptimizationHidden2024} for necessary and
sufficient conditions for \eqref{eq:HCDefinitionNormCfunctionInequality}.
In particular, for continuously differentiable transformations
\eqref{eq:HCDefinitionNormCfunctionInequality} is equivalent to a uniformly
bounded operator norm of the Jacobian of $c^{-1}(\innerEmpty{})$, i.e.,
$\norm{J_{c^{-1}}(u)}_{\text{op}} \leq 1/\mu_c$ for all $u \in \UU.$ We
must highlight the importance of the consistent transformation function
$c(\innerEmpty{})$ for both the objective $F_1$ and the constraint $F_2$.
If the transformations are inconsistent, the problem becomes intractable
due to multiple isolated feasible points (cf.
\Cref{subsec:NP_hardness_under_inconsistent_transformation}).

\paragraph{Relaxing invertibility condition.} Assuming the existence of an invertible map $c(\innerEmpty{})$ may sound
limiting. Here we comment on possible relaxations of the invertibility of
the transformation $c(\innerEmpty{}).$

\begin{itemize}
 \item \textbf{Local Invertibility.} The invertibility of $c(\innerEmpty{})$ can be
       relaxed to a local neighborhood $\VV_{\xx}$ around each $\xx \in \XX$,
       where $c(\innerEmpty{})$ is a bijection between $\VV_{\xx}$ and
       $\WW_{c(\xx)} \define c(\VV_{\xx})$. This allows us to relax the Item 2 of \Cref{def:HC} as
       invertibility is only required locally. By carefully examining the proofs of
       \Cref{thm:MainResultHCPPM,thm:MainResultHCPolyakMinorant,thm:MainResultHCPolyakMinorantUnknownOptimalValue},
       we find that we only require an
       $\bar{\alpha}$ such that for all $\alpha \in [0, \bar{\alpha}]$, the
       combination $(1-\alpha)c(\xx) + \alpha c(\xxOpt) \in \WW_{c(\xx)}$
       for all $\xx \in \XX$. An example of this relaxation
       appears in the context of \eqref{eq:IntroExampleCvx-CMDP} problem, cf.
       \cite[Ass. 5.11]{zhangConvergenceSampleEfficiency2021}, \cite[Ass. 2]{yingPolicybasedPrimalDualMethods2024}.
 \item \textbf{Generalized Inverse of Non-linear Transformations.}
       Another possible relaxation is to use the \emph{generalized inverse of a non-linear
        mapping} introduced by Gofer and Gilboa~\cite{goferGeneralizedInversionNonlinear2024},
       which generalizes the well-known notion of pseudoinverse of a linear map due to Penrose~\cite{penroseGeneralizedInverseMatrices1955}.
       Following their Definition 3 in \cite{goferGeneralizedInversionNonlinear2024}, the
       non-linear generalized inverse $c^\dagger: \UU \to \XX$ of the map $c(\innerEmpty{})$ is defined as:
       \begin{align}
        c^\dagger & (\uu) \define \arg\min\{\norm{\xx} \; \vert \;
        \xx \in \arg\min_{\yy \in \XX} \norm{c(\yy) - \uu}\},
        \tag{NLGI-1} \label{eq:PseudoInverseProperty1}                  \\
        c^\dagger & (c(c^\dagger(\uu))) = c^\dagger(\uu), \quad \forall
        \uu \in \UU.
        \tag{NLGI-2} \label{eq:PseudoInverseProperty2}
       \end{align}
       Our Item 2 in \Cref{def:HC} can be relaxed by
       requiring the existence of $c^{\dagger}(\innerEmpty{})$ and replacing
       the condition \eqref{eq:HCDefinitionNormCfunctionInequality} with
       \begin{align}
        \norm{c^\dagger(\uu) - c^\dagger(\vv)} \leq
        \frac{1}{\mu_c}\norm{\uu - \vv} \tag{Gen-HC-2},
        \label{eq:GenHC2-pseudoInverse}
       \end{align}
       for all $\uu, \vv \in \UU$. By examining the proof of \Cref{lemma:HCContractionInequality}, it becomes
       clear that the above relaxation is sufficient for our analysis.
       We provide an educational example for a transformation satisfying
       \eqref{eq:GenHC2-pseudoInverse}. With $d = \dim(\XX)$ and $m \define \dim(\UU)$,
       we consider a \textit{piece-wise linear} map $c(\innerEmpty{})$ of the form
       \begin{align*}
        c(\xx) = \sum_{i=1}^{k} \indicator_{R_i}(\xx) \cdot A_i \xx,
       \end{align*}
       for $\xx \in \RR^d$, with
       an arbitrary partition
       $\cbrac{R_i}_{i=1}^k$ of $\RR^d$, indicator functions
       $\indicator_{R_i}(\xx) = 1$ if $\xx \in R_i$ and
       $\indicator_{R_i}(\xx) = 0$ otherwise, and
       and arbitrary matrices $\{A_i\}_{i=1}^k \subset \RR^{m\times d}$ satisfying the boundary conditions,
       \begin{align*}
        A_i \xx = A_j \xx, \quad \forall
        x \in \overline{R_i} \cap \overline{R_j}
       \end{align*}
       for all $i, j=1, \ldots, k$ with $i\neq j,$ to ensure the continuity of $c(\innerEmpty{}).$ We distinguish two cases
       depending on the relative dimensions of $\XX$ and $\UU$:
       \begin{itemize}
        \item \textbf{Case} $d<m$ (underdetermined case): Each $A_i$ has full column rank,
              and its Moore-Penrose pseudoinverse is given by
              \begin{align*}
               A_i^\dagger \define (A_i^\top A_i)^{-1} A_i^\top.
              \end{align*}
        \item \textbf{Case} $d>m$ (overdetermined case):
              Each $A_i$ has full row rank,
              and its Moore-Penrose pseudoinverse is given by
              \begin{align*}
               A_i^\dagger \define A_i^\top (A_i A_i^\top)^{-1}.
              \end{align*}
       \end{itemize}
       Then $c(\innerEmpty{})$ satisfies
       \eqref{eq:PseudoInverseProperty1} and \eqref{eq:PseudoInverseProperty2}
       with
       \begin{align*}
        c^\dagger(\uu) \define \sum_{i=1}^{k} \indicator_{c(R_i)}(\uu) \cdot A_i^\dagger \uu,
       \end{align*}
       for $\uu \in \RR^m$. Such $c(\innerEmpty{})$ satisfies
       \eqref{eq:GenHC2-pseudoInverse} with
       \begin{align*}
        \mu_c \define \paren*{\min_{i=1,\ldots, k} \sigma_{\min}(A_i)}^{-1},
       \end{align*}
       i.e. the reciprocal
       of the minimal non-zero singular value. The special instance of the case
       where $\dim(\UU) \gg \dim(\XX)$ is particularly interesting, e.g. in the context of
       controller synthesis in optimal control
       \cite{boydLinearMatrixInequalities1994, andersonSystemLevelSynthesis2019,
        sunLearningOptimalControllers2021}, where the original problem, despite
       being non-convex, is still favorable, due to its lower dimensionality.
\end{itemize}

\subsection{Globally Optimal Solutions and KKT Points}
We start our analysis by proving elementary properties of constrained
optimization under the structural assumption of hidden convexity.
\begin{definition}\label{def:KKTcondition}
 We say $\xxHat \in \XX$ is a \emph{Karush-Kuhn-Tucker (KKT) point} of
 \eqref{eq:IntroEquation} if $F_2(\xxHat) \leq 0$
 and there exists a \emph{Lagrangian Multiplier} $\lambdaHat \geq 0$ such that
 $$
  0 \in \partial F_1(\xxHat) + \lambdaHat \, \partial F_2(\xxHat)
  + \partial_x \delta_{\XX}(\xxHat) , \quad \text{and } \quad \lambdaHat \, F_2(\xxHat) = 0 .
 $$
\end{definition}

\begin{definition}[Slater's Condition]\label{def:SlaterCondition}
 We say that a $\theta$--Slater's (or simply Slater's) condition
 holds if there exists a point $\xx \in \text{relint}(\XX)$ with
 $F_2(\xx) \leq - \theta$ for some $\theta > 0$. Such $\xx$ is called a $\theta$--Slater point.
\end{definition}
Similar to Proposition 1 in the unconstrained case
\cite{fatkhullinStochasticOptimizationHidden2024}, we relate the KKT
points of \eqref{eq:IntroEquation} with the global minima.

\begin{proposition}
 \label{prop:RelationshipKKT}
 Let $F_i(\innerEmpty{})$ be weakly convex on $\XX$ for $i = 1, 2$ and problem
 \eqref{eq:MainProblem} be hidden convex. Assume $c(\innerEmpty{})$ is differentiable for
 some $\xxHat \in \XX$.
 Then the following implications hold:
 \begin{enumerate}
  \item If $\xxHat$ is a KKT point of \eqref{eq:IntroEquation}, then $\xxHat$ is a
        global minimum.
  \item If $\xxHat$ is a global minimum of \eqref{eq:IntroEquation} and Slater's
        condition holds, i.e. also $\xxHat \in \mathrm{relint}(\XX)$, then $\xxHat$
        is a KKT point.
 \end{enumerate}
\end{proposition}
\begin{proof}
 First, we show that $\xxHat\in \XX$ being a KKT point of \eqref{eq:IntroEquation}
 is equivalent to $\uuHat = c(\xxHat) \in \UU$ being a KKT point of \eqref{eq:IntroEquationConvex}.
 Indeed, conditions $H_2(\uuHat) \leq 0$ and $\lambdaHat \, H_2(\uuHat) = 0$ follow immediately from the
 reformulation, and by the chain rule we have
 \begin{equation}\label{eq:chain_rule}
  0 \in \partial F_1(\xxHat) + \lambdaHat \, \partial F_2(\xxHat) + \partial_x \delta_{\XX}(\xxHat)
  = J_c^{\top}(\xxHat) (\partial H_1(\uuHat)
  + \lambdaHat \, \partial H_2(\uuHat) + \partial_\uu \delta_{\UU}(\uuHat) ).
 \end{equation}
 As the map $c(\innerEmpty{})$ is invertible with a Lipschitz continuous inverse
 by \eqref{eq:HCDefinitionNormCfunctionInequality}, then its
 Jacobian $J_c(\xxHat)$ is invertible at $\xxHat$
 (see e.g., Corollary 3.3.\,in
 \cite{lawsonInverseFunctionTheorem2020}). Therefore, \eqref{eq:chain_rule} is equivalent to
 $0 \in \partial H_1(\uuHat) + \lambdaHat \, \partial H_2(\uuHat) + \partial_u \delta_{\UU}(\uuHat)$
 and the promised equivalence holds. Since problem \eqref{eq:IntroEquationConvex} is
 convex and $\uuHat$ is a KKT point, by the sufficient optimality condition, $\uuHat$ is a
 globally optimal solution, i.e., $H_1(\uuHat) \leq H_1(\uu)$ for any
 $\uu\in \left\{\vv \in \UU \mid H_2(\vv) \leq 0 \right\}$. As a result, we have
 $$F_1(\xxHat) \leq F_1(\xx) \quad \text{for any} \quad \xx
  \in \left\{\yy \in \XX \mid F_2(\yy) \leq 0 \right\}. $$

 To prove the second claim, we notice that if $\xxHat \in \XX$ is a global
 minimum of \eqref{eq:IntroEquation} implies that $\uuHat = c(\xxHat)$ is a
 global minimum of the convex problem \eqref{eq:IntroEquationConvex}.
 Moreover, Slater's condition also holds for the convex reformulated
 problem. Then by \cite[Thm. 3.1.26]{nesterovLecturesConvexOptimization2018}
 $\uuHat$ is a KKT point of \eqref{eq:IntroEquationConvex}. It remains to
 use the equivalence of KKT points of problems \eqref{eq:IntroEquation} and
 \eqref{eq:IntroEquationConvex} to conclude the proof.
\end{proof}

\subsection{Non-convex Equality Constraints}\label{subsec:nonconvex_equality_constraints}
Throughout the paper, our main focus is hidden convex inequality
constraints. However, many applications involve non-convex equality type
constraints, which are known to be more challenging to handle than
inequality constraints. We show that our results can be extended to hidden
linear equality type constraints. In particular, for a given matrix $A \in
 \RR^{n\times d}$, consider the problem
\begin{align}
 \min_{\xx \in \XX} \;F_1(\xx) \define H_1(c(\xx)), \quad
 \text{s.t. } F_{2, \text{eq}}(\xx) \define A \,c(\xx) = 0 .
 \tag{HC-Eq}
 \label{eq:hidden_convex_equality_constraints}
\end{align}
The equality constraint of this type can be written in the form of inequality constraint by increasing the number of constraints, i.e.,
$$
 F_{2, \text{eq}}(\xx) = 0 \quad \Leftrightarrow \quad  F_{2, \text{eq}}(\xx) \leq 0, \quad \text{and} \quad  -F_{2, \text{eq}}(\xx) \leq 0 .
$$
Here, $F_{2, \text{eq}}$ is vector valued when $n>1$, and the equality and inequality constraints should be understood
component-wise, i.e., $F_{2, \text{eq}}^i(\xx) = 0$ for all $i \in [n]$. Moreover,
one can combine these constraints into a single one and directly reduce the problem to our original form
\eqref{eq:IntroEquation} with a single constraint:
$$
 F_{2, \text{eq}}(\xx) = 0 \quad \Leftrightarrow \quad F_2(\xx) \define \max_{1\leq i \leq n}
 \max\{ F_{2, \text{eq}}^i(\xx), - F_{2, \text{eq}}^i(\xx) \} \leq 0 .
$$
It is important to note that in this formulation, Slater's condition typically fails and the $\max$ operators introduce non-smoothness
even when the map $c(\innerEmpty{})$ is sufficiently smooth.

\subsection{NP-Hardness under Inconsistent Transformation}\label{subsec:NP_hardness_under_inconsistent_transformation}

A natural question is whether the assumption about the common
transformation $c(\innerEmpty{})$ in \eqref{eq:MainProblem} is essential.
Suppose the objective and constraint are each hidden convex but under
\emph{different} transformations, $c_1(\innerEmpty{})$ and
$c_2(\innerEmpty{})$. Does such a relaxation still admit efficient
algorithms?

We call this the setting of \emph{inconsistent transformations}, where
\begin{align}
 F_1(x) = H_1(c_1(x)), \qquad
 F_2(x) = H_2(c_2(x)),
 \tag{HC-ICT}\label{eq:IntroEquationInconsistent_Transform}
\end{align}
for convex functions $H_1, H_2$ and invertible maps $c_1, c_2 : \XX \to \UU$. We establish the following hardness result for such problems.

\begin{proposition}\label{prop:inconsistent-hardness}
 Hidden convex optimization under inconsistent transformations
 \eqref{eq:IntroEquationInconsistent_Transform}
 is NP-hard: there is no polynomial-time global solution method, unless $\textsf{P} = \textsf{NP}$.
\end{proposition}

\begin{proof}
 We reduce from $0$-$1$ Integer Linear Programming \cite{karp2009reducibility}.
 Given an integer matrix $A \in \ZZ^{n \times d}$ and vectors $b, c \in \ZZ^n$,
 consider the problem
 \[
  \min_{x \in \XX} \; F_1(x)
  \qquad \text{s.t. } F_2(x) \leq 0 ,
 \]
 with
 \[
  F_1(x) \define \langle c, y \rangle + \norm{z - \mathds{1}},
  \qquad
  \XX \define \{ (y, z) \mid A y \leq b \},
  \qquad
  x \define (y, z) \in \RR^d \times \RR^d ,
 \]
 where $\mathds{1} = (1,\ldots,1)^\top$, and
 \[
  F_2(x) \define \max_{1 \leq i \leq d} \max \{ F_2^i(x), - F_2^i(x) \},
  \qquad
  F_2^i(x) \define z_i - \cos(2\pi y_i).
 \]

 \emph{Step 1: Hidden convex structure.}
 Define the maps
 \[
  c_1(x) = x,
  \qquad
  c_2(x) = \big( (y_1, z_1 - \cos(2\pi y_1)), \ldots, (y_d, z_d - \cos(2\pi y_d)) \big) .
 \]
 Then $F_1(x)$ is trivially hidden convex as it is convex. Observe that the
 constraint can be written as $F_2(x) = H_2(c_2(x))$ with $H_2(u) \define
  \|w\|_\infty,$ $u \define (v,w),$ where $v = y.$ The map
 $c_2(\innerEmpty{})$ is invertible with a well-conditioned Jacobian of
 $c_2^{-1}$ (operator norm bounded by $1+2\pi$), so $F_2$ is hidden convex
 under the transformation $c_2(\innerEmpty{})$.

 \emph{Step 2: Constraint interpretation and hardness.}
 The constraint $F_2(x) \leq 0$, together with $z = \mathds{1}$,
 forces each $y_i$ to be an integer (since $\cos(2\pi y_i) = 1$ iff $y_i \in \ZZ$).
 Thus the feasible region encodes the Boolean cube, and the problem simulates
 a $0$-$1$ Integer Linear Program of the form $\min_{y\in \ZZ} \langle c, y \rangle,$ s.t. $A \, y \leq b.$

\end{proof}

\Cref{prop:inconsistent-hardness} shows that inconsistent transformations make hidden convex optimization intractable. Moreover,
this result can be extended to show that such
problems cannot even be well-approximated in polynomial time, unless $\textsf{P} = \textsf{NP}$.
We refer to \cite{trevisan2014inapproximability,leeClassifyingApproximationAlgorithms2021a}
for the complexity of approximation algorithms
and related canonical hard problems such as Independent Set, MAX-3SAT, and MAX-CUT.

Proposition~\ref{prop:inconsistent-hardness} has an additional consequence:
it highlights the strength of hidden convexity compared to related
structural assumptions such as gradient-domination. Indeed, using
Proposition~2 from \cite{fatkhullinStochasticOptimizationHidden2024}, one
can verify that in the inconsistent-transformation setting
\eqref{eq:IntroEquationInconsistent_Transform}, both $F_1$ and $F_2$
satisfy a gradient-domination condition of the form
\[
 \inf_{s_{\xx} \in \partial(F_i+\delta_{\XX})} \|s_{\xx}\|
 \;\;\geq\;\; \frac{\mu_{c_i}}{\DDU_i} \big(F_i(\xx) - F_i^*\big),
 \qquad \forall \, \xx \in \XX, \; i \in \{1,2\},
\]
where $F_i^* \define \min_{\xx \in \XX} F_i(\xx),$ $\DDU_i \define
 \sup_{\xx, \yy\in\XX} \norm{c_i(\xx) - c_i(\yy)}$. Nevertheless,
Proposition~\ref{prop:inconsistent-hardness} shows that such condition is
insufficient to guarantee tractability: gradient domination (at least in
this form) does not guarantee the tractability of global solutions.
Figure~\ref{fig:ConsistencyExample} illustrates this phenomenon on a
simpler instance. There, the transformation, $c(\xx) \define (\xx_1 - 1,
 2\abs{\xx_1}- \xx_2 -1)^\top$, convexifies the feasible set, $\cbrac{F_2
  \leq 0} \define \cbrac{\xx \in \XX \;\vert \; \norm{c(\xx) + (0.5,
   0.6)^\top }_1 - 0.8 \leq 0}$, but simultaneously destroys the convexity of
the objective $F_1(\xx) \define \norm{\xx - (0,2)^\top}_2$, turning it into
a nonconvex function. This contrast underscores the necessity of a
\emph{consistent transformation} applied jointly to the objective and
constraints in hidden convex optimization.

\begin{figure}
 \centering
 \begin{subfigure}{0.49\textwidth}
  \centering
  \includegraphics[height=.83\linewidth, keepaspectratio]{
   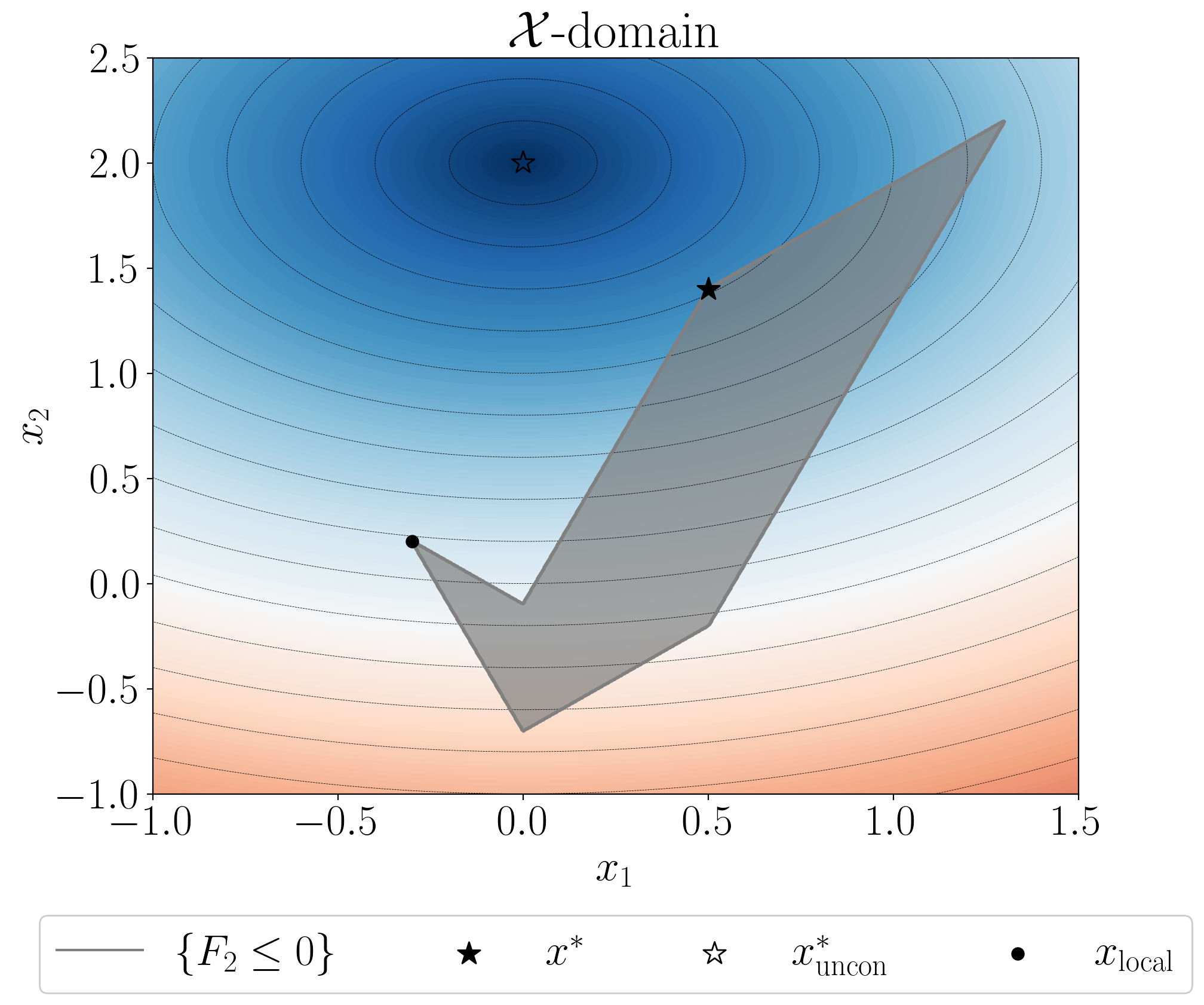
  }
  \caption{\textit{Non-convex Constraint.}
   {Level sets of $F_1$} and the \gray{feasible set
    $\{F_2 \leq 0\}$} in the $\XX$-domain.\\}
  \label{fig:ConsistencyExampleNonSmoothX}
 \end{subfigure}
 \hfill
 \begin{subfigure}{0.49\textwidth}
  \centering
  \includegraphics[height=.83\linewidth,keepaspectratio]{
   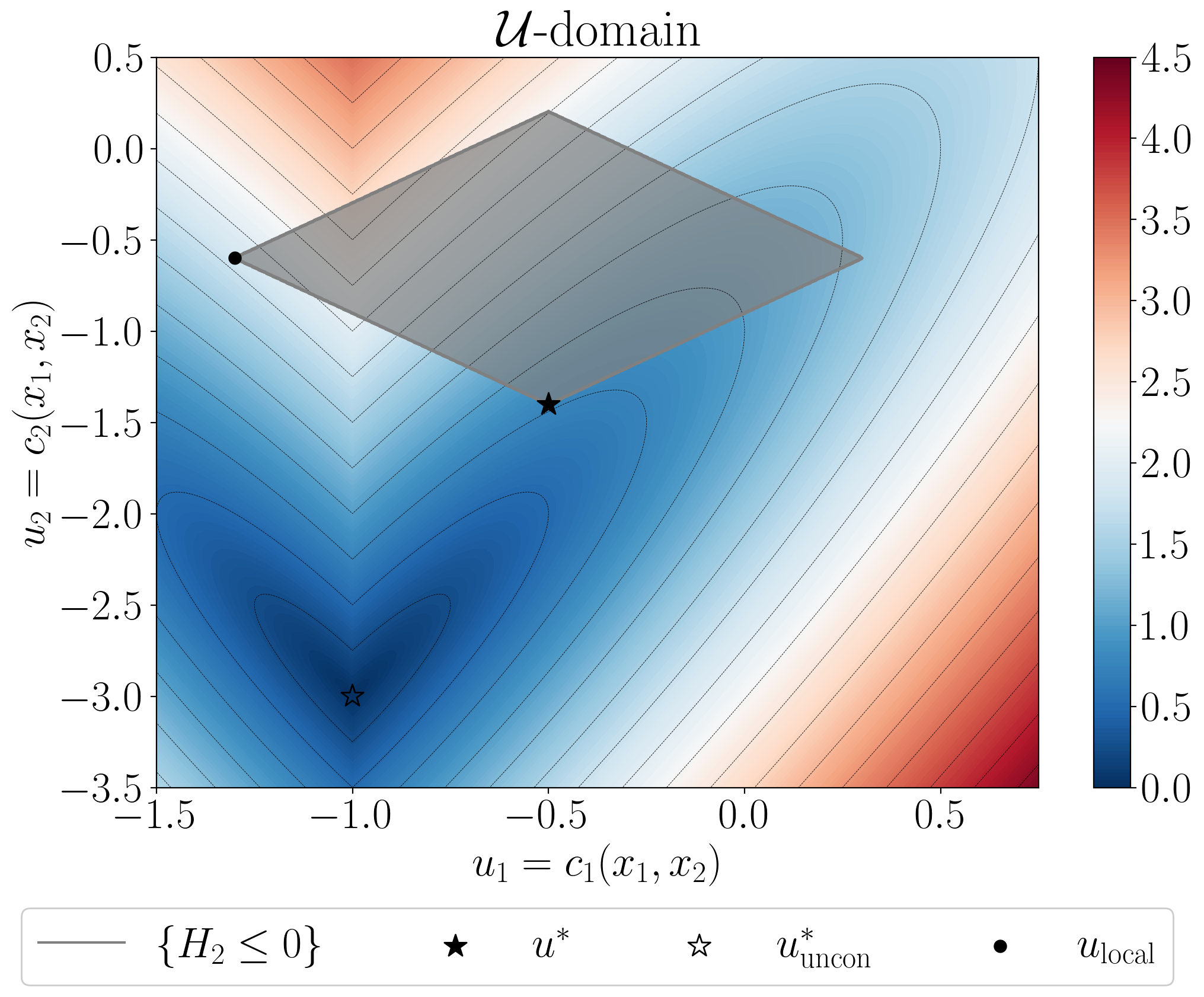
  }
  \caption{\textit{Convex Constraint after Reformulation.}
   Level sets of transformed objective and the \gray{feasible set
    $\{H_2 \leq 0\}$} in the $\UU$-domain.}
  \label{fig:ConsistencyExampleNonSmoothU}
 \end{subfigure}
 \caption[Transformation Function Consistency Condition]{
  An illustrative example of a hidden convex problem with \textit{inconsistent transformations}
  \eqref{eq:IntroEquationInconsistent_Transform}.The grey regions illustrate the feasible set, and the objective value is shown in color.
  This problem has two local minima: the sub-optimal leftmost point $\xx_{\mathrm{local}} = (-0.3, 0.2)^\top$,
  and the optimal rightmost point $\xxOpt = (0.5, 1.5);$ $\xxOptGlobal$ denotes the global optimum without constraints.
  A local search method may terminate at the sub-optimal local  minima $\xx_{\mathrm{local}}.$
 } \label{fig:ConsistencyExample}
\end{figure}

\subsection{Motivating Applications}\label{subsec:motivating_applications}
Now we describe three motivating practical examples for hidden convex
problems under functional constraints with \textit{consistent
 transformations}. For additional examples, we refer, for example, to
\cite{xiaSurveyHiddenConvex2020,
 fatkhullinStochasticOptimizationHidden2024}.

\paragraph{Geometric Programming
 \cite{eckerGeometricProgrammingMethods1980,
  boydTutorialGeometricProgramming2007, xiaSurveyHiddenConvex2020}.}
In the context of power control and communication systems,
as well as optimal doping profile,
the problems often involve so-called \emph{posinomial functions}
$F_j: \RR^d_{>0} \to \RR$ with $j=1,\ldots, p$ of the form
\begin{align*}
 F_j(\xx) \define \sum_{k_j=1}^{K_j} b_{k_j} \xx_1^{a_{1,k_j}} \cdots \xx_d^{a_{d,k_j}}
\end{align*}
with coefficients $b_{k_j} > 0$, $K_j \in \NN_+$ and $a_{i,k_j} \in \RR$ for all
$j=1,\ldots, p$, $k_j=1, \ldots, K_j$, $i=1,\ldots, d$. {C}onstrained
{G}eometric {P}rogramming problem in the standard form
\cite[Eq. 3]{boydTutorialGeometricProgramming2007} is given by
\begin{align}
 \begin{aligned}
  \min_{\xx \in \RR_{>0}^d} \;F_1(\xx),  \qquad
  \mathrm{s.t.} \; F_j(\xx) \leq 1, \; j=2,\ldots, p ,
 \end{aligned}
 \tag{CGP}\label{eq:ExampleIntroGeometricProgramming}
\end{align}
which is non-convex in $\xx \in \XX \define \RR_{>0}^d$ but admits a convex reformulation
via the variable change
$\uu \define c(\xx) \define \paren{\log(\xx_1), \ldots, \log(\xx_d)}^\top$.
The resulting convex reformulation is
\begin{align*}
 H_j(\uu) \define F_j(c^{-1}(\uu))
 = \sum_{k_j=1}^{K_j} b_{k_j} \exp(a_{1,k_j}\uu_1) \cdots \exp(a_{d,k_j}\uu_d) ,
\end{align*}
where $H_j(\innerEmpty{})$, $j=1,\ldots, p$ are convex. One can easily see
that \eqref{eq:ExampleIntroGeometricProgramming} is hidden convex
and satisfies \eqref{eq:HCDefinitionNormCfunctionInequality} with
$\mu_c \define (\max_{\xx\in\XX} \norm{\xx}_{\infty})^{-1}$, since
$\log \xx_i$ is $\paren{\max_{\xx \in \XX}\xx_i}^{-1}$-strongly monotone.

\paragraph{Convex Constrained Markov Decision Process (CCMPD)
 \cite{zhangVariationalPolicyGradient2020}.}
\label{paragraph:IntroMotivatingExampleConvexCMDP}
Convex reinforcement learning (RL) under convex constraints generalizes
the classical (constrained) RL setting. Based on a discounted constrained
Markov Decision Process
(CMDP) of the form $\mathcal{M}(\SS, \AA, \PP, \mu_0, \gamma, r, c)$, where $\SS$ and
$\AA$ denote the (finite) state and action spaces respectively,
$\PP : \SS \times \AA \to \Delta(\SS)$ represents the state-action transition
probability kernel, $\Delta(\SS)$
is the probability simplex
over $\SS$, $\mu_0$ is the initial state distribution and $\gamma \in (0,1)$
is the discount factor. Based on the reward $r:\SS \times \AA \to \RR$
and the penalty cost $c: \SS \times \AA \to \RR$, the classical RL problem is formulated in
finding an optimal stationary policy $\pi: \Delta(\AA)^{\card{\SS}} \to \Delta(\AA)$ by maximizing
the reward function while satisfying the constraint on the cost function.
With the notation of $X \sim \rho$ for a random variable $X$ following a
probability distribution $\rho$, as well as $\Exp{\innerEmpty{}}$ and
$\Prob{\innerEmpty{}}$ for the expectation and probability, respectively,
we formally define
\begin{align}
 \begin{aligned}
  \min_{\pi \in \Pi}\; F_1(\pi) &
  \define -
  \Expu{\ss_0 \sim \mu_0, \pi}{\sum_{h=0}^{\infty} \gamma^h
  r(\ss_h, \aa_h)} ,                      \\
  \text{s.t.}\; F_2(\pi)        & \define
  \Expu{\ss_0 \sim \mu_0, \pi}{\sum_{h=0}^{\infty} \gamma^h
   c(\ss_h, \aa_h)} \leq 0 ,
 \end{aligned}
 \tag{CMDP}\label{eq:IntroExampleClassicalCMDP}
\end{align}
where $\Pi \define \Delta(\AA)^{|\SS|}$ is the set of all stationary policies.
This set is the product of simplices, which admits an efficient projection. Note that, for consistency with our
formulation \eqref{eq:IntroEquation}, we minimize the negative expected reward,
which is equivalent to maximizing the expected reward in standard RL formulations.

Given a policy $\pi$, at each time $h\in \NN$, the agent is in a state
$\ss_h$ and chooses an action $\aa_h \sim \pi(\innerEmpty{}|\ss_h)$,
resulting in a transition $\ss_{h+1} \sim \PP(\innerEmpty{} \vert \ss_h,
 \aa_h)$. With $\PP_{\mu_0, \pi}$ we denote the induced probability
distribution of the Markov chain $(\ss_h, \aa_h)_{h \in \NN}$ with an
initial state distribution $\mu_0$. Under a transformation via the
\emph{state-action occupancy measure}
\cite{suttonReinforcementLearningIntroduction2018}, defined by
\begin{align}
 \lambda^\pi: \SS \times \AA \to (0,1), \;
 \lambda^\pi(\ss, \aa)
 \define \sum_{h=0}^{\infty} \gamma^h \PP_{\mu_0, \pi}\paren{\ss_h = \ss, \aa_h = \aa},
\end{align}
the classical constrained RL problem becomes linear
in the objective and the constraint, i.e. \eqref{eq:IntroExampleClassicalCMDP}
is equivalent to
$\max_{\lambda \in \UU} \IP{r}{\lambda^\pi}$
s.t. $\IP{c}{\lambda^\pi} \leq 0$,
where $\UU \define \cbrac{\lambda^\pi \; \vert \; \pi \in \Pi}$.
Convex constrained RL generalizes this optimization problem to
\begin{align}
 \begin{aligned}
  \min_{\pi \in \Pi} \; & F_1(\pi) \define H_1(\lambda^\pi) ,        \\
  \mathrm{s.t.} \;      & F_2(\pi) \define H_2(\lambda^\pi) \leq 0 ,
 \end{aligned}
 \tag{CCMDP}
 \label{eq:IntroExampleCvx-CMDP}
\end{align}
where the utility functions $H_1, H_2: \UU \to \RR$ are convex functions in $\lambda^\pi$,
but the resulting optimization problem over the policy space $\Pi$ is non-convex. Beyond the standard
\eqref{eq:IntroExampleClassicalCMDP}, popular examples of \eqref{eq:IntroExampleCvx-CMDP} encompass safe exploration,
i.e. $H_1$ is the negative entropy,
under safety constraints, e.g. staying close to an experts trajectory
$H_2(\lambda^\pi) \define \norm{\lambda^\pi -
  \lambda^{\pi_{\mathrm{exp}}}}$. Other instances include safe apprenticeship
learning and safe learning, see e.g.,
\cite{geistConcaveUtilityReinforcement2021,zahavy2021reward, mutti2022challenging} for details.

The problem \eqref{eq:IntroExampleCvx-CMDP} is hidden convex by
construction with $\XX = \Pi$ and $c(\xx) \define \lambda^\pi$, (with $\xx
 = \pi$). As shown in \cite[Proposition
 H.1]{zhangVariationalPolicyGradient2020}, the constant $\mu_c$ can be
estimated under mild assumptions on the initial distribution $\mu_0$. Note
that in convex RL, we can control $\lambda^\pi$ only implicitly by changing
the policy $\pi$, i.e. via the policy gradient theorem, and thus, the exact
computation of the transformation map and its inverse would require the
knowledge of the state-action transition probability kernel and can be
either computationally expensive or even intractable.

\paragraph{Controller Synthesis in Optimal Control.}
Given a continuous-time linear time-invariant (LTI) system
\begin{equation}\label{eq:LTI}
 \dot x(t) = A \, x(t) + B\, u(t) , \qquad x(0) = x_0, \tag{LTI}
\end{equation}
where $x(t) \in \RR^n$ is a state vector, $u(t) \in \RR^m$ a control input vector, and $x_0 \in \RR^n$ is an initial
state generated as a random variable with covariance matrix $W = \Exp{x_0 \,x_0^{\top}} \in \SS_{+}^n$,
where $\SS_{+}^n$ denotes the set of symmetric
positive semidefinite $n\times n$ matrices.
With static state feedback $u(t) = K \, x(t),$ $K \in \RR^{m\times n}$, the classical linear quadratic regulator
(LQR) problem reads \cite{kalman1960contributions,fatkhullinOptimizingStaticLinear2021}:
\begin{align}\label{eq:LQR_original_problem}
  & \min_{K\in \RR^{m\times n}}  J_{\text{LQR}}(K) \define  \Exp{\int_{0}^{\infty} \rb{ x^{\top}(t) Q x(t) + u^{\top}(t) R u(t) } \dd t} \tag{LQR} \\
  & \quad  \text{s.t.} \quad \eqref{eq:LTI}, \quad u(t) = K \, x(t) , \quad A + BK \text{ is Hurwitz stable,} \notag
\end{align}
where $Q \succeq 0$, $R \succ 0$, and Hurwitz stable means that all real parts of eigenvalues of $A + BK$ are negative.
It is known that \eqref{eq:LQR_original_problem} admits an optimal solution $K^*$
(that is unique when $W\succ 0$), which is also optimal for any square integrable control laws $u(t)$ under mild assumptions.

A direct gradient-based optimization of $J_{\text{LQR}}(K)$ over $K$ is
possible, but it requires initialization with a stabilizing $K$ such that
$A+BK$ is Hurwitz. Finding such an initialization is often nontrivial,
particularly in large-scale or poorly conditioned systems. To circumvent
this difficulty, one can instead reformulate the problem with equality-type
constraints that implicitly enforce stability and allow the user an
arbitrary initialization of $K$.

Specifically, consider the nonlinear equality-constrained problem
\begin{align}\label{eq:LQR_nonlinear_constraint}
  & \min_{\xx \define (P, K) }  F_1(\xx) \define \operatorname{Tr}\rb{(Q + K^{\top} R K) P } \tag{NL-Eq-LQR} \\
  & \quad  \xx \in \XX \define \cb{ (P, K) \mid P \succ 0, P = P^{\top}} , \notag                            \\
  & \quad \text{s.t.}  \quad F_{2, \text{eq}}(x) \define (A+BK) P + P (A+BK)^{\top} + W = 0 . \notag
\end{align}
Above formulation is equivalent to \eqref{eq:LQR_original_problem}, but \eqref{eq:LQR_nonlinear_constraint}
implicitly enforces stability via the Lyapunov equation $F_{2, \text{eq}}(\xx) = 0$ and $P  \succ 0$, avoiding the need for stabilizing initialization.
Although the strong duality holds under mild assumptions \cite[Thm. 1]{watanabe2025revisiting},
the presence of the equality constraint implies that Slater's condition does not hold for this problem.

Nevertheless, a convex reformulation of \eqref{eq:LQR_nonlinear_constraint}
can be obtained by the variable change $$ \uu \define (P, Y) \define
 c((P,K)) \define \rb{P, (K - K^*) P } , $$ where $K^*$ is the optimal
controller.

The inverse of $c(\innerEmpty{})$ is given by $$ c^{-1}((P, Y)) = (P, K^* +
 YP^{-1}) $$ which admits a Jacobian of the form $$ D(c^{-1})_{(P,
   Y)}\brac{\Delta P, \Delta Y} = \rb{\Delta P, \Delta Y P^{-1} - Y P^{-1}
  \Delta P P^{-1}}^\top $$ defined in perturbation notation in order to avoid
tensors. For the paired spectral norm of the Jacobian we have
\begin{align*}
  & \norm{ D(c^{-1})_{(P,Y)}\brac{\Delta P,\Delta Y}}                                   \\
  & \quad\quad\le \sqrt{ \norm{\Delta P}^2
  + \Big( \norm{\Delta Y} \cdot \norm{P^{-1}}
  + \norm{Y} \cdot \norm{P^{-1}}^2 \cdot \norm{\Delta P}
 \Big)^2 }                                                                              \\
  & \quad\quad\le \sqrt{\, 1 + \norm{P^{-1}}^2 + \norm{Y}^2 \, \norm{P^{-1}}^4 \,}\cdot
 \norm{(\Delta P,\Delta Y)}.
\end{align*}
Let us assume that the variable $P$ is bounded away from zero with a sufficiently small constant, i.e.,
$P\succeq \delta I,$ and the variable $Y$ is bounded. Then we can bound the operator norm by
$$
 \norm{ D(c^{-1})_{(P,Y)}} \leq
 \sqrt{1 + \tfrac{1}{\delta^2}
  + \tfrac{\norm{Y}^2}{\delta^4}}
 =
 \sqrt{\, 1 + \tfrac{1}{\delta^2}
  + \tfrac{\norm{(K-K^*)\cdot P}^2}{\delta^4}},
$$
showing that $c^{-1}$ is Lipschitz, which verifies \eqref{eq:HCDefinitionNormCfunctionInequality}.

We define $A^* \define A + B K^*$, $Q^* \define Q + (K^*)^{\top} R K^*,$
then \eqref{eq:LQR_nonlinear_constraint} is reformulated as
\begin{align}\label{eq:LQR_linear_constraint}
  & \min_{\uu := (P, Y) }  H_1(u) \define \operatorname{Tr}\rb{(Q^* P + P^{-1} Y^{\top} R Y +  (K^*)^{\top} R Y + Y^{\top} R K^* } \tag{L-Eq-LQR} \\
  & \quad  \uu \in \UU \define \cb{ (P, Y) \mid P \succ 0, P = P^{\top}} , \notag                                                                 \\
  & \quad \text{s.t.} \quad H_{2, \text{eq}}(\uu) \define A^* P + B Y + (A^*P + B Y)^{\top} + W = 0 . \notag
\end{align}
Notice that the objective is convex and the constraint is linear
in the new variables. Although such reformulation exists, the
change of variables $\uu = c(\xx)$ depends on the unknown
optimal controller $K^*$, thus we need to solve the problem in
the original variable $\xx = (P, K).$
We note that \eqref{eq:LQR_linear_constraint} involves linear
equality constraints; refer to
\Cref{subsec:nonconvex_equality_constraints},
where we show how multiple scalar equality type constraints can be
reformulated into a single inequality constraint using the max-operator, matching our problem formulation
\eqref{eq:MainProblem}.

\subsection{Observations for Convergence Analysis}
We state the following basic proposition from
\cite{fatkhullinStochasticOptimizationHidden2024}, which is a useful
observation for convergence analysis of gradient methods under hidden
convexity.
\begin{proposition}[\cite{fatkhullinStochasticOptimizationHidden2024} Prop. 3]
 \label{lemma:HCContractionInequality}
 Let \eqref{eq:MainProblem} be hidden convex with $\mu_c > 0$. For any $\alpha \in [0,1]$
 and $\xx, \yy \in \XX$, define $\xx_\alpha \define c^{-1}((1-\alpha) c(\xx) + \alpha c(\yy))$,
 then, for $i =1, 2$, the following functional inequality
 \begin{align}
  F_i(\xx_\alpha) \leq (1-\alpha)F_i(\xx) + \alpha F_i(\yy)
  \tag{HC-FI}
  \label{eq:HCContractionInequalityObjective}
 \end{align}
 and the norm inequality
 \begin{align}
  \norm{\xx_\alpha - x} \leq \frac{\alpha}{\mu_c} \norm{c(\xx) - c(\yy)}
  \tag{HC-NI}
  \label{eq:HCContractionInequalityNorm}
 \end{align}
 hold.
\end{proposition}
These two inequalities will be used multiple times in the subsequent sections. The proof is included in \Cref{Appendix-sec:ProofsHiddenConvexClass} for completeness.

%% file: sections/03_PPM_analysis.tex
\section{Proximal Point Method (PPM)}
\label{sec:AnalysisPPM}

We begin the algorithmic part by analyzing the Proximal Point Method under
hidden convexity. This approach seems promising in light of the prior work,
which used PPM framework for non-asymptotic analysis, cf.
\Cref{subsec:related_work}. Specifically, we focus on an inexact variant
(IPPM) using a \textit{feasible inner solver} $\Oracle$; see
\Cref{algo:PPM}. Such feasible inner solver will be crucial to avoid
constraint qualification (CQ) assumptions. The main challenge is to prove
convergence of IPPM to a global optimum under hidden convexity when
$\Oracle$ is only approximate. Our key observation is that, even if the
original problem fails Slater's condition, each IPPM subproblem can be made
to satisfy Slater's condition via a \emph{constraint shift}. In
\Cref{sec:ComplexityAnalysis} we instantiate the inner solver $\Oracle$ and
derive total (sub-)gradient oracle complexities for \eqref{eq:MainProblem}
problems, depending on the smoothness levels of $F_1$ and $F_2$.

The central device enabling our IPPM analysis is a two-level shift of the
constraint that enlarges the feasible set of each subproblem: (i) an outer
shift by a budget $\tau$ (chosen on the order of the target accuracy
$\epsilon$), and (ii) an inner shift used by the inner solver $\Oracle$
that depends on the Slater gap of the current IPPM subproblem (typically
smaller), cf. \Cref{algo:SwitchingSubgradient,algo:ACGD}.

\begin{algorithm}[H]
 \caption{$\IPPM(F_1, F_2, \xxZero, \epsilon, \tau, N, \rhoHat)$\\
  Inexact Proximal Point Method for \eqref{eq:MainProblem}}
 \begin{algorithmic}[1]
  \State \textbf{Input:} Objective $F_1$, constraint $F_2$, initial point $\xxZero \in \mathcal{X} \cap
   \{F_2(\innerEmpty{}) \leq \tau \}$,
  accuracy $\epsilon$,
  \highlight{constraint violation budget $\tau$},
  outer loops $N$,
  inner (feasible) algorithm $\Oracle$, regularization parameter $\rhoHat >
   \rho$ \For {$k = 0, 1, \dots, N-1$} \State
  Define for $\xx \in \XX$
  \begin{align*}
   \phi_{1}^{(k)}(\xx) \define F_1(\xx) + \frac{\rhoHat}{2}\norm{\xx - \xxK}^2, \qquad \phi_{2}^{(k)}(\xx) \define F_2(\xx) + \frac{\rhoHat}{2}\norm{\xx - \xxK}^2
  \end{align*}
  \State Compute an approximate feasible solution to \eqref{eq:InexactPPMProblemInnerLoopHCHC} via
  $$\xxKnext \gets
   \Oracle(\phi_1^{(k)}, \phi_2^{(k)}, \xxK,
   {\Tinner}, \tau)$$
  \EndFor
  \State \textbf{Return:} $\xx^{(N)}$
 \end{algorithmic}
 \label{algo:PPM}
\end{algorithm}

\subsection{Convergence Analysis of Inexact PPM}\label{subsec:InexactPPM}
IPPM solves at each iteration $k \in \mathbb{N}$ the (strongly convex)
\highlight{shifted} subproblem
\begin{align}
 \begin{aligned}
  \xxKnext \approx \xxHatKnext \define \arg\min_{\xx \in \XX} \; \phi_1^{(k)}(\xx)
  \define F_1(\xx) & + \frac{\rhoHat}{2}\norm{\xx - \xxK}^2 , \\
  \text{s.t. }  \phi_2^{(k)}(\xx)
  \define F_2(\xx) & + \frac{\rhoHat}{2} \norm{\xx - \xxK}^2
  \;\highlight{-}\; \highlight{\tau} \leq 0.
 \end{aligned}
 \tag{Shifted-IPPM}
 \label{eq:InexactPPMProblemInnerLoopHCHC}
\end{align}
Here, $\xxHatKnext$ is the exact subproblem minimizer and $\xxKnext$ is an approximate solution.

The \textbf{proof sketch} of the analysis is the following:
\begin{itemize}
 \item \emph{Initialization.} Assume $\xxZero$ is $\tau$-feasible for \eqref{eq:MainProblem}, i.e., $F_2(\xxZero) \le \tau$. If not, a simple (sub-)gradient method on $F_2$ finds such a point in $\OOTilde{1/\tau}$ gradient evaluations in the smooth case or $\OOTilde{1/\tau^3}$ in the non-smooth case; this does not change the overall complexity.
 \item \emph{Feasibility preservation.} If $\xxK$ is $\tau$-feasible for \eqref{eq:MainProblem}, it is (trivially) feasible for \eqref{eq:InexactPPMProblemInnerLoopHCHC}. Running a \emph{feasible} inner method on \eqref{eq:InexactPPMProblemInnerLoopHCHC} from $\xxK$ yields $\phi_2^{(k)}(\xxK) \le \tau$ and hence $F_2(\xxK) \le \tau$. Thus all outer iterates remain $\tau$-feasible for \eqref{eq:MainProblem}.

 \item \emph{Slater points for subproblems.} Define $\xxK_\alpha \define c^{-1}\big((1-\alpha)c(\xxK) + \alpha c(\xx^\star)\big)$. For sufficiently small $\alpha$ (relative to $\tau$), the point $\xxK_\alpha$ is
       \emph{strictly feasible} for \eqref{eq:InexactPPMProblemInnerLoopHCHC}. This proves Slater's condition for \eqref{eq:InexactPPMProblemInnerLoopHCHC} at each iteration $k\geq0$ and allows us to apply a feasible method to solve this subproblem.

 \item \emph{Optimality Improvement.} It remains to build a PPM recursion similar to analysis in \cite{fatkhullinStochasticOptimizationHidden2024} to guarantee the improvement in $F_1$, up to errors from inexact inner solvers.

\end{itemize}

Now we formalize our key ``HC--Slater's Lemma'', which verifies the
subproblems \eqref{eq:InexactPPMProblemInnerLoopHCHC} satisfy Slater's
condition under a suitable reference point $\xxK$.

\begin{lemma}[HC--Slater's Lemma]\label{le:Slater_PPM_inner}
 Assume that \eqref{eq:MainProblem} is hidden convex,
 \Assref{ass:AssumptionExactPPM}.\ref{ass:AssumptionExactPPM-ITEM3} holds and $\xxK$ is $\tau$--feasible for \eqref{eq:MainProblem}. Then
 \begin{enumerate}
  \item \eqref{eq:InexactPPMProblemInnerLoopHCHC} satisfies
        $\frac{\alpha \tau }{2}$--Slater's condition with
        $\alpha \leq \min\left\{1, \frac{\mu_c^2 \tau}{\rhoHat \DDUsq} \right\}$.
        \label{le:Slater_PPM_inner-ITEM1}
  \item If, additionally, \eqref{eq:MainProblem} satisfies $\theta$--Slater's
        condition, then \eqref{eq:InexactPPMProblemInnerLoopHCHC} satisfies
        $\frac{\beta \theta }{2}$--Slater's condition with $\beta \leq
         \min\left\{1, \frac{\mu_c^2 \theta}{\rhoHat \DDUsq} \right\}$.
        \label{le:Slater_PPM_inner-ITEM2}
 \end{enumerate}
\end{lemma}

\begin{proof}
 \textit{Part 1. } Fix an iteration index $k \in \setN$, and for any $\alpha \in [0,1]$ define
 $$\xx^{(k)}_\alpha \define c^{-1}\left((1-\alpha) c(\xxK) +
  \alpha c(\xxOpt)\right).$$
 We will show that
 $\xx^{(k)}_\alpha$ is a Slater point for subproblem \eqref{eq:InexactPPMProblemInnerLoopHCHC}. Indeed,
 \begin{align}
                  & F_2(\xx^{(k)}_\alpha) + \frac{\rhoHat}{2} \norm{\xx^{(k)}_\alpha - \xxK}^2 - \tau \notag \\
  \overset{(i)}   & {\leq} (1-\alpha) F_2(\xxK) + \alpha F_2(\xxOpt) + \frac{\rhoHat}{2}
  \norm{\xx^{(k)}_\alpha - \xxK}^2 - \tau \notag                                                             \\
  \overset{(ii)}  & {\leq} (1-\alpha) \tau + \alpha \cdot 0 + \frac{\rhoHat}{2}
  \norm{\xx^{(k)}_\alpha - \xxK}^2 - \tau \notag                                                             \\
                  & = -\alpha \tau + \frac{\rhoHat}{2} \norm{c^{-1}\left((1-\alpha) c(\xxK)
  + \alpha c(\xxOpt)\right) - c^{-1}(c(\xxK))}^2 \notag                                                      \\
  \overset{(iii)} & {\leq} -\alpha \tau + \frac{\rhoHat}{2} \frac{\alpha^2}{\mu_c^2}
  \norm{c(\xxK) - c(\xxOpt)}^2 \notag                                                                        \\
  \overset{(iv )} & {\leq} -\alpha \tau + \frac{\rhoHat}{2} \DDUsq \frac{\alpha^2}{\mu_c^2}
  \overset{\space}{\leq} - \frac{\alpha \, \tau}{2} < 0, \label{eq:strict_feasibility}
 \end{align}
 where we used \eqref{eq:HCContractionInequalityObjective} for $F_2$
 in $(i)$, $\tau$-feasibility of $\xxK$ in $(ii)$, \eqref{eq:HCContractionInequalityNorm} in $(iii),$
 as well as the boundedness of $\UU$ domain in $(iv)$. The last inequality follows by the choice of $\alpha.$

 \textit{Part 2. } Let $\yySlater \in \XX$ be a $\theta$--Slater point of \eqref{eq:MainProblem},
 and for any $k \in \setN$, $\beta \in [0,1]$ define
 $$\yySlater^{(k)}_\beta \define c^{-1}\left((1-\beta) c(\xxK) +
  \beta c(\yySlater)\right).$$
 We will show that
 $\yySlater^{(k)}_\beta$ is a Slater point for subproblem \eqref{eq:InexactPPMProblemInnerLoopHCHC}. Indeed,
 \begin{align}
                  & F_2(\yySlater^{(k)}_\beta) + \frac{\rhoHat}{2} \norm{\yySlater^{(k)}_\beta - \xxK}^2 - \tau \notag \\
  \overset{(i)}   & {\leq} (1-\beta) F_2(\xxK) + \beta F_2(\yySlater) + \frac{\rhoHat}{2}
  \norm{\yySlater^{(k)}_\beta - \xxK}^2 - \tau \notag                                                                  \\
  \overset{(ii)}  & {\leq} (1-\beta) \tau - \beta \cdot \theta + \frac{\rhoHat}{2}
  \norm{\yySlater^{(k)}_\beta - \xxK}^2 - \tau \notag                                                                  \\
  \overset{(iii)} & {\leq} -\beta \theta + \frac{\rhoHat}{2} \norm{c^{-1}\left((1-\beta) c(\xxK)
  + \beta c(\yySlater)\right) - c^{-1}(c(\xxK))}^2 \notag                                                              \\
  \overset{(iv)}  & {\leq} -\beta \theta + \frac{\rhoHat}{2} \frac{\beta^2}{\mu_c^2}
  \norm{c(\xxK) - c(\yySlater)}^2 \notag                                                                               \\
  \overset{(v )}  & {\leq} -\beta \theta + \frac{\rhoHat}{2} \DDUsq \frac{\beta^2}{\mu_c^2}
  \overset{\space}{\leq} - \frac{\beta \, \theta}{2} < 0, \notag
 \end{align}
 where we used \eqref{eq:HCContractionInequalityObjective} for $F_2$
 in $(i)$, $\theta$--Slater point in $(ii)$,
 $(1-\beta)\tau \leq \tau$ since $\beta \leq 1$ in $(iii)$,
 \eqref{eq:HCContractionInequalityNorm} in $(iv)$,
 as well as the boundedness of $\UU$ domain in $(v)$. The last inequality follows by the choice of $\beta.$
\end{proof}

\Cref{le:Slater_PPM_inner} says that regardless whether \eqref{eq:MainProblem} satisfies Slater's condition,
a Slater point for PPM sub-problem \eqref{eq:InexactPPMProblemInnerLoopHCHC} always exists.
This allows us to apply a feasible method to solve such sub-problems, i.e., we suppose the algorithm $\Oracle_{\epsin}$ solves
\eqref{eq:InexactPPMProblemInnerLoopHCHC} to $(\epsin,0)$-optimality for any
target precision $\epsin > 0$, i.e.
\begin{equation}
 \begin{aligned}
  \phiK{1}(\xxKnext) - \phiK{1}(\xxHatKnext) & \leq \epsin , \\
  \phiK{2}(\xxKnext) - \tau                  & \leq 0 .
 \end{aligned}
 \tag{IPPM-Feas}
 \label{eq:Feasible_oracle_PPM}
\end{equation}
We will see in \Cref{sec:ComplexityAnalysis} that such algorithms are
readily available in the literature (with a slight modification of shifting the constraint $\phiK{2}$
with a value less than $\frac{\alpha \tau}{2}$ or $\frac{\alpha \theta}{2}$), e.g.,
switching sub-gradient \cite{lanAlgorithmsStochasticOptimization2020, jiaFirstOrderMethodsNonsmooth2025}
or fast primal-dual \cite{zhangSolvingConvexSmooth2022} methods.

Now we are set to show the main result, i.e. the convergence of IPPM.
% \begin{tcolorbox}[colback=white,colframe=blueETH!75!black]
\begin{theorem}[Inexact PPM]
 \label{thm:MainResultHCPPM}
 Assume that \eqref{eq:MainProblem} is hidden convex and \Assref{ass:AssumptionExactPPM} holds.
 Let $\xxZero$ be $\tau$--feasible for \eqref{eq:MainProblem} and algorithm
 $\Oracle_{\epsin}$ initialized with a feasible point $\xxK$ outputs a  point
 $\xxKnext$ satisfying \eqref{eq:Feasible_oracle_PPM} after $\Tinner = \Tinner(\epsin)$ (sub-)gradient and function evaluations.
 Given a lifting parameter $\rhoHat > \rho$ (here $\rho>0$ is the weak
 convexity parameter) and a desired tolerance $\epsilon > 0$ for the
 optimality gap, assume $\epsilon\leq \frac{3\rhoHat\DDUsq}{2\mu_c^2}$ and
 $\tau \leq \frac{\rhoHat\DDUsq}{2\mu_c^2}$ hold.
 Then setting $\rhoHat \define 2\rho$ and $\epsin \leq \fr{\epsilon}{ 3}
  \min\left\{\frac{2\mu_c^2 \epsilon}{3\rhoHat\DDUsq}, \frac{\mu_c^2
   \tau}{\rhoHat \DDUsq}\right\}$, the last iterate of \Cref{algo:PPM}
 satisfies
 \begin{align}
  F_1(\xxN) - \Fopt \leq \epsilon,
  \quad F_2(\xxN) \leq \tau
  \tag{IPPM-Opt}
  \label{eq:HCHCInexactPPMlastIterateResult}
 \end{align}
 after
 \begin{align*}
  N \geq \max\left\{\frac{3\rho\DDUsq}{\mu_c^2 \epsilon},
  \frac{2 \rho \DDUsq}{\mu_c^2 \tau}\right\}
  \cdot \log\left(\frac{3 \Delta_0}{\epsilon}\right)
 \end{align*}
 iterations, where $\Delta_0 \define F_1(\xxZero)- \Fopt$. The total oracle complexity is given
 by
 \begin{align*}
  \Ttotal \geq N \cdot \Tinner\paren{\epsin}
  = \max\left\{\frac{3\rho\DDUsq}{\mu_c^2 \epsilon},
  \frac{2 \rho \DDUsq}{\mu_c^2 \tau}\right\}
  \cdot \log\left(\frac{3 \Delta_0}{\epsilon}\right) \cdot
  \Tinner\paren{\epsin}.
 \end{align*}
\end{theorem}
% \end{tcolorbox}

In \Cref{sec:ComplexityAnalysis}, we will compute the bounds for the oracle
complexity $\Ttotal$ as a function of $\varepsilon,$ specifying $\Tinner,$
but first we prove the above central theorem.

\begin{proof}
 By the HC--Slater's \Cref{le:Slater_PPM_inner}.\ref{le:Slater_PPM_inner-ITEM1},
 we have the point
 $$\xx^{(k)}_\alpha \define c^{-1}\left((1-\alpha) c(\xxK) +
  \alpha c(\xxOpt)\right)$$
 with $\alpha \leq
  \frac{\mu_c^2 \tau}{\rhoHat \DDUsq} $ is feasible for \eqref{eq:InexactPPMProblemInnerLoopHCHC} at iteration $k$, i.e.,
 $\xx^{(k)}_\alpha \in \XX \cap \{F_2(\innerEmpty{}) + \frac{\rhoHat}{2}
  \norm{\innerEmpty{} - \xxK}^2 - \tau \leq 0\}$. Since $\Oracle_{\epsin}$ satisfies $\phiK{1}(\xxKnext) \leq \phiK{1}(\xxHatKnext) + \epsin$ and using the optimality condition for $\xxHatKnext$, we derive
 \begin{align}
  F_1(\xxKnext) - \Fopt & \leq F_1(\xxKnext) - \Fopt + \frac{\rhoHat}{2}
  \norm{\xxKnext - \xxK}^2 \notag                                                                                                                            \\
                        & = \phiK{1}(\xxKnext) - \Fopt \notag                                                                                                \\
                        & \leq \phiK{1}(\xxHatKnext)  + \epsin - \Fopt \notag                                                                                \\
                        & \leq \phiK{1}(\xxK_{\alpha}) + \epsin - \Fopt \label{eq:req_main_x_alpha1}
  \\
                        & = F_1(\xx^{(k)}_\alpha) - \Fopt + \frac{\rhoHat}{2}
  \norm{\xx^{(k)}_\alpha - \xxK}^2  + \epsin \notag                                                                                                          \\
                        & \leq (1-\alpha) \left(F_1(\xxK)- \Fopt \right) + \frac{\rhoHat\DDUsq \alpha^2}{2 \mu_c^2}  + \epsin , \label{eq:req_main_x_alpha2}
 \end{align}
 where the second inequality holds by the property of $\Oracle_{\epsin}$, the third is by optimality of $\xxHatKnext$ and feasibility of $\xxK_\alpha$.
 The last step is due to \eqref{eq:HCContractionInequalityObjective},
 \eqref{eq:HCContractionInequalityNorm}, and the bound on the diameter of $\UU$.
 By unrolling the above recursion from $k=1, \ldots, N-1$ and using that the partial sum
 of the geometric series
 is bounded by $\nfr{1}{\alpha}$, we derive
 \begin{align*}
  F_1(\xxN) - \Fopt \leq (1-\alpha)^N (F_1(\xx^{(0)}) - \Fopt)
  + \frac{\rhoHat\DDUsq \alpha}{2 \mu_c^2} + \frac{\epsin}{\alpha}  \overset{\space}{\leq} \epsilon,
 \end{align*}
 where the last step holds by the choice of $\epsin$, $N$ as in the Theorem statement, and $\alpha \leq \min\left\{\frac{2\mu_c^2 \epsilon}{3\rhoHat\DDUsq}, \frac{\mu_c^2 \tau}{\rhoHat \DDUsq}\right\} \in [0,1].$
\end{proof}

\begin{remark}
 It is possible to improve the above oracle complexity in the case of $\mu_H$--hidden strong convexity,
 cf. \Cref{thm:MainResultHSC-PPM} in
 \Cref{appendix:InexactPPMunderHiddenStrongConvexity}.
\end{remark}

\subsection{Oracle Complexity Analysis for IPPM}
\label{sec:ComplexityAnalysis}

In this section, we specify algorithms for solving the sub-problem in
\Cref{algo:PPM} and compute the upper bounds on the oracle complexity of
the resulting methods. We distinguish between two important cases. First,
we consider the problem without any CQ in the non-smooth and smooth
settings. Second, we establish faster convergence assuming a Slater point
exists for problem \eqref{eq:MainProblem}.

\subsubsection{Oracle complexity without Slater's assumption}
\label{subsubsec:IPPM_no_Slater}
To compute the total oracle ((sub-)gradient and function evaluation)
complexity of IPPM in the non-smooth setting without any constraint
qualification (CQ), we use a slight modification of the SwSG method
\cite{polyakGeneralMethodSolving1967, lanAlgorithmsStochasticOptimization2020, jiaFirstOrderMethodsNonsmooth2025} in
place of $\Oracle_{\epsin}$, see \Cref{algo:SwitchingSubgradient}. We use SwSG since it is optimal for non-smooth convex (and strongly convex) constrained optimization and does not require CQ assumptions.

\begin{algorithm}[H]
 \caption{$\mathrm{SwSG}(\phi_1^{(k)}, \phi_2^{(k)}, \xxK, \Tinner, \tau)$
  \\Switching Sub-Gradient Method
  \cite[Algo 1]{jiaFirstOrderMethodsNonsmooth2025} for solving \eqref{eq:InexactPPMProblemInnerLoopHCHC}}
 \label{algo:SwitchingSubgradient}
 \begin{algorithmic}[1]
  \State \textbf{Input:} Regularized objective $\phiK{1}$ and constraint $\phiK{2}$,
  current IPPM-iterate $\xxK \in \XX$, number of steps $\Tinner \in \NN$, constraint violation budget $\tau$
  \State Define precision $\epsin>0$, parameter $\alpha > 0$, and stepsizes $(\eta_t)_t$ as in \Cref{cor:FinalRateNonSmoothIPPM}
  \State Initialize $\zzZero \gets \xx^{(k)} \in \XX$
  \State Define $\feas \gets \emptyset, \infeas \gets \emptyset$
  \State {\highlight{Shift constraint $\phiKshift{2}(\xx) \define \phiK{2}(\xx) - \tau + \frac{\alpha \tau}{3}$, $\xx \in \XX$}}
  \For{$t = 1, 2, 3, \dots, \Tinner - 1$}
  \If{$\phiKshift{2}(\zzT) \leq \epsin$} \algorithmiccomment{feasible}
  \State $\zzTnext \gets \proj_{\XX}\paren{\zzT - \eta_t \zeta^{(t)}_1}$,
  with $\zeta^{(t)}_1 \in \partial \phiK{1}(\zzT)$,
  $\feas \gets \feas \cup \cbrac{t}$
  \Else \algorithmiccomment{infeasible}
  \State $\zzTnext \gets \proj_{\XX}\paren{\zzT - \eta_t \zeta^{(t)}_2}$,
  with $\zeta^{(t)}_2 \in \partial \phiK{2}(\zzT)$,
  $\infeas \gets \infeas \cup \cbrac{t}$
  \EndIf
  \EndFor
  \State \Return $\xxKnext \gets \sum_{t\in \feas} (t+1) \zzT / \left( \sum_{t\in \feas} (t+1) \right)$
 \end{algorithmic}
\end{algorithm}

\begin{corollary}[IPPM+SwSG, Non-smooth, No CQ]
 \label{cor:FinalRateNonSmoothIPPM}
 Under the assumptions of \Cref{thm:MainResultHCPPM}, when using the
 SwSG (\Cref{algo:SwitchingSubgradient}) method as the inner solver,
 the total number of sub-gradient calls and function evaluations
 required to achieve \eqref{eq:HCHCInexactPPMlastIterateResult} is given by:
 \begin{align*}
  \TtotalUp{\nonSmoothLabel} & \geq N \cdot \TinnerUp{SwSG}\paren*{\epsin}              \\
                             & \geq \frac{3\rho\DDUsq}{\mu_c^2 \min\left\{ \varepsilon,
   \tau\right\} } \cdot \log\left(\frac{3 \Delta_0}{\epsilon}\right) \cdot \max\left\{
  \frac{216 (3 G^2 - 4\rho \FLB{2})  \DDUsq }{ \mu_c^2 \min\left\{ \varepsilon^2,
   \tau^2\right\} } ,
  \frac{51 \rho \DDX \DDU }{ \mu_c \min\left\{ \varepsilon,
   \tau\right\} }
  \right\},
 \end{align*}
 where $\FLB{2} \define \min_{\xx \in \XX} F_2(\xx) > -\infty$.
\end{corollary}
\begin{proof}
 The proof is deferred to \Cref{subsec:ProofofCorollaries}.
\end{proof}

Remarkably, the above result matches (in terms of $\varepsilon$) the
sub-gradient complexity for solving hidden convex unconstrained problems
\cite{fatkhullinStochasticOptimizationHidden2024} (see also
\Cref{corr:UnconstrainedHiddenConvex} in
\Cref{subsubsec:OptCritUnconstrainedcase} for discussion about the
optimality criterion). Importantly, in the constrained setting, $F_2(\cdot)
 \not \equiv 0$ requires no constraint qualification (CQ) conditions. This
is consistent with similar results in convex case, where the oracle
complexity of SwSG method \cite{polyakGeneralMethodSolving1967} matches the
one of (sub-)gradient method, i.e., $\OO{\varepsilon^{-2}}$.

Next, we consider the smooth setting, where the oracle complexity can be
further improved from $\OOTilde{\varepsilon^{-3}}$ to
$\OOTilde{\varepsilon^{-2}}$ under hidden convexity. To achieve this, we do
not require CQs, but we need to use a faster inner solver to reduce the
inner solver complexity, $\Tinner\paren{\epsin}.$ We use a suitably
modified (shifted) version of ACGD algorithm from
\cite{zhangSolvingConvexSmooth2022}, see \Cref{algo:ACGD}. ACGD achieves
the optimal oracle complexities in (strongly) convex smooth constrained
optimization (improving, e.g., SwSG), which translates to improvements of
the total oracle complexities for our problem.

\begin{algorithm}[H]
 \caption{$\mathrm{ACGD}(\phi_1^{(k)}, \phi_2^{(k)}, \xxK, \Tinner, \tau)$\\
  Accelerated Constrained Gradient Descent Method
  \cite[Algo 1]{zhangSolvingConvexSmooth2022} for solving
  \eqref{eq:InexactPPMProblemInnerLoopHCHC}}
 \label{algo:ACGD}
 \begin{algorithmic}[1]
  \State \textbf{Input:} Regularized objective $\phiK{1}$ and constraint $\phiK{2}$,
  current PPM-iterate $\xxK \in \XX$,
  constraint violation budget $\tau$,
  \State
  Define stepsizes $\{\theta_t\}, \{\eta_t\}, \{\tau_t\}$,
  and weights $\{\omega_t\}$ according to
  \Cref{cor:FinalRateSmoothIPPM_ACGD} or
  \Cref{cor:FinalRateSmoothIPPM_ACGD_Slater} if Slater's condition holds
  \State Counter $t \gets 0$ and initialize $\zzTminusOne \gets \xxK,
   \zzTUnder \gets \xxK, \zzT \gets \xx^{(k)} \in \XX$
  \State Set $\pi^{(0)} \gets \nabla \phiK{1}(\zzZero)$, $\nu^{(0)} \gets \nabla \phiK{2}(\zzZero)$
  \If{Slater's condition holds}
  \State $\offset \gets - \tau + \frac{\beta \theta}{3}$ according to \Cref{cor:FinalRateSmoothIPPM_ACGD_Slater}
  \Else
  \State $\offset \gets - \tau + \frac{\alpha \tau}{3}$
  according to \Cref{cor:FinalRateSmoothIPPM_ACGD}
  \EndIf
  \State {\highlight{Shift constraint $\phiKshift{2}(\xx) \define \phiK{2}(\xx) - b$, $\xx \in \XX$}}
  \For{$t = 1, 2, 3, \dots, \Tinner$}
  \State Set $\zzTUnder \gets (\tau_t \zzTminusOneUnder + \zzTildeT)/(1 + \tau_t)$
  where $\zzTildeT = \zzTminusOne + \theta_t (\zzTminusOne - \zzTminusTwo)$
  \State Set $\pi^{(t)} \gets \nabla \phiK{1}(\zzTUnder)$ and
  $\nu^{(t)} \gets \nabla \phiKshift{2}(\zzTUnder)$
  \State Solve
  \[
   \zzT \gets \arg\min_{\xx \in \XX} \cbrac*{ \langle \pi^{(t)}, \xx \rangle + \eta_t \norm{ \xx - \zzTminusOne}^2 / 2
    \;\; \text{s.t. } \nu^{(t)}\cdot(\xx - \zzTUnder) + \phiKshift{2}(\zzTUnder) \leq 0 }
  \]
  \EndFor
  \State \Return $\xxKnext \gets \sum_{t=1}^{\Tinner} \omega_t \zzT / \left( \sum_{t=1}^{\Tinner} \omega_t \right)$
 \end{algorithmic}
\end{algorithm}

\begin{corollary}[IPPM+ACGD, Smooth, No CQ]
 \label{cor:FinalRateSmoothIPPM_ACGD}
 Under the assumptions of \Cref{thm:MainResultHCPPM} and assuming that
 $F_1, F_2$ are $L$-smooth, when using the ACGD (\Cref{algo:ACGD}) method as the inner solver, the total number of gradient calls and function evaluations
 required to achieve \eqref{eq:HCHCInexactPPMlastIterateResult} is given by:
 \begin{align*}
  \TtotalUp{\smoothLabel} & \geq N \cdot \TinnerUp{ACGD}\paren*{\epsin}           \\
                          & = \frac{3\rho\DDUsq}{\mu_c^2 \min\left\{ \varepsilon,
   \tau\right\} }
  \cdot \OOTilde{\frac{ \DDU \sqrt{L (\FUB{1} - \FLB{1} + \varepsilon)}}{ \mu_c \sqrt{\min\left\{\varepsilon,
     \tau \right\} \tau } }},
 \end{align*}
 where $\FUB{1} \define \max_{\xx \in \XX} F_1(\xx)$, $\FLB{1} \define \min_{\xx \in \XX} F_1(\xx).$
\end{corollary}
\begin{proof}
 The proof idea of the above result is to upper bound the Lagrange
 multiplier of the sub-problem \eqref{eq:InexactPPMProblemInnerLoopHCHC},
 and use it to calculate the oracle complexity bound of ACGD, see
 \Cref{subsec:ProofofCorollaries} for details.
\end{proof}

\subsubsection{Oracle complexity with Slater's assumption}
\label{subsubsec:IPPM_with_Slater}
Unfortunately, the oracle complexity in
\Cref{cor:FinalRateSmoothIPPM_ACGD} does not recover the oracle
complexity of unconstrained hidden convex problems in the smooth case
\cite{zhangVariationalPolicyGradient2020,fatkhullinStochasticOptimizationHidden2024}, which is $\OOTilde{\varepsilon^{-1}}.$ The loss in the complexity happens since the oracle complexity of the ACGD depends on
the bound of the optimal dual variable $\lambda^*$, which in our case
happens to be large -- of order $\OOTilde{\varepsilon^{-2}}$ (assuming $\varepsilon =
 \tau$),
see the proof details in \Cref{subsec:ProofofCorollaries}. The following Corollary improves this
complexity in the setting when Slater's condition holds.

\begin{corollary}[IPPM+ACGD, Smooth, Slater's Condition]
 \label{cor:FinalRateSmoothIPPM_ACGD_Slater}
 Let the assumptions of \Cref{thm:MainResultHCPPM} hold and that
 $F_1, F_2$ are $L$-smooth. Assume additionally that \eqref{eq:MainProblem} satisfies $\theta$--Slater's
 condition with $\frac{\mu_c^2 \theta}{\rhoHat \DDUsq} \leq 1$. Then when using ACGD (\Cref{algo:ACGD}) method as the inner solver,
 the total number of gradient calls and function evaluations
 required to achieve \eqref{eq:HCHCInexactPPMlastIterateResult} is given by:
 \begin{align*}
  \TtotalUp{\smoothSlaterLabel} & \geq N \cdot \TinnerUp{ACGD}\paren*{\epsin}           \\
                                & = \frac{3\rho\DDUsq}{\mu_c^2 \min\left\{ \varepsilon,
   \tau\right\} }
  \cdot \OOTilde{\frac{ \DDU \sqrt{L (\FUB{1} - \FLB{1})}}{ \mu_c \theta }} ,
 \end{align*}
 where $\FUB{1}\define \max_{\xx \in \XX} F_1(\xx)$, $\FLB{1} \define \min_{\xx \in \XX} F_1(\xx)$.
\end{corollary}
\begin{proof}
 The proof is analogous to the proof of \Cref{cor:FinalRateSmoothIPPM_ACGD} and
 the details are deferred to \Cref{subsec:ProofofCorollaries}.
\end{proof}

This result implies that if our original problem \eqref{eq:MainProblem}
satisfies the $\theta$--Slater's condition, the oracle complexity of
\Cref{algo:PPM} can be further improved to $\OOTilde{\theta^{-1}
  \varepsilon^{-1}}$, up to logarithmic factors, matching the complexity for
solving unconstrained hidden convex problems when $\theta = \OOTilde{1}.$
\begin{remark}
 The results of \Cref{cor:FinalRateSmoothIPPM_ACGD,cor:FinalRateSmoothIPPM_ACGD_Slater}
 only focus on the oracle complexity and do not take into account the computational complexity for
 solving quadratic sub-problems with linear constraints in ACGD method. However, the computational complexity
 (number of arithmetic operations and matrix-vector products) is mild and can be estimated based on Corollary 4 in
 \cite{zhangSolvingConvexSmooth2022}.
\end{remark}

\section{Bundle-level Approach}\label{sec:bundle_level}

As established in the previous section, the inexact proximal point approach
(IPPM) yields oracle complexities matching those for unconstrained hidden
convex optimization \cite{fatkhullinStochasticOptimizationHidden2024}. In
the smooth setting, however, our approach either relies on Slater's
condition (and results in $\OOTilde{\theta^{-1}\varepsilon^{-1}}$
complexity) or incurs a suboptimal $\OOTilde{\varepsilon^{-2}}$ oracle
complexity in terms of accuracy~$\varepsilon$. The former bound is not
satisfactory since it scales with the inverse of the Slater gap~$\theta$, a
quantity that may be arbitrarily close to zero in practice. This limitation
appears inherent to the IPPM approach, as our proof technique requires each
subproblem to satisfy Slater's condition. Obtaining the
$\OOTilde{\varepsilon^{-1}}$ rate (matching the unconstrained setting)
without Slater's condition would require an algorithm for smooth, strongly
convex, constrained optimization with $\OOTilde{\sqrt{\kappa}}$ complexity
independent of dual variable bounds. Such an algorithm is not known, and to
the best of our knowledge, its existence remains an open problem.

In this section, we aim to achieve fast $\OOTilde{\varepsilon^{-1}}$
gradient complexity for hidden convex problems without Slater's condition.
To do so, we take a different approach based on the cutting
plane/bundle-level idea \cite{lemarechal1995new,lan2015bundle}. First, we
want to highlight that the existing theory of bundle-level type methods is
predominantly limited to the convex setting. This is because the core idea
of this approach is to build a global lower model of the objective using
linear or piece-wise linear approximations. While such philosophy is
powerful for convex programming, it fails even on simple non-convex
problems. Assume we know the optimal value (level) $\Fopt$ and consider a
variant of bundle-level (BL) method from \cite{devanathan2024polyak}, which
has an update rule:

\begin{align}\label{eq:Polyak_Minorant}
 \xxTnext & = \argmin_{\xx \in \XX} \, \norm{\xx - \xxT}^2,
 \tag{StarBL}                                                                            \\
 \text{s.t. } \;\;
          & F_1(\xxT) + \langle\nabla F_1(\xxT), \xxT - \xx \rangle  \leq \Fopt , \notag \\
          & F_2(\xxT) + \langle\nabla F_2(\xxT), \xxT - \xx \rangle \leq 0 . \notag
\end{align}

This algorithm finds the next point $\xxTnext$ as the projection of the
current iterate $\xxT$ to the intersection of two linear constraints. We
refer to this algorithm \eqref{eq:Polyak_Minorant} as it uses the optimal
value $\Fopt.$ While this method converges for convex problems (see
\cite{devanathan2024polyak,deng2024uniformly}), it may fail for non-convex
problems.

\begin{figure}
 \centering
 \begin{subfigure}[t]{0.48\textwidth}
  \centering
  \includegraphics[height=.9\linewidth, keepaspectratio]{
   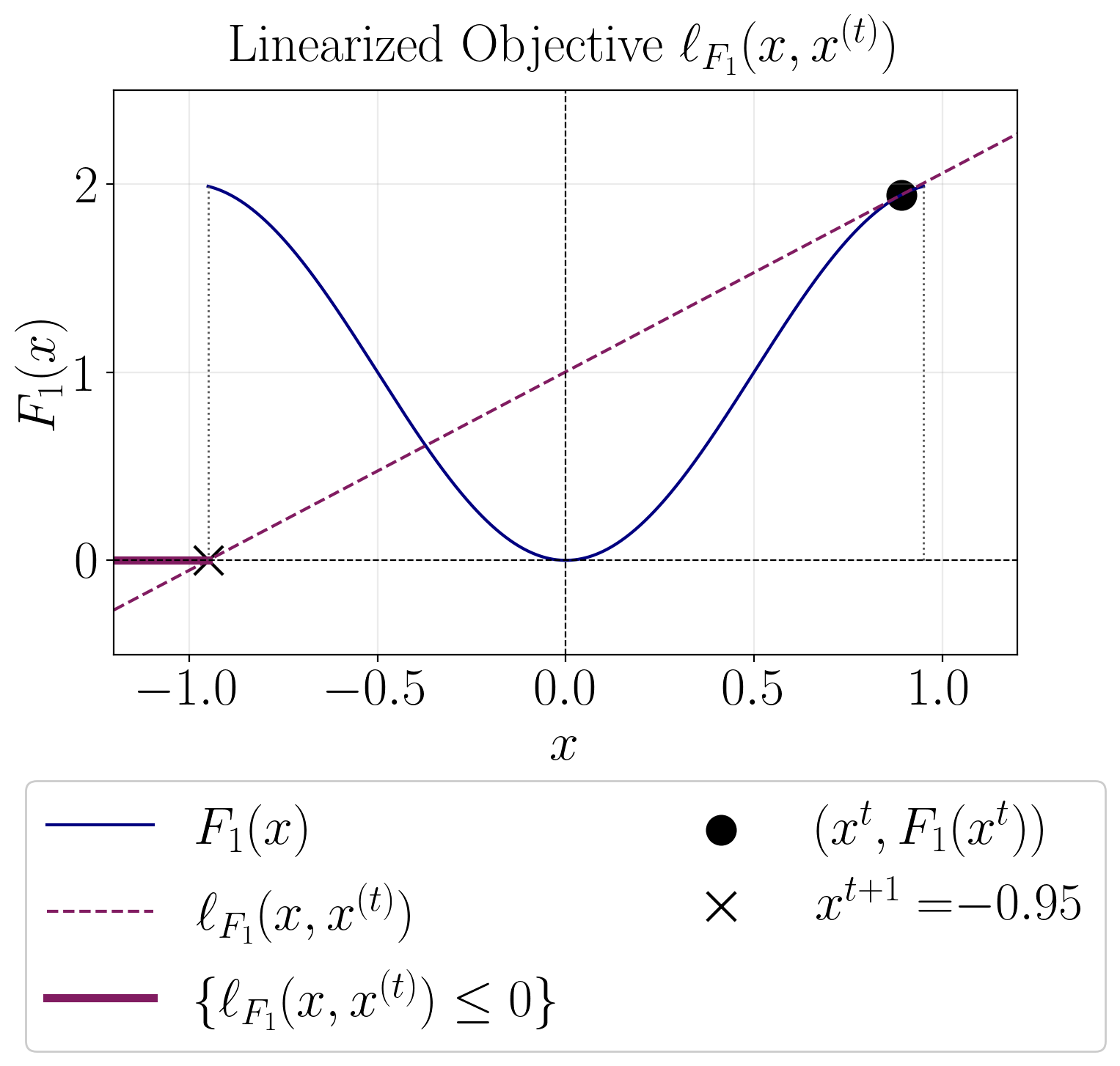
  }
  \caption[Linearized Objective Bundle-level Approach]{
  Illustration of the linearized objective \textit{without shifting},
  $\ell_{F_1}(x, x^{(t)}),$ over the set $\XX=[-0.95,0.95]$. Starting at $x^{(t)} = 0.891$,
  the lack of shifting causes \eqref{eq:Polyak_Minorant} to diverge; the next iterate $x^{(t+1)} = -0.95$ is maximum on $\XX.$}
  \label{fig:polyak_minorant_linearized_objective}
 \end{subfigure}
 \hfill
 \begin{subfigure}[t]{0.5\textwidth}
  \centering
  \includegraphics[height=.9\linewidth, keepaspectratio]{
   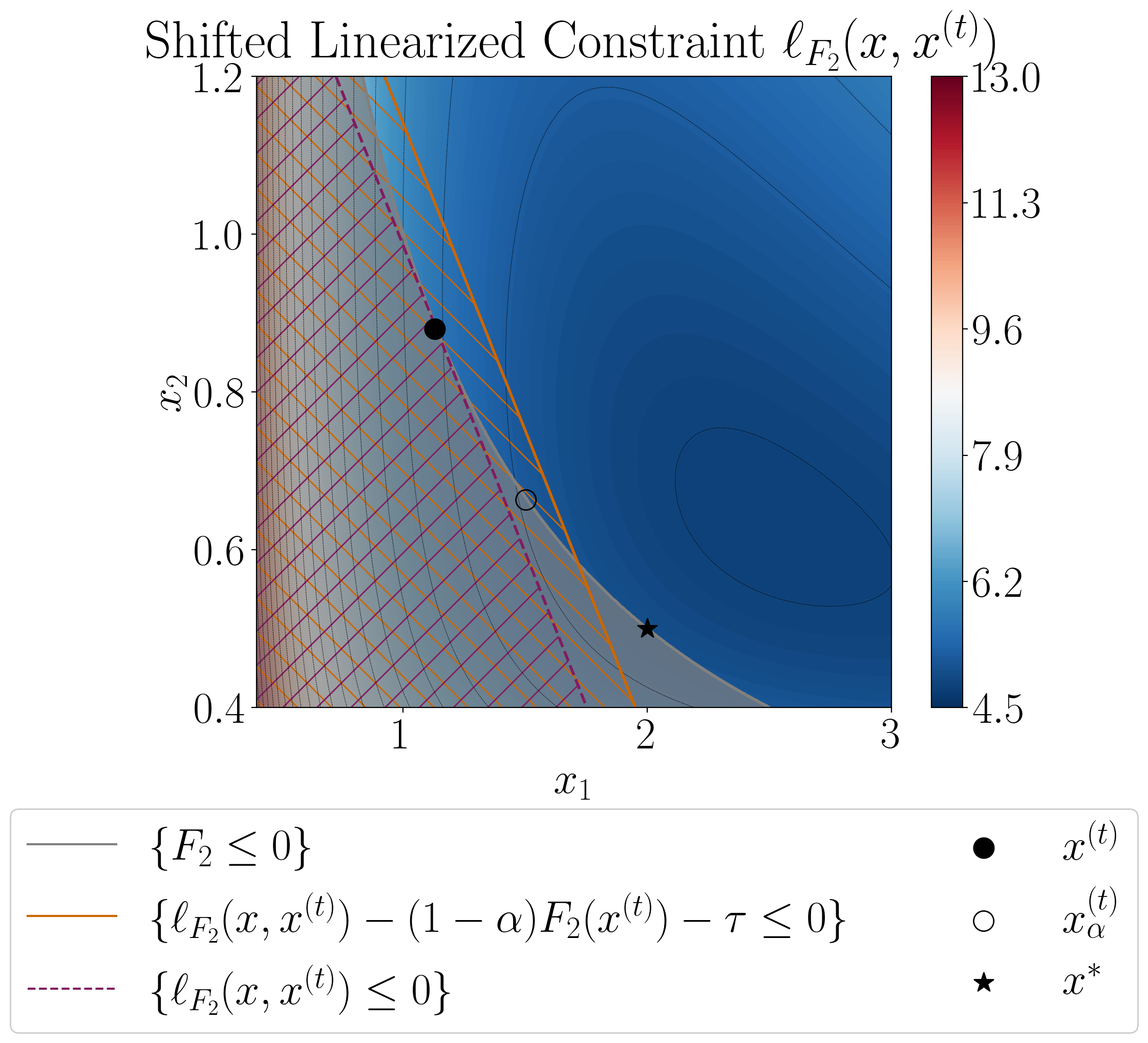
  }
  \caption[Linearized Constraint Bundle-level Approach]{
   Illustration of the \textit{shifted} linearized constraint, $\ell_{F_2}(\xx, \xx^{(t)}
    ),$ at $\xx^{(t)} = (1.34, 0.74)^\top$. The point
   $\xx^{(t)}_\alpha$ is feasible for \eqref{eq:modified_Polyak_Minorant_subproblem} subproblem when
   $\alpha, \tau$ are chosen according to \Cref{lemma:polyak_minorant_approx_feasibility}
   and \Cref{thm:MainResultHCPolyakMinorant} respectively.
  }
  \label{fig:polyak_minorant_linearized_constraint}
 \end{subfigure}
 \hfill
 \begin{subfigure}{0.8\textwidth}
  \centering
  \includegraphics[height=.6\linewidth, keepaspectratio]{
   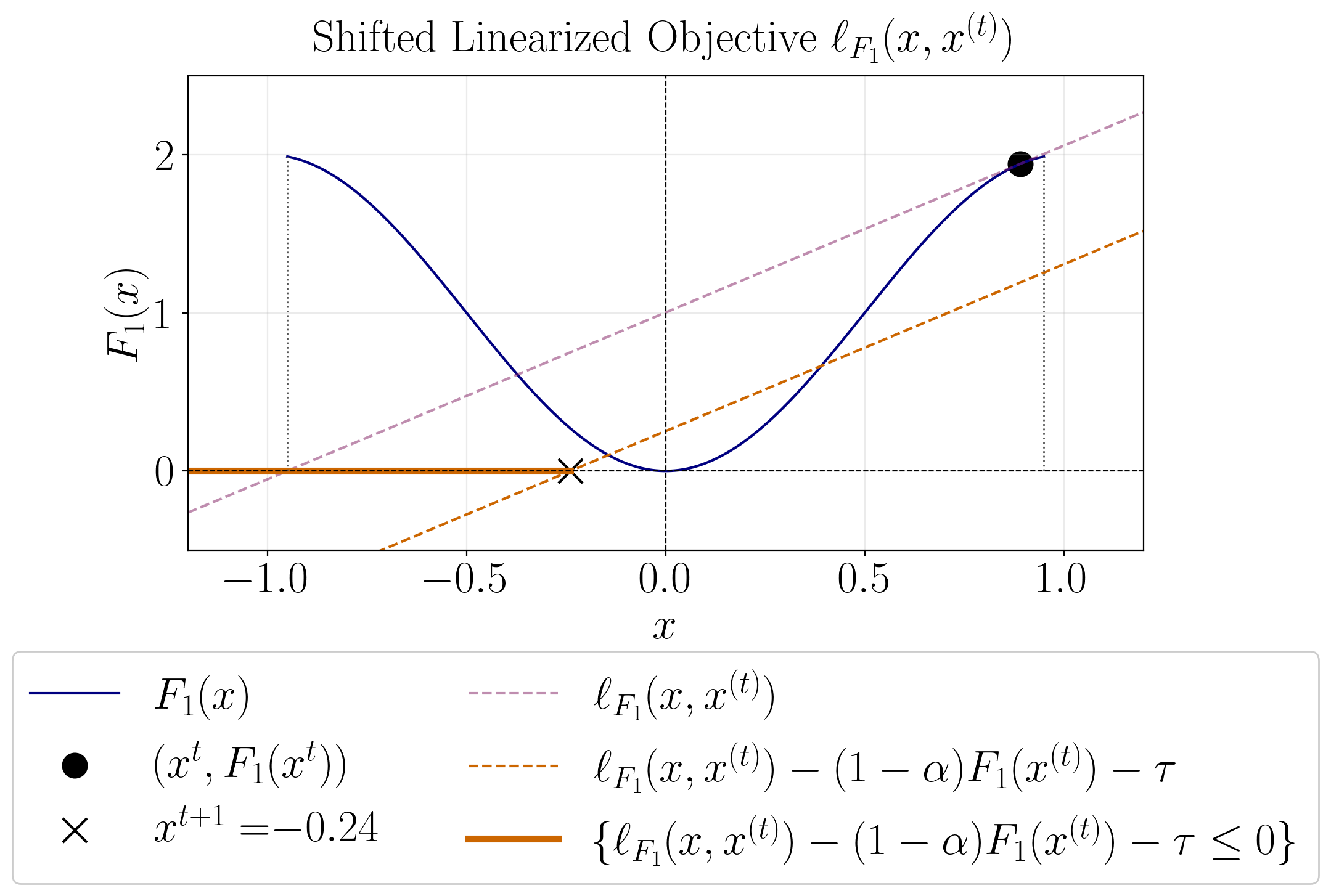
  }
  \caption[Shifted Linearized Objective Bundle-level Approach]{Illustration of the \textit{shift} of the linearized objective
   $\ell_{F_1}(x, x^{(t)})$. By introducing the shift, \eqref{eq:modified_Polyak_Minorant_subproblem} makes a more careful step than \eqref{eq:Polyak_Minorant} and improves the next iterate to $x^{(t+1)} = -0.24.$}
  \label{fig:polyak_minorant_linearized_objective_shifted}
 \end{subfigure}
 \caption{(a) and (c): To illustrate the need for the
  \textit{shift},
  we use the hidden convex
  function $F_1(\xx) \define 1 - \cos(\pi \cdot \xx)$ without
  constraints. (b): We illustrate the \textit{shifted constraint}
  on the constrained geometric programming
  example
  $F_1(\xx) \define \xx_1 \cdot \xx_2 + \frac{4}{\xx_1} + \frac{1}{\xx_2}$
  constrained to $\{F_2 \leq 0\}$, where
  $F_2(\xx) \define \xx_1 \cdot \xx_2 - 1$, cf.
  \eqref{eq:ToyExampleGeometricProgramming} in \Cref{subsec:ExampleNonsmoothConstrainedLeastSquares}.
 }
 \label{fig:PolyakMinorantMethod}
\end{figure}

\begin{example}
 Consider minimizing $F_1(x) =
  1 - \cos(\pi x) $ with $F_2(\innerEmpty{})\equiv 0$ over the set $\XX = [-0.95, 0.95].$ This problems is
 hidden convex with $c(x) = \sin(\tfrac{\pi}{2}\, x ),$ $H(u) = 2\, \uu^2$, $\mu_c \geq \tfrac{\pi}{2}\cos(\tfrac{\pi}{2} \cdot 0.95) > 0.$ However,
 starting from a point $\xxT \geq 0.891$ the method jumps to $\xxTnext =
  -0.95$. In subsequent iterations the algorithm infinitely oscillates between
 $\xx = \pm 0.95$, which correspond to the global maxima of this
 problem, see \Cref{fig:polyak_minorant_linearized_objective} for an
 illustration of this failure example.
\end{example}

The failures as in the above example are common and occur due to the
\textit{negative curvature} of the objective, which makes the linear
approximation an invalid global lower model of $F_1,$ see
\Cref{fig:polyak_minorant_linearized_constraint}. This motivates our key
algorithmic modification involving a carefully chosen shift of the
linearized objective and constraints. The introduced shifts relax the
linearized constraints allowing the bundle-level method to approach the
optimum as illustrated in
\Cref{fig:polyak_minorant_linearized_objective_shifted}.

Before we proceed with the formal algorithm description, we introduce some
useful notation for this section. We define the \emph{optimal value
 function} as $$ V(\eta) \define \min_{\xx \in \XX} v(\xx, \eta) , \qquad
 v(\xx, \eta) \define \max\left\{ F_1(\xx) - \eta, F_2(\xx) \right\} . $$
Notice that when the argument $\eta=\Fopt=\min_{\xx\in \XX} \{F_1(x), \,\,
 \text{s.t. }\, F_2(x)\leq 0\}$, then $V(\Fopt) = 0.$ Define a linear
minorant of a differentiable function $F: \XX \rightarrow \RR$ as $$
 \ell_F(\xx, \yy) = F(\yy) + \langle\nabla F(\yy), \xx - \yy \rangle . $$
Now we are ready to introduce the proposed algorithm. First, we consider
the case when the optimal value $\Fopt$ is known.\footnote{In practice, the
 optimal value can be obtained by solving the Lagrange dual problem if
 strong duality holds provided that the dual problem is easy to solve.} Let
$\alpha \in [0,1]$, $\tau > 0$ are some parameters, then our \emph{Shifted
 Star Bundle-level} (S-StarBL) algorithm has the update rule:
\begin{align}\label{eq:modified_Polyak_Minorant_subproblem}
 \xxTnext & = \argmin_{\xx \in \XX} \, \norm{\xx - \xxT}^2, \quad \tag{S-StarBL}                                   \\
 \text{s.t. } \;\;
          & \ell_{F_1}(\xx, \xxT) \leq \Fopt + {\highlight{ \rb{1-\alpha} (F_1(\xxT) - \Fopt) +  \tau }}  , \notag \\
          & \ell_{F_2}(\xx, \xxT) \leq {\highlight{ (1-\alpha) F_2(\xxT) + \tau }} \notag .
\end{align}
This algorithm shifts the feasible set of Bundle-level subproblem allowing us to search for $\xxTnext$ in the larger set, see
\Cref{fig:polyak_minorant_linearized_constraint,fig:polyak_minorant_linearized_objective_shifted} for illustrations of the update rule.

Next, we consider the setting when $\Fopt$ is unknown, which is arguably
much more challenging. In this setting, we design a double-loop procedure,
which dynamically searches for the optimal value $\Fopt$ using the exact
penalty, $F_1(\xxT) + \lambda \plus{F_2(\xxT)},$ as a convergence
criterion. The method is described in
\Cref{algo:PolyakMinorant_LB,algo:PolyakMinorant_LB_outer_line_search}. The
outer loop \Cref{algo:PolyakMinorant_LB_outer_line_search} repeatedly calls
the Shifted Bundle-level \Cref{algo:PolyakMinorant_LB} and updates the
current lower bound estimate $\eta_{k}$ of the optimal value $\Fopt.$

\begin{remark}\label{rem:lower_bound_Fopt}
 In \Cref{algo:PolyakMinorant_LB_outer_line_search}, we assume that the input lower bound value satisfies $\eta_0 \leq \Fopt$.
 This is not limiting as we can use the following simple initialization protocol. We first run the (projected) gradient descent
 on the unconstrained problem, $\min_{\xx \in \XX} F_1(\xx)$, for $N_{\text{init}} = \OOTilde{\frac{ (\rho + L) \DDUsq}{\mu_c^2 \varepsilon}}$
 iterations to find a point $\zz^{(N_{\text{init}})}$ with $F_1(\zz^{(N_{\text{init}})}) - \varepsilon \leq \min_{\xx \in \XX} F_1(\xx) \leq \Fopt $.
 If we have $F_2(\zz^{(N_{\text{init}})}) \leq \varepsilon$, we return the point $\zz^{(N_{\text{init}})}$ as the solution,
 otherwise we initialize $\eta_0 = F_1(\zz^{(N_{\text{init}})}) - \varepsilon$, which is a valid lower bound for $\Fopt.$
\end{remark}

\begin{remark}
 We assume that in \Cref{algo:PolyakMinorant_LB_outer_line_search} the input estimate of the Lagrange multiplier $\lambda$ is lower
 bounded by the optimal multiplier $\lambda^*.$ Such upper bound of $\lambda^*$ can be obtained if
 Slater's gap is known, e.g. using $\lambda^* \leq (F_1(y) - \Fopt) / \theta$ \cite[Lem. 3.1.21]{nesterovLecturesConvexOptimization2018},
 where $y$ is a $\theta$--Slater point. However, oftentimes, although Slater's condition fails, the strong duality can be verified
 with an estimate of $\lambda^*$, see, e.g., the controller synthesis example with equality constraints in \Cref{subsec:motivating_applications}.
\end{remark}

\begin{algorithm}
 \caption{$\text{S-BL}(\xxZero, \eta, T, \tau, \alpha, \beta, \lambda)$ \\
  Shifted Bundle-level for \eqref{eq:MainProblem}}
 \begin{algorithmic}[1]
  \State \textbf{Input:} Initial point $\xxZero \in \mathcal{X}$, lower bound of the optimal value $\eta \leq \Fopt$ of \eqref{eq:MainProblem}, number of iterations $T,$ maximum minorant violation budget $\tau$, contraction factors $\alpha$ and $\beta$,
  estimate of Lagrange multiplier $\lambda$
  \For {iteration $t = 0, 1, \dots, T-1$}
  \State Solve
  \begin{align}\label{eq:modified_Polyak_Minorant_approx_subproblem}
   \xxTnext & = \argmin_{\xx \in \XX} \, \norm{\xx - \xxT}^2,
   \tag{SBL-QP}                                                                                                                                                    \\
   \text{s.t. } \;\;
            & \ell_{F_1}(\xx, \xxT) \leq \rb{1-\alpha\, \beta} F_1(\xxT) + \alpha \, \beta \, \eta + (1-\beta) \alpha  \lambda \plus{F_2(\xxT)}   +  \tau , \notag \\
            & \ell_{F_2}(\xx, \xxT) \leq (1-\alpha) F_2(\xxT) + \tau \notag
  \end{align}
  \EndFor
  \State \textbf{Return:} $\xx^{(t^*)}$, where $ t^* = \argmin_{t \leq T} F_1(\xxT) + \lambda \plus{F_2(\xxT)}$
 \end{algorithmic}
 \label{algo:PolyakMinorant_LB}
\end{algorithm}

\begin{algorithm}[H]
 \caption{$\text{Ada-LS}(\xxZero, \eta_0, N, \tau, \alpha, \beta, \lambda)$\\
  Adaptive Line Search for Shifted Bundle-level (S-BL) Method}
 \begin{algorithmic}[1]
  \State \textbf{Input:} Initial point $\xxZero \in \mathcal{X}$, lower bound of the optimal value $\eta_0 \leq \Fopt$ of \eqref{eq:MainProblem}, number of epochs $T,$ maximum minorant violation budget $\tau$, contraction factors $\alpha$ and $\beta$, estimate of Lagrange multiplier $\lambda$
  \For {iteration $k = 1, \dots, N$}
  \State
  \begin{eqnarray*}
   \xxBarK &=& \text{S-BL}(\xxZero, \eta_k, T, \tau, \alpha, \beta, \lambda) \\
   \eta_{k+1} &=& (1-\beta) \rb{\eta_k + F_1(\xxBarK) + \lambda \plus{F_2(\xxBarK) } }  - (1-2\beta) \eta_k
  \end{eqnarray*}
  \EndFor
  \State \textbf{Return:} $\bar{\xx}^{(k^*)}$, where $ k^* = \argmin_{k \leq N} F_1(\xxBarK) + \lambda \plus{F_2(\xxBarK) } $
 \end{algorithmic}
 \label{algo:PolyakMinorant_LB_outer_line_search}
\end{algorithm}

The following lemma is central to establish convergence of both above
proposed algorithms. Since \eqref{eq:modified_Polyak_Minorant_subproblem}
is a special case of \Cref{algo:PolyakMinorant_LB} with $\eta = \Fopt$,
$\beta = 1$ and arbitrary $\lambda \in \RR$, the following unified lemma
applies to both schemes.

\begin{lemma}\label{lemma:polyak_minorant_approx_feasibility}
 Let \Assref{ass:AssumptionExactPPM} hold and let $\xxOpt \in \XX$ be an optimal solution to the hidden convex problem \eqref{eq:MainProblem}. For any $t \geq 1$ and $\alpha \in [0,1],\, $ define the point $\xxT_{\alpha} \define c^{-1}( (1-\alpha) c(\xxT) + \alpha c(\xxOpt) ).$ Then we can distinguish between two cases:
 $$
  \texttt{Good case:}\qquad F_1^* - (1-\beta) \rb{\eta_k + F_1(\xxT) + \lambda \plus{F_2(\xxT)} } + (1-2\beta) \eta_k  \leq 0,   $$
 for all $t \in \setT$.
 By setting $\tau \geq \frac{\rho \alpha^2 \DDUsq}{2\mu_c^2}, $ the point
 $\xxK_{\alpha}$ is feasible for subproblem \eqref{eq:modified_Polyak_Minorant_approx_subproblem}.
 Moreover, for any $t=0, \ldots, T-1$:
 \begin{equation*}
  \norm{\xxTnext - \xxT} \leq \norm{\xxT_{\alpha} - \xxT} \leq \frac{\alpha \DDU}{\mu_c}.
 \end{equation*}
 Alternatively, there exists $\bar t \in \setT$ such that
 $$
  \texttt{Bad case: } \qquad F_1^* - (1-\beta) \rb{\eta_k + F_1(\xx^{(\bar{t})}) + \lambda \plus{F_2(\xx^{(\bar t)})} } + (1-2\beta) \eta_k  > 0 .
 $$
 Then we have $\eta_{k+1} < \Fopt$. If, additionally, strong duality holds with a Lagrange multiplier $\lambda^*$, then
 $$
  \Fopt - \eta_{k+1} \leq \beta (\Fopt - \eta_{k}) + (1-\beta) (\lambda^* - \lambda) \plus{F_2(\xxBarK)}  .
 $$
\end{lemma}

\begin{proof}
 Assume we are in the ``\texttt{Good case}'', then using weak convexity and hidden convexity of $F_1$, we have
 \begin{eqnarray*}
  \ell_{F_1}(\xxT_{\alpha}, \xxT)
  &=& F_1(\xxT) + \langle\nabla F_1(\xxT), \xxT_{\alpha} - \xxT \rangle  \\
  &\leq & \frac{\rho}{2}\norm{\xxT_{\alpha} - \xxT}^2 + F_1(\xxT_{\alpha}) \\
  &\leq & \frac{\rho}{2}\norm{\xxT_{\alpha} - \xxT}^2 + (1-\alpha ) F_1(\xxT) + \alpha \Fopt \\
  &\leq & \frac{\rho \alpha^2 \DDUsq}{2\mu_c^2} + (1-\alpha) F_1(\xxT) + \alpha \Fopt \\
  &\leq& \tau + (1-\alpha) F_1(\xxT)  + (1-\beta) \alpha \rb{\eta_k + F_1(\xxT) + \lambda \plus{F_2(\xxT)} } \\
  && \qquad - (1-2\beta) \alpha  \eta_k \\
  & = & \tau + (1-\alpha \, \beta) F_1(\xxT) + \alpha \beta  \eta_k + (1-\beta) \alpha  \lambda \plus{F_2(\xxT)} ,
 \end{eqnarray*}
 where the second and third inequalities follow from \Cref{lemma:HCContractionInequality}, and the last two are due to our choice of $\tau,$ and the fact that we are in the ``\texttt{Good case}''.

 Similarly for $F_2$, we have
 \begin{eqnarray*}
  \ell_{F_2}(\xxT_{\alpha}, \xxT)
  &=& F_2(\xxT) + \langle\nabla F_2(\xxT), \xxT_{\alpha} - \xxT \rangle  \\
  &\leq & \frac{\rho}{2}\norm{\xxT_{\alpha} - \xxT}^2 + F_2(\xxT_{\alpha}) \\
  &\leq & \frac{\rho}{2}\norm{\xxT_{\alpha} - \xxT}^2 + (1-\alpha ) F_2(\xxT) + \alpha F_2(\xxOpt) \\
  & \leq & \tau + (1-\alpha ) F_2(\xxT) ,
 \end{eqnarray*}
 where in the last step we used the assumption on $\tau$ and the feasibility of $\xxOpt.$ The above two inequalities imply that $\xxT_{\alpha}$ is feasible for sub-problem \eqref{eq:modified_Polyak_Minorant_approx_subproblem}. Using the optimality of $\xxTnext$, and above established feasibility of $\xxT_{\alpha}$, we obtain
 $$
  \norm{\xxTnext - \xxT} \leq \norm{\xxT_{\alpha} - \xxT} \leq \frac{ \alpha \DDU}{\mu_c} .
 $$

 Now assume we are in the ``\texttt{Bad case}'', then by construction of
 $\eta_{t+1}$ we have
 \begin{eqnarray*}
  \eta_{k+1} &=& (1-\beta) \rb{\eta_k + F_1(\xxBarK) + \lambda \plus{F_2(\xxBarK)}  }  - (1-2\beta) \eta_k \\
  &\leq& (1-\beta) \rb{\eta_k + F_1(\xx^{(\bar t)}) +   \lambda \plus{F_2(\xx^{(\bar t)})}  }  - (1-2\beta) \eta_k < \Fopt ,
 \end{eqnarray*}
 where the inequality holds since $\xxBarK$ satisfies $F_1(\xxBarK) + \lambda \plus{F_2(\xxBarK)} \leq F_1(\xx^{(\bar t)}) + \lambda \plus{F_2(\xx^{(\bar t)})}$ due to the output criterion of \Cref{algo:PolyakMinorant_LB}. By strong duality, we have for any $x \in \XX$, that $F_1(x) \geq \Fopt - \lambda^* \plus{F_2(x)}.$ Therefore,
 \begin{eqnarray*}
  \Fopt - \eta_{k+1} &=& \Fopt - (1-\beta) \rb{\eta_k + F_1(\xxBarK) + \lambda \plus{F_2(\xxBarK) }} + (1-2\beta) \eta_k \\
  &\leq& \Fopt - (1-\beta) \rb{\eta_k + \Fopt + (\lambda - \lambda^*) \plus{F_2(\xxBarK)} }  + (1-2\beta) \eta_k \\
  &=& \beta (\Fopt - \eta_k) + (1-\beta) (\lambda^* - \lambda) \plus{F_2(\xxBarK)}
 \end{eqnarray*}
 concluding the proof.
\end{proof}

Now we are ready to formulate the main results of this section, i.e. the
convergence of
\Cref{algo:PolyakMinorant_LB,algo:PolyakMinorant_LB_outer_line_search}. We
distinguish between two most interesting cases: when we know the optimal
value $\Fopt$ and when we do not know it.

\subsection{Convergence with Known $\Fopt$}
\label{subsec:PMMknownFstarNoCQ}
In this case we can set
$\eta = \Fopt$, $\beta
 = 1$, arbitrary
$\lambda \in \RR$, and
we do not need the
outer loop
``line-search''
procedure
\Cref{algo:PolyakMinorant_LB_outer_line_search}, since our algorithm simplifies to \eqref{eq:modified_Polyak_Minorant_subproblem}.
Under this choice of
parameters, we always
fall in the
``\texttt{Good case}''
of
\Cref{lemma:polyak_minorant_approx_feasibility},
and we have the
following convergence
result.

% \begin{tcolorbox}[colback=white,colframe=blueETH!75!black]
\begin{theorem}[S-StarBL, Known $\Fopt$, No CQ]
 \label{thm:MainResultHCPolyakMinorant}
 Assume that \eqref{eq:MainProblem} is hidden convex, \Assref{ass:AssumptionExactPPM} holds and that
 $F_1, F_2$ are $L$-smooth. Set $\tau = \frac{\rho \alpha^2 \DDUsq}{2\mu_c^2},$ $\alpha = \frac{\varepsilon \mu_c^2}{(\rho + L) \DDUsq},$ then
 the last iterate of \eqref{eq:modified_Polyak_Minorant_subproblem} satisfies
 $$
  F_1(\xx^{(T)}) - \Fopt \leq \epsilon,
  \, F_2(\xx^{(T)}) \leq \epsilon
 $$
 after
 \begin{align*}
  \Ttotal \geq \frac{(\rho+L) \DDUsq}{\mu_c^2 \, \epsilon}
  \cdot \log\left(\frac{2 v(\xxZero, \Fopt) }{\epsilon}\right) .
 \end{align*}
\end{theorem}
% \end{tcolorbox}

\begin{proof}
 The choice of $\eta$, $\beta$ and $\lambda$ implies that for any $t \geq 0$
 $$
  \rb{1-\alpha\, \beta} F_1(\xxT) + \alpha \, \beta \, \eta + (1-\beta) \alpha  \lambda \plus{F_2(\xxT)}   +  \tau = \rb{1-\alpha} F_1(\xxT) + \alpha \, \Fopt  +  \tau.
 $$
 Therefore, by smoothness of $F_1$ we obtain
 \begin{eqnarray*}
  F_1(\xxTnext) &\leq& F_1(\xxT) + \langle \nabla F_1(\xxT), \xxTnext - \xxT\rangle + \frac{L}{2} \norm{\xxTnext - \xxT}^2 \\
  &=& \ell_{F_1}(\xxTnext, \xxT) + \frac{L}{2} \norm{\xxTnext - \xxT}^2 \\
  &\leq&  \Fopt + (1-\alpha) (F_1(\xxT) - \Fopt) + \tau +  \frac{L}{2} \norm{\xxTnext - \xxT}^2 ,
 \end{eqnarray*}
 where the last inequality follows from the update rule of \Cref{algo:PolyakMinorant_LB} due to feasibility of $\xxTnext$ for \eqref{eq:modified_Polyak_Minorant_approx_subproblem}. Subtracting $\Fopt$ from both sides and noticing that the choice of $\eta,$ $\beta$ and $\lambda$ imply that we are in the ``\texttt{Good case}'' of \Cref{lemma:polyak_minorant_approx_feasibility}, we obtain for any $\tau \geq \frac{\rho \alpha^2 \DDUsq}{2\mu_c^2}$
 the recursion
 \begin{eqnarray*}
  F_1(\xxTnext) - \Fopt
  &\leq&  (1-\alpha) (F_1(\xxT) - \Fopt) + \tau +  \frac{L \alpha^2 \DDUsq}{2\mu_c^2} .
 \end{eqnarray*}
 An analogous derivation for $F_2$ results in
 \begin{eqnarray*}
  F_2(\xxTnext)
  &\leq&  (1-\alpha) F_2(\xxT) + \tau +  \frac{L \alpha^2 \DDUsq}{2\mu_c^2} .
 \end{eqnarray*}
 Combining the above two inequalities, we can establish a recursion for the value function
 \begin{eqnarray*}
  v(\xxTnext, \Fopt) & = & \max\left\{ F_1(\xxTnext) - \Fopt, F_2(\xxTnext) \right\} \\
  &\leq& (1-\alpha)  \max\left\{ F_1(\xxT) - \Fopt, F_2(\xxT) \right\} + \tau +  \frac{L \alpha^2 \DDUsq}{2\mu_c^2} \\
  &=& (1-\alpha) v(\xxT, \Fopt) + \frac{(\rho + L) \alpha^2 \DDUsq}{2\mu_c^2} ,
 \end{eqnarray*}
 where in the last step we used the definition of the value function and set $\tau = \frac{\rho \alpha^2 \DDUsq}{2\mu_c^2}.$
 Then by unrolling the recursion for $t = 0, \ldots, T-1$, we have for any
 $\alpha \in [0,1]$
 \begin{align*}
  v(\xx^{(T)}, \Fopt) \leq (1-\alpha)^T v(\xxZero, \Fopt) + \frac{(\rho + L) \alpha \DDUsq}{2\mu_c^2} .
 \end{align*}
 Setting $\alpha = \frac{\varepsilon \mu_c^2}{(\rho + L) \DDUsq},$ we obtain the desired result.
\end{proof}

\subsection{Convergence with Unknown $F_1^*$ under Strong Duality}
\label{subsec:PMMunknownFstarStrongDuality}
In this subsection,
we will deal with
the situation when
the exact value of
$\Fopt$ is unknown.
Assume for
simplicity that we
have a valid lower
bound for the
optimal value
$\Fopt$ to
initialize our
\Cref{algo:PolyakMinorant_LB_outer_line_search},
i.e., $\eta_0 \leq
 \Fopt$. Such lower
bound can be straightforwardly
obtained as we
explain in
\Cref{rem:lower_bound_Fopt}.
We also assume that
the initial point
$\xxZero$ is nearly
feasible, i.e.,
$\plus{F_2(\xxZero)}
 \leq \varepsilon /
 (2\lambda)$. Such
point can be easily
found by solving
$\min_{\xx \in \XX}
 F_2(\xx)$ to this
accuracy in
$\OOTilde{\frac{\lambda
   (\rho + L)
   \DDUsq}{\mu_c^2
   \varepsilon}}$
iterations of
projected gradient
descent applied to
$F_2(\innerEmpty{})$.
Now we are ready to
prove the main
convergence result
of
\Cref{algo:PolyakMinorant_LB_outer_line_search}
which calls
\Cref{algo:PolyakMinorant_LB}
at each iteration.
% \begin{tcolorbox}[colback=white,colframe=blueETH!75!black]
\begin{theorem}[S-BL+AdaLS, Unknown $\Fopt$, Strong Duality]
 \label{thm:MainResultHCPolyakMinorantUnknownOptimalValue}
 Assume that \eqref{eq:MainProblem} is hidden convex, \Assref{ass:AssumptionExactPPM} holds, and that
 $F_1, F_2$ are $L$-smooth. Let the strong duality hold for \eqref{eq:MainProblem} and the optimal Lagrange muliplier is at most $\lambda^*$. Set $\eta_0 \leq \Fopt$, $\beta = \nicefrac{1}{2}$, $\lambda \geq \lambda^*$, and $\plus{F_2(\xxZero)} \leq \nfr{\varepsilon}{(2\lambda)}$. Then \Cref{algo:PolyakMinorant_LB_outer_line_search} has the oracle complexity
 $$
  T \cdot N =  16 (1+\lambda)  \frac{(\rho + L) \DDUsq }{\mu_c^2\epsilon} \log^2\rb{\frac{8(F_1(\xxZero) - \Fopt)}{\varepsilon}} \cdot
  \log\rb{\frac{\Fopt - \eta_0}{\varepsilon}},
 $$
 to find a point $\bar \xx^{(k^*)} \in \XX$ with
 $
  F_1(\bar{\xx}^{(k^*)}) - \Fopt + \lambda \plus{F_2(\bar{\xx}^{(k^*)})} \leq \varepsilon.
 $
\end{theorem}
% \end{tcolorbox}

\begin{proof}
 Assume, for the sake of contradiction, that for the first $k \leq N$ iterations $F_1(\xxBarK) - \Fopt + \lambda \plus{F_2(\xxBarK)} > \varepsilon$. Our proof strategy is to show that we always fall into the ``\texttt{Good case}'' of \Cref{lemma:polyak_minorant_approx_feasibility} at least once before $k \leq N.$

 First, if for some $k \leq N$, the lower bound $\eta_k$ is a good
 approximation of the optimal value $\eta_k$, i.e., $\Fopt - \eta_k \leq
  \varepsilon$, then the choice $\beta = \nicefrac{1}{2}$ implies that
 \begin{eqnarray*}
  (*) \define \Fopt &-& (1-\beta) \rb{\eta_k + F_1(\xxT) + \lambda \plus{F_2(\xxT)} } + (1-2\beta) \eta_k \\
  &=& \Fopt - \frac{1}{2} \rb{\eta_k + F_1(\xxT) + \lambda \plus{F_2(\xxT)} }  \\
  &\leq&  - \frac{1}{2} \rb{ F_1(\xxBarK) - \Fopt + \lambda \plus{F_2(\xxBarK)} - \varepsilon  }  \leq 0 ,
 \end{eqnarray*}
 for all $t$ in iteration $k$ and we automatically fall into the ``\texttt{Good case}'' of \Cref{lemma:polyak_minorant_approx_feasibility}.

 Otherwise, if $\Fopt - \eta_k > \varepsilon$, we are going to estimate the
 number of outer loop iterations $N$ (of
 \Cref{algo:PolyakMinorant_LB_outer_line_search}) to achieve $\Fopt - \eta_N
  \leq \varepsilon$. Since $\lambda \geq \lambda^*$, the ``\texttt{Bad
  case}'' of \Cref{lemma:polyak_minorant_approx_feasibility} implies $\Fopt -
  \eta_{N} \leq (\Fopt - \eta_0) / 2^N$. Thus, after at most $N \geq
  \log_2\rb{\frac{\Fopt - \eta_0}{\varepsilon}}$ iterations, we have $\Fopt -
  \eta_{N} \leq \varepsilon$, and we end up in the ``\texttt{Good case}'' of
 \Cref{lemma:polyak_minorant_approx_feasibility}.

 Therefore, there exists at least one $k = k^* \leq N$ when we fall into the
 ``\texttt{Good case}''. By \Cref{lemma:polyak_minorant_approx_feasibility},
 $\xx_\alpha^{(t)}$ is feasible for each subproblem
 \eqref{eq:modified_Polyak_Minorant_approx_subproblem}, and $\norm{\xxTnext
   - \xxT} \leq \norm{\xxT_{\alpha} - \xxT} \leq \frac{\alpha \DDU}{\mu_c}$
 for all $t = 0, \ldots, T-1$. By $L$-smoothness, following similar steps as
 in the proof of \Cref{thm:MainResultHCPolyakMinorant}, we obtain
 \begin{eqnarray*}
  F_1(\xxTnext)
  &\leq&  \rb{1-\frac{\alpha}{2}} F_1(\xxT) + \frac{\alpha \, \eta_{k^{*}}}{2} + \frac{\alpha \, \lambda}{2} \plus{F_2(\xxT)}  + \tau + \frac{ L \alpha^2 \DDUsq}{2\mu_c^2},  \\
  F_2(\xxTnext) &\leq& (1-\alpha) F_2(\xxT) + \tau + \frac{ L \alpha^2 \DDUsq}{2\mu_c^2} .
 \end{eqnarray*}
 Defining $A \define \frac{(\rho + L) \DDUsq}{2\mu_c^2}$, using $\eta_{k^*} \leq \Fopt$ (guaranteed by \Cref{lemma:polyak_minorant_approx_feasibility}) and setting $\tau = \frac{\rho \alpha^2 \DDUsq}{2\mu_c^2}$
 \begin{eqnarray}
  F_1(\xxTnext) - \Fopt
  &\leq&  \rb{1-\frac{\alpha}{2}} \rb{F_1(\xxT) - \Fopt} + \frac{\alpha \, \lambda}{2} \plus{F_2(\xxT)}  + A \alpha^2,  \label{eq:Aprox_PMM_optimality}\\
  \plus{F_2(\xxTnext)} &\leq& (1-\alpha) \plus{F_2(\xxT)} +  A \alpha^2 ,  \label{eq:Aprox_PMM_feasibility}
 \end{eqnarray}
 where the $\plus{\innerEmpty{}}$ appears by considering the cases when $F_2(\xxTnext) \geq 0$ and $F_2(\xxTnext) < 0$. Summing up \eqref{eq:Aprox_PMM_feasibility}, we have
 $$
  \alpha \sum_{t=0}^T \plus{F_2(\xxT)} \leq \plus{F_2(\xxZero)} + A \alpha^2 T.
 $$
 Unrolling \eqref{eq:Aprox_PMM_optimality} and using the above average feasibility bound, we have
 \begin{eqnarray*}
  F_1(\xx^{(T)}) - \Fopt &\leq& \rb{1-\frac{\alpha}{2}}^T \rb{F_1(\xxZero) - \Fopt} +    \frac{\lambda \plus{F_2(\xxZero)}}{2} + \frac{\lambda A \alpha^2}{2} T + 2 A \alpha .
 \end{eqnarray*}
 Now it remains to set
 $$
  \alpha = \min\cb{\frac{\varepsilon}{16 A}, \frac{\varepsilon}{8 \lambda A} \frac{1}{\log\rb{\nicefrac{8(F_1(\xxZero) - \eta_0)}{\varepsilon}}}}, \quad T = \frac{2}{\alpha} \log\rb{\nicefrac{8(F_1(\xxZero) - \eta_0)}{\varepsilon}} ,
 $$
 which implies $F_1(\xx^{(T)}) - \Fopt \leq \varepsilon / 2$ and $\lambda \plus{F_2(\xx^{(T)})} \leq \varepsilon / 2$ and, therefore, their sum is at most $\varepsilon$. We arrive at the contradiction with the initial assumption $F_1(\xxBarK) - \Fopt + \lambda \plus{F_2(\xxBarK)} > \varepsilon,$ there must be some $k^* \leq N$ such that
 $$
  F_1(\bar{\xx}^{(k^*)}) - \Fopt + \lambda \plus{F_2(\bar{\xx}^{(k^*)})} \leq \varepsilon.
 $$
 This concludes the proof.
\end{proof}

If we select $\lambda > \lambda^*$, the above theorem implies by strong
duality that $ F_1(\bar{\xx}^{(k^*)}) - \Fopt \leq \varepsilon,$ $
 \plus{F_2(\bar{\xx}^{(k^*)})} \leq \frac{\varepsilon}{\lambda - \lambda^*}
 . $ In particular, if $\lambda^* > 1$ and $\lambda = 2 \lambda^*$, then
$\plus{F_2(\bar{\xx}^{(k^*)})} \leq \varepsilon.$

%% file: sections/07_numericalExamples.tex
\section{Numerical Simulations}
\label{chapter:Experiments}
This section presents numerical simulations that illustrate and verify the proposed algorithms. The examples are small or medium scale, and primarily serve to demonstrate the concepts and theoretical properties in a transparent manner.

\subsection{Visualization of Convergence and Optimization Trajectories (2D Case)}
\label{subsec:ExampleNonsmoothConstrainedLeastSquares}

\begin{figure}
 \centering
 \begin{subfigure}{0.49\textwidth}
  \centering
  \includegraphics[height=.75\linewidth, keepaspectratio]{
   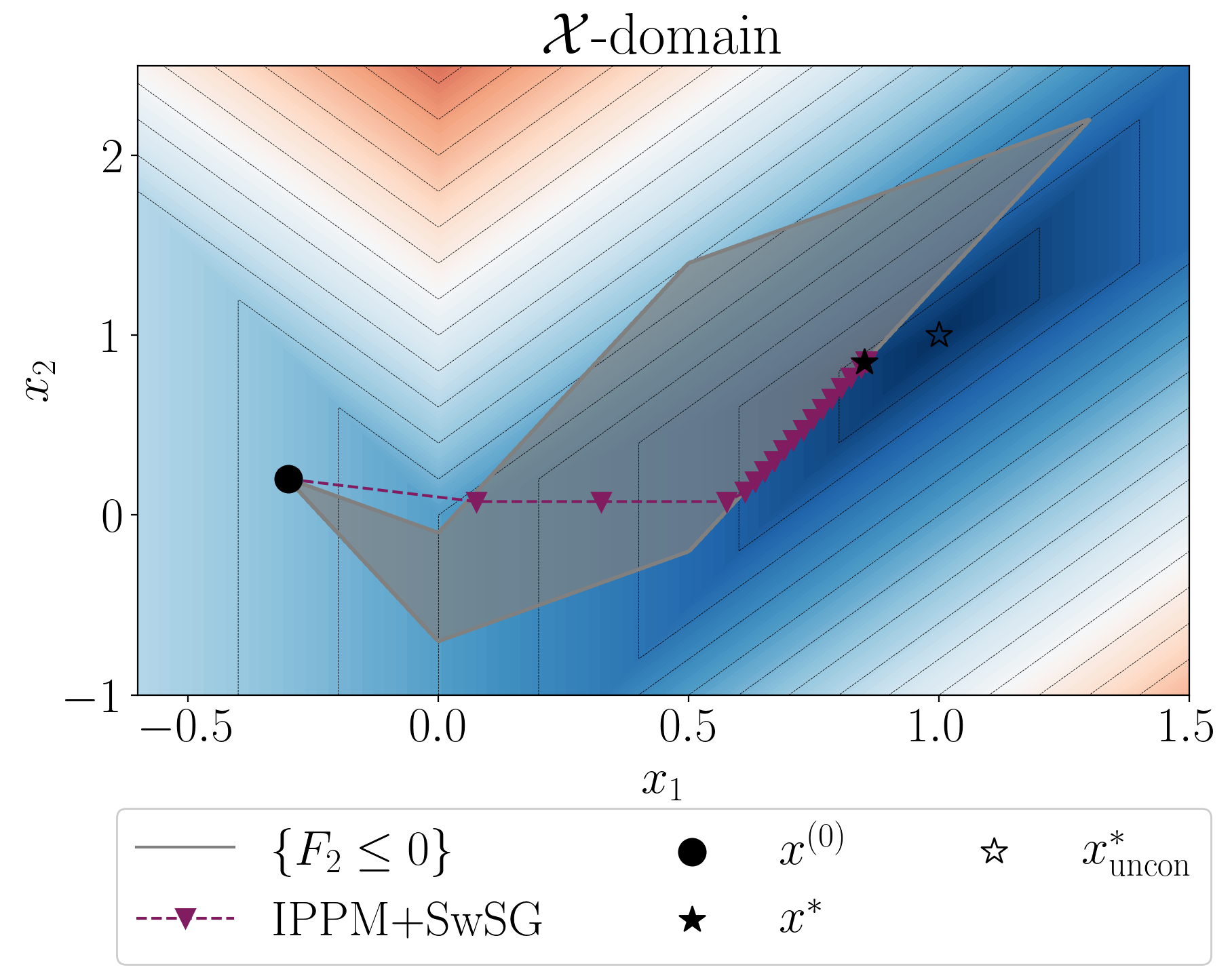
  }
  \caption{Iterates, $(\xxK)_k$, of IPPM+SwSG in the \textbf{non-convex} formulation.}
  \label{fig:NonsmoothLeastSquaresTracePlotX}
 \end{subfigure}
 \hfill
 \begin{subfigure}{0.49\textwidth}
  \centering
  \includegraphics[height=.75\linewidth, keepaspectratio]{
   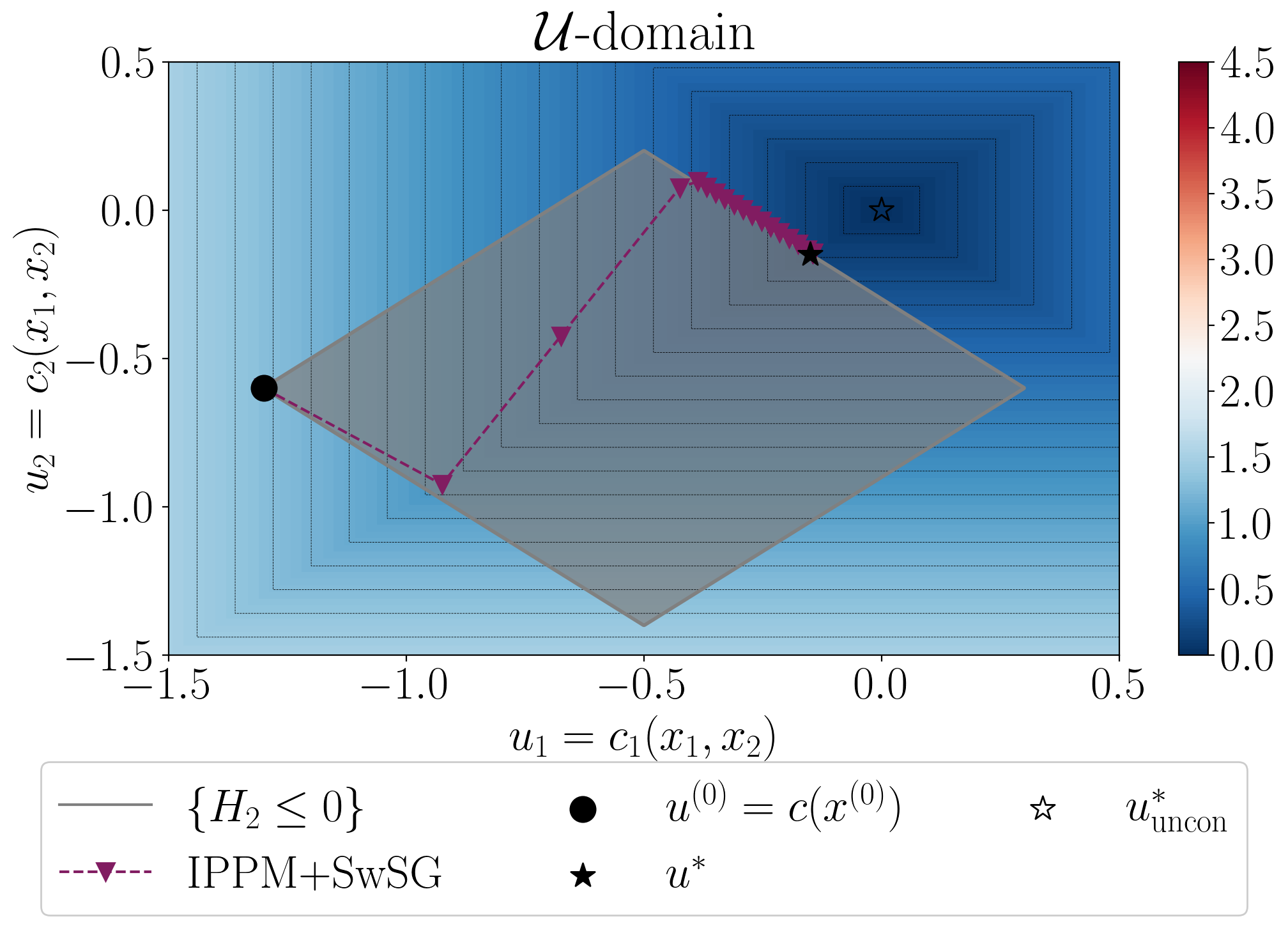
  }
  \caption{(Induced) iterates, $(c(\xxK))_k$, of IPPM+SwSG in the \textbf{convex} formulation.}
  \label{fig:NonsmoothLeastSquaresTracePlotU}
 \end{subfigure}
 \hfill
 \begin{subfigure}{0.99\textwidth}
  \centering
  \includegraphics[width=0.95\linewidth,keepaspectratio]{
   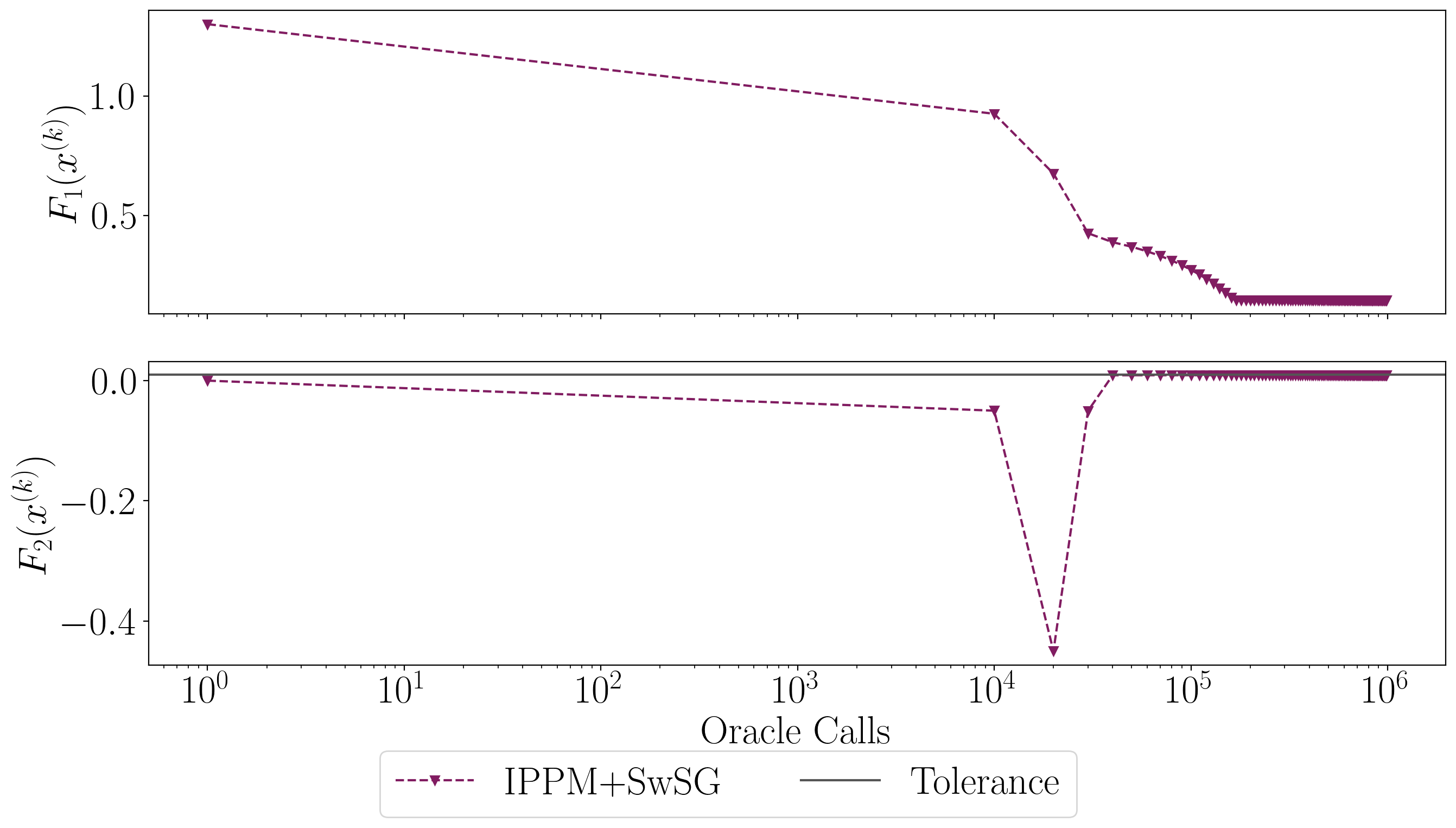
  }
  \caption{Objective value and constraint violation for IPPM+SwSG method (outer loop iterates, $(\xxK)_k$), oracle calls in log-scale.}
  \label{fig:NonsmoothLeastSquaresFuncValConstrViol}
 \end{subfigure}
 \caption[Non-Smooth Constrained Non-Linear Least Squares]{
  Solving the non-smooth \eqref{eq:ToyExampleLeastSquares} using
  the {Inexact Proximal Point Method with
    Switching Sub-Gradient (IPPM+SwSG)},
  cf. \Cref{subsubsec:IPPM_no_Slater}.
 }
 \label{fig:NonsmoothLeastSquaresComparison}
\end{figure}

\paragraph{Non-smooth Non-linear Least Squares.} First, we focus on the IPPM+SwSG algorithm and test it on a toy non-smooth
problem. In the two-dimensional case $\XX = \UU = \RR^2$, we use an
invertible map
\begin{align}
 \begin{aligned}
  c : \XX & \to \UU ,                                                                          \\
  \xx     & \mapsto (c_1(\xx), c_2(\xx))^\top \define (\xx_1 - 1, 2\abs{\xx_1}- x_2 -1)^\top ,
 \end{aligned}\;
 \label{eq:ToyExampleFunctionC}
\end{align}
which is non-smooth and has a Lipschitz inverse on $\XX=[-1, 2.5]^2$.
Our goal is to minimize the following non-smooth non-convex
{c}onstrained
{n}on-linear {l}east-{s}quares ({CNLS}) problem \cite{gurbuzbalabanNesterovsNonsmoothChebyshev2012,
 jarreNesterovsSmoothChebyshev2013} of the form:
\begin{align}
 \begin{aligned}
  \min_{\xx \in [-1,2.5]^2} F_1(\xx) & \define \norm{c(\xx) - b_1}_\infty ,       \\
  \text{s.t. } F_2(\xx)              & \define \norm{c(\xx) - b_2}_1 - 0.8 \leq 0
 \end{aligned}\;
 \tag{Ex-CNLS}
 \label{eq:ToyExampleLeastSquares}
\end{align}
with $b_1 \define (0,0)^\top, b_2 \define (-0.5, -0.6)^\top$. Note that
the problem above is non-smooth and non-convex (in variable $\xx \in \XX$),
in particular, we cannot use a projection onto a non-convex set $\cb{F_2 \leq 0}$, cf. feasible region in \Cref{fig:NonsmoothLeastSquaresTracePlotX}.

It can be verified that \eqref{eq:ToyExampleLeastSquares} is hidden convex
with $\mu_c = 4$ and satisfies \Assref{ass:AssumptionExactPPM} with $\rho =
 2$, and $G = 2$, e.g., using \cite[Thm.
 4.2]{drusvyatskiyEfficiencyMinimizingCompositions2019}, \cite[Cor.
 16.72]{bauschkeConvexAnalysisMonotone2017}. The solution to this problem is
$\xxOpt = (0.85, 0.85)^\top$ with $F_1(\xxOpt) = 0.15$, $\uuOpt = (-0.15,
 -0.15)^\top$, and $\lambdaOpt = 0.5$.\footnote{In this toy example, the
 optimizer can be obtained by solving the problem in the convex space $\UU$,
 projecting the global optima onto the feasible set, and mapping the point
 back to $\XX$.} We also verify that Slater's condition
(Definition~\ref{def:SlaterCondition}) holds, noting that $F_2((0.5,
 0.5)^\top) = -0.7 < 0$. The optimization trajectory of IPPM+SwSG is
presented in
\Cref{fig:NonsmoothLeastSquaresTracePlotX,fig:NonsmoothLeastSquaresTracePlotU}.
The proximal point iterates approach the constrained solution and remain in
the constrained (gray) region; the induced trajectory in the convex space
has a ``z'' shape due to non-linearity of the transformation
$c(\innerEmpty{}).$ Convergence and constraint violation plots in
\Cref{fig:NonsmoothLeastSquaresFuncValConstrViol} show that the proposed
algorithm successfully reduces the objective value and ensures the
constraint satisfaction.

\begin{figure}
 \centering
 \begin{subfigure}{0.49\textwidth}
  \centering
  \includegraphics[height=.75\linewidth, keepaspectratio]{
   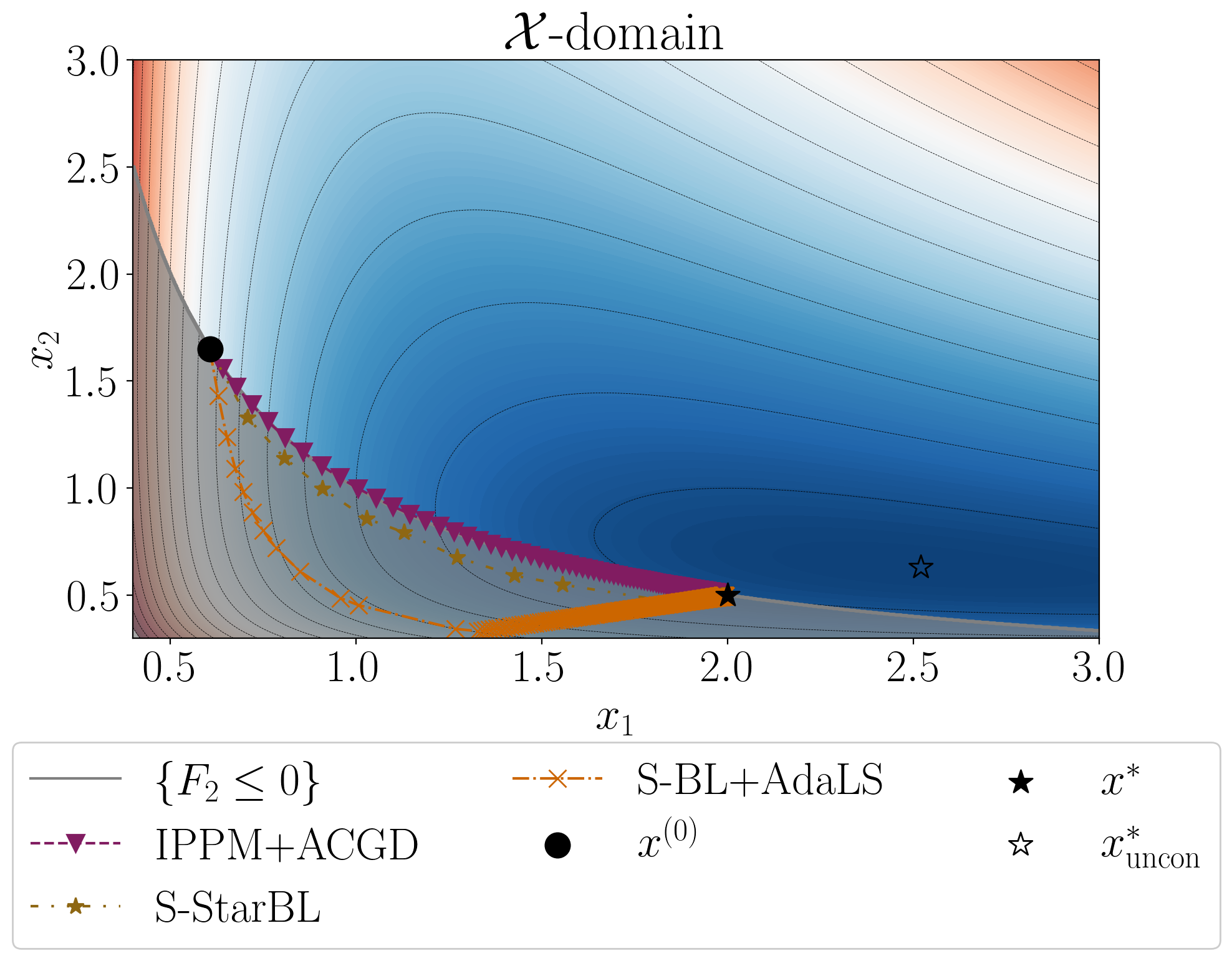
  }
  \caption{Iterates in the \textbf{non-convex} formulation.}
  \label{fig:SmoothGeometricProgrammingTracePlotX}
 \end{subfigure}
 \hfill
 \begin{subfigure}{0.49\textwidth}
  \centering
  \includegraphics[height=.75\linewidth, keepaspectratio]{
   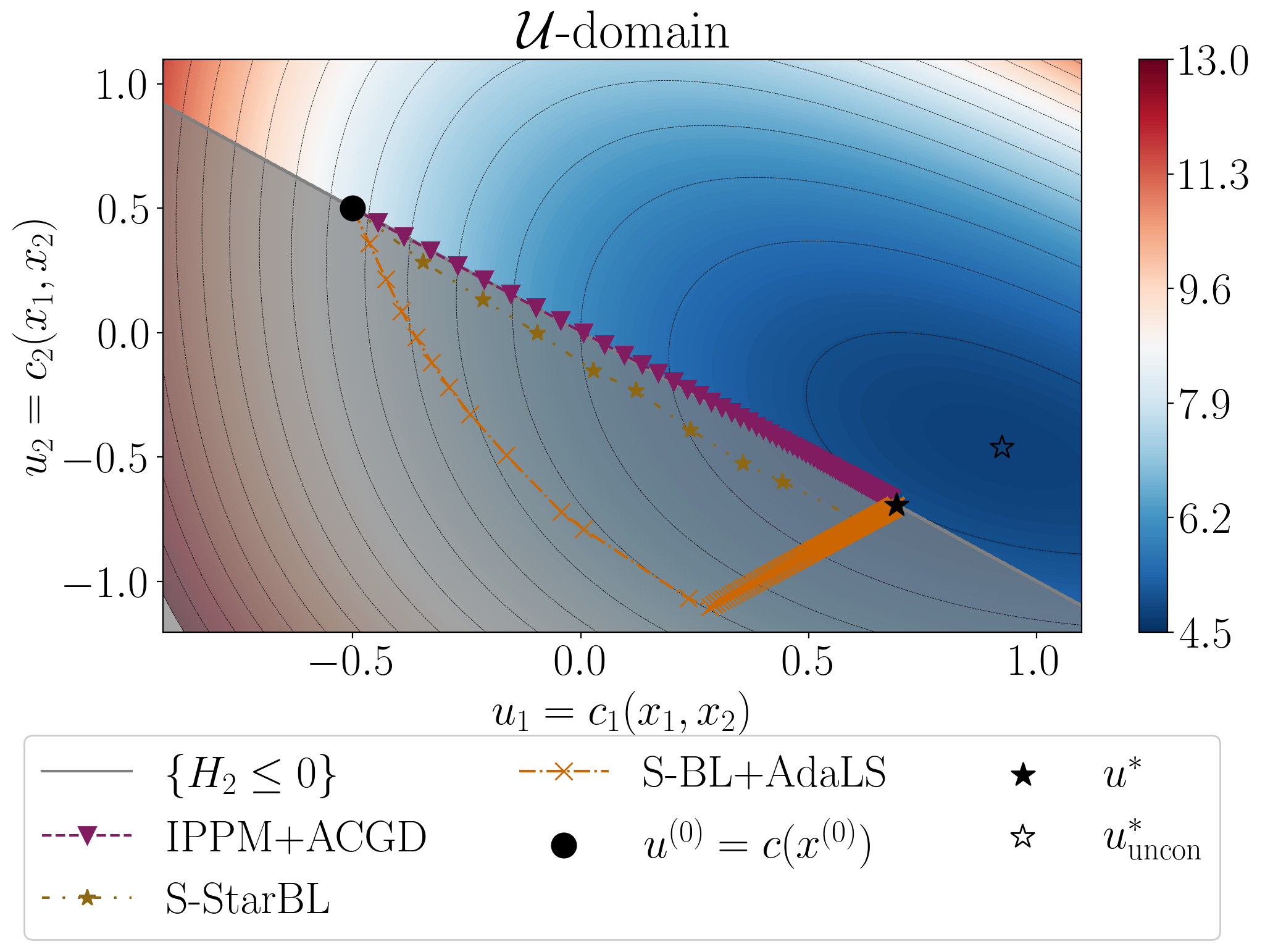
  }
  \caption{Iterates in the \textbf{convex} reformulation.}
  \label{fig:SmoothGeometricProgrammingTracePlotU}
 \end{subfigure}
 \hfill
 \begin{subfigure}{0.99\textwidth}
  \centering
  \includegraphics[width=0.95\linewidth,keepaspectratio]{
   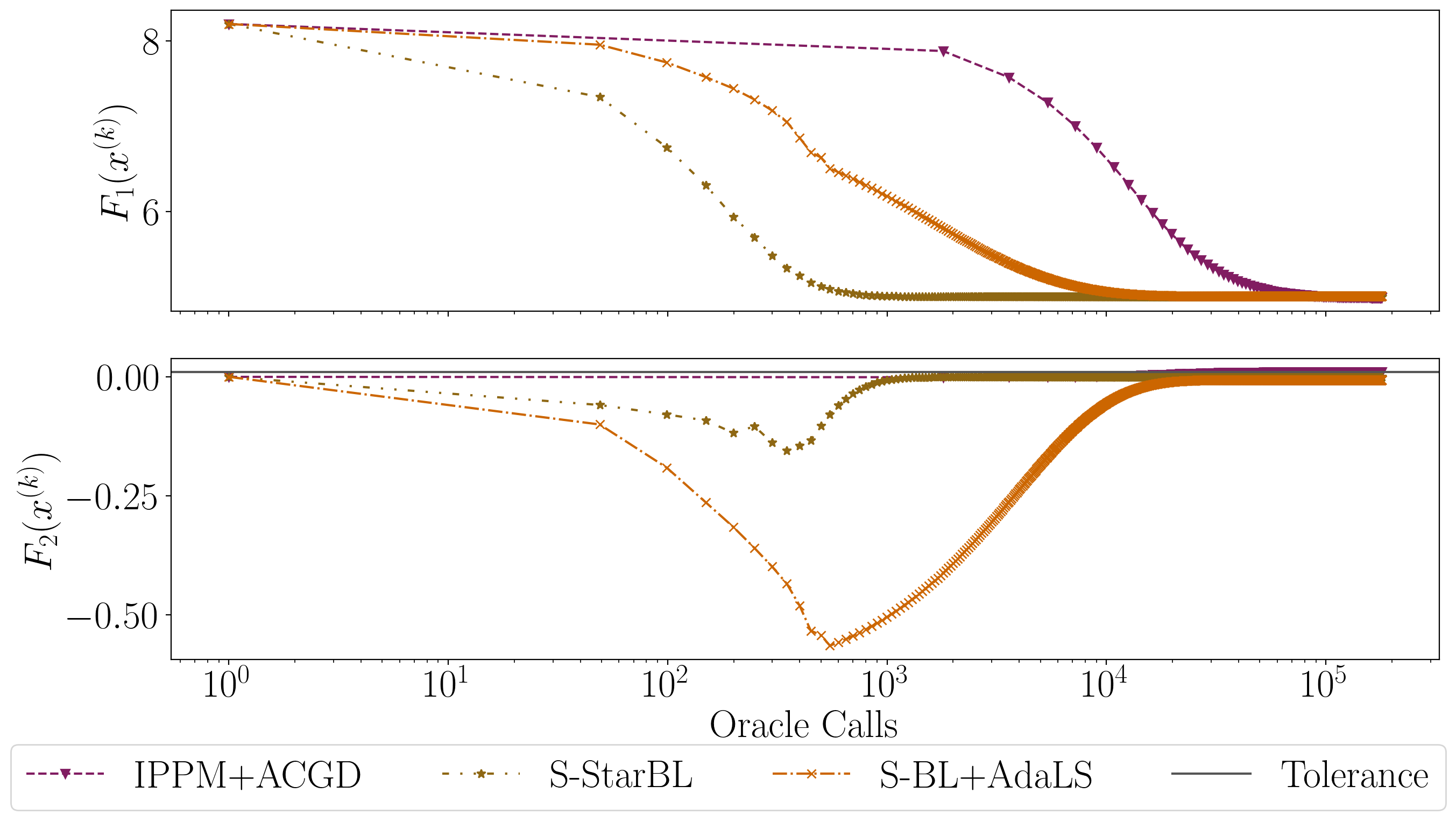
  }
  \caption{Objective value and constraint violation for $(\xxK)_k$, {oracle calls in log-scale.}}
  \label{fig:SmoothGeometricProgrammingFuncValConstrViol}
 \end{subfigure}
 \caption[Smooth Constrained Geometric Programming]{
  Solving the smooth
  \eqref{eq:ToyExampleGeometricProgramming} problem.
  Comparison of the
  \textcolor{black}{Inexact Proximal Point Method with
   Accelerated Constrained Gradient Descent (IPPM+ACGD)},
  cf. \Cref{subsubsec:IPPM_with_Slater},
  the
  Shifted Star Bundle-level (S-StarBL) with
  known $F_1^*$, cf. \Cref{subsec:PMMknownFstarNoCQ},
  and the
  and Shifted Bundle-level
  with Adaptive Line Search (S-BL+AdaLS)
  with unknown $F_1^*$ and $\eta_0 = 0$, cf.
  \Cref{subsec:PMMunknownFstarStrongDuality}.
  For both Bundle-level variants we only plot every 50th iterate
  to simplify the visualization.
 }
 \label{fig:SmoothGeometricProgrammingExample}
\end{figure}

\paragraph{Smooth Geometric Programming.} Now we test our more advanced algorithms IPPM+ACGD, S-StarBL, and
S-StarBL+AdaLS designed for smooth optimization. We use the previous
instance of \eqref{eq:ExampleIntroGeometricProgramming} problem from
\Cref{fig:IntroMinExample}:
\begin{align}
 \begin{aligned}
  \min_{\xx \in [0.4, 3]^2} \; & F_1(\xx)
  \define \xx_1 \cdot \xx_2 + \frac{4}{\xx_1} + \frac{1}{\xx_2}           ,     \\
  \mathrm{s.t.}\;              & F_2(\xx) \define \xx_1 \cdot \xx_2 - 1\leq 0 ,
 \end{aligned}
 \tag{Ex-CGP}
 \label{eq:ToyExampleGeometricProgramming}
\end{align}
where $\XX \define [0.4, 3]^2 \subset \RR^2_+$. As shown in
\Cref{sec:Introduction}, \eqref{eq:ToyExampleGeometricProgramming} is non-convex
in $\xx \in \XX$ but
hidden convex under the transformation function $c(\xx) \define (\log \xx_1,
 \log \xx_2)^\top$, and satisfies our assumptions with
$\mu_c = \frac{1}{3}, \rho = 1, G = 25.286$. This problem has the minimizer $\xxOpt = (2,0.5)^\top$ with $F_1(\xxOpt) = 5$, and $\uuOpt = (\log 2,-\log2)^\top$.
Slater's condition holds with $F_2((0.5, 0.5)) = -0.75 < 0$.

The results are shown in \Cref{fig:SmoothGeometricProgrammingExample}. We
observe that all three proposed algorithms successfully converge to the
global minima of this problem and respect the constraint satisfaction
during optimization process. IPPM+ACGD progresses slowly at the initial
phase due to the conflicting (non-convex) objective and constraint. The
S-StarBL method converges the fastest among the three, efficiently
leveraging the knowledge of the optimal value $\Fopt.$ In the absence of
$\Fopt$, our adaptive line-search (AdaLS) procedure converges slower than
the star version of the algorithm due to an underestimation of the
linearized objective caused by the misspecification of the optimal value
estimate $\eta_0$, cf. \Cref{fig:PolyakMinorantMethod}.

\subsection{Scalablility to High Dimensions}\label{subsec:NumExSmoothConstrainedGeometricProgramming}

\begin{figure}
 \centering
 \begin{subfigure}{0.49\textwidth}
  \centering
  \includegraphics[height=.65\linewidth, keepaspectratio]{
   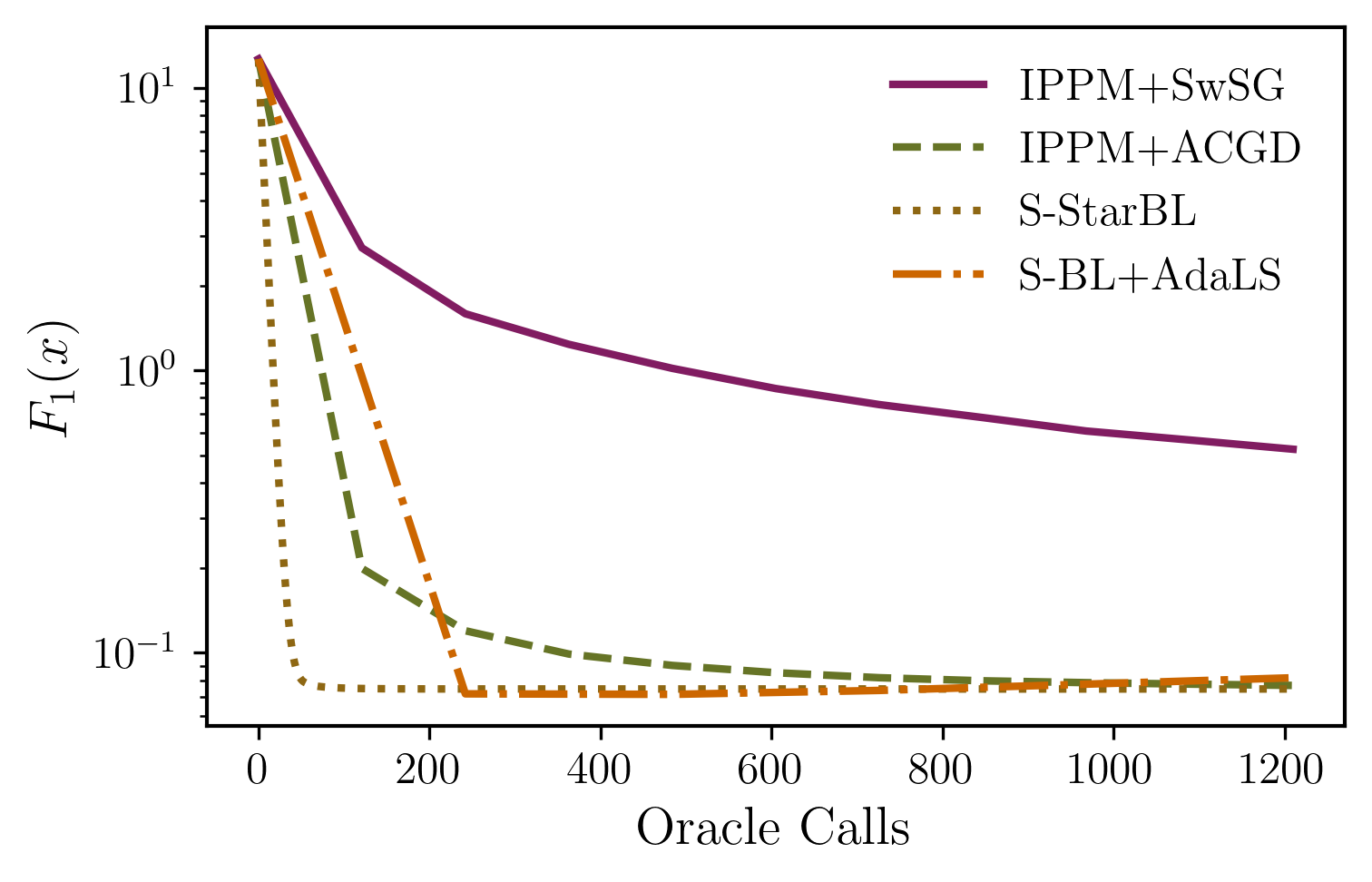
  }
  \caption{Objective value.}
  \label{fig:ObjectiveValue}
 \end{subfigure}
 \hfill
 \begin{subfigure}{0.49\textwidth}
  \centering
  \includegraphics[height=.65\linewidth,keepaspectratio]{
   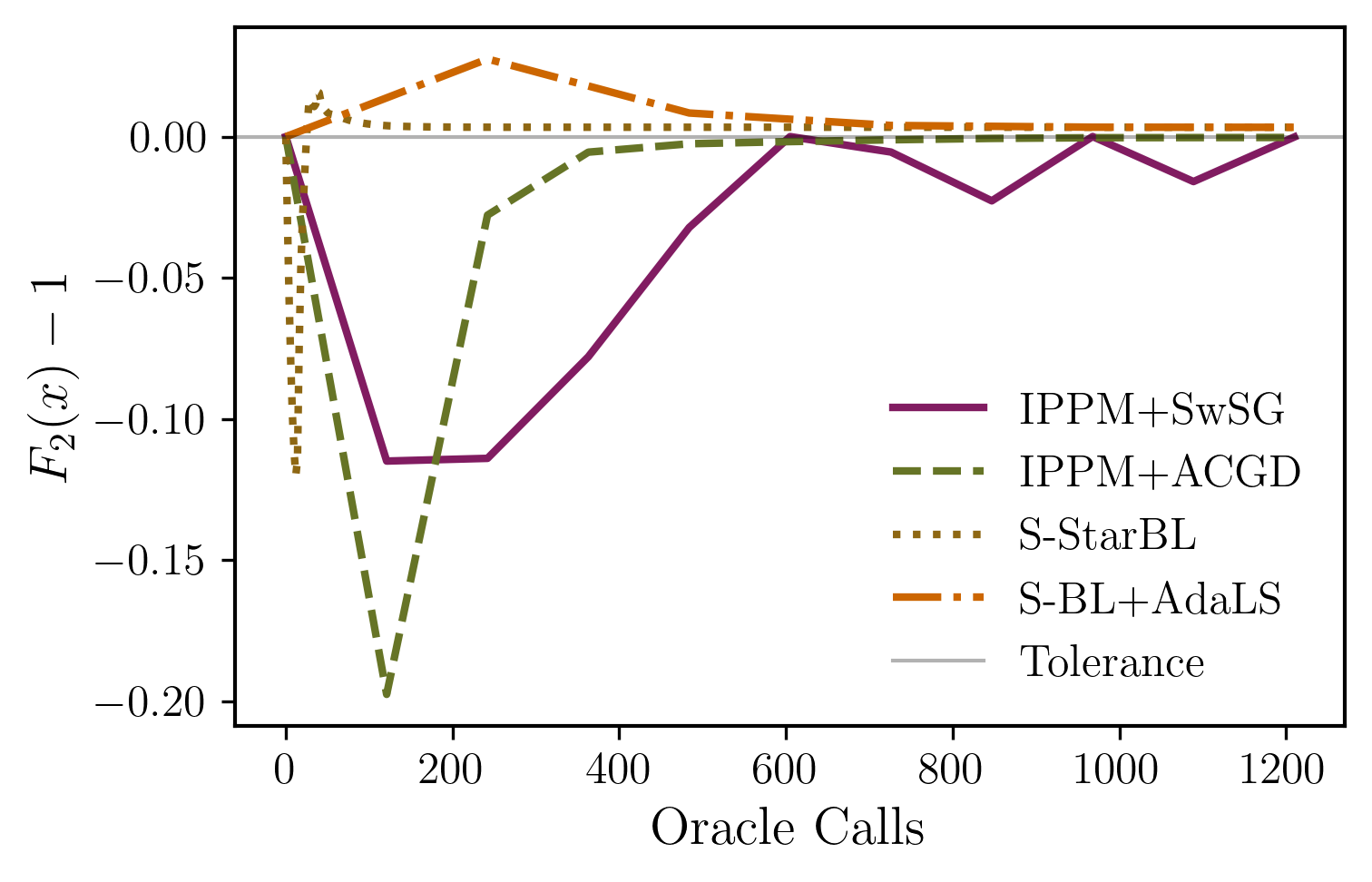
  }
  \caption{Constraint violation.}
  \label{fig:ConstraintViolation}
 \end{subfigure}
 \caption{Solving a high dimensional instance of \eqref{eq:ExampleIntroGeometricProgramming} using our IPPM+SwSG, IPPM+ACGD,
  S-StarBL and S-BL+AdaLS algorithms. We report the objective value
  and constraint violation vs. the cumulative oracle calls.} \label{fig:High_dim_experiment}
\end{figure}

Finally, we compare both proposed algorithms (with two variations each) on
a high dimensional constrained geometric programming problem
\eqref{eq:ExampleIntroGeometricProgramming}. We evaluate them on randomly
generated instances where both the objective and constraint are posynomials
$F_j(\xx) = \sum_{k_j} b_{k_j} \prod_i x_i^{a_{i,k_j}}$, $j\in \cb{1,2}.$
We fix the dimension $d = 100$ and the number of components $K_1 = 10$ and
$K_2 = 8$. Exponent rows $(a_{\cdot,k_j})$ are drawn uniformly from $[-0.5,
   0.5]$, and coefficients $b_{k_j}$ are sampled from a log-normal
distribution (normalized for the constraint so that $\sum_{k_2=1}^{K_2}
 b_{k_2} = 1$, making it nearly tight at $\xx = \mathbf{1}$). Iterates are
projected component-wise onto the box $[0.5, 2]^d$, starting from
$\xx^{(0)} = \mathbf{1}$, which is feasible. Before running our algorithms,
we use CVX solver \cite{cvx} to solve our problem in log-variables, which
provides the reference value $F_1^\star \approx 0.07114$. Our proximal
point algorithms use identical outer-loop settings: ten proximal epochs,
proximal weight $\rhoHat = 0.02$, tolerance $\tau = 10^{-3}$, oracle
tolerance $10^{-4}$, and inner budgets chosen so that each method performs
exactly 1,210 first-order oracle calls. IPPM+SwSG runs 121 switching
subgradient steps per epoch with step sizes $0.05/(t+1)$ and relaxation
$\alpha = 0.1$; IPPM+ACGD keeps $T_{\text{in}} = 60$, the same $\alpha =
 0.1$, and no additional constraint shift; S-StarBL performs 605
bundle-level projections with $\alpha = 0.3$, $\beta = 1$; and S-BL+AdaLS
executes five adaptive epochs of 121 bundle steps using $\alpha = 0.3$,
$\beta = 0.5$, $\lambda = 0.25$, and starting bound $\eta_0 = 0.5
 F_1^\star$.

The convergence and constraint satisfaction are illustrated in
\Cref{fig:High_dim_experiment}. We observe that all four algorithms
eventually satisfy the constraint up to a small tolerance of order
$10^{-3}$. Three algorithms, IPPM+ACGD, S-StarBL and S-BL+AdaLS, converge
to the global optimal value $\Fopt$, precomputed by the CVX solver. As
expected from theory, the convergence of IPPM+SwSG is relatively slow;
although it improves the objective value and respects the constraint, it
does not reach the optimum within the fixed number of oracle calls. Using
faster ACGD inner solver shows a clear benefit compared to SwSG and solves
the problem to high accuracy. The fastest algorithm is S-StarBL, which
converges using less than $100$ oracle calls, which is explained by its
access to the optimal value $\Fopt.$ S-BL+AdaLS also converges, but
slightly slower due to its online adaptive estimation of optimal value
$\Fopt$. We can also notice that S-BL+AdaLS initially violates the
constraint after the first outer loop iteration, however it quickly
recovers afterwards and its final iterate results in a small constraint
violation $\approx 3 \cdot 10^{-3}$. All above observations are in line
with our theoretical upper complexity bounds in
\Cref{sec:ComplexityAnalysis,sec:bundle_level}.

%% file: sections/08_conclusion.tex
\section{Conclusion}
In this work, we study a class of non-convex constrained problems under
hidden convexity, proposing the first efficient algorithmic solutions under
minimal assumptions. Several questions arise from our work:
\begin{itemize}
 \item While we established a strong baseline for hidden convex constrained
       optimization in terms of oracle complexity, our algorithms are fairly
       complicated. It would be interesting to develop simple single-loop methods
       without additional quadratic linearly constrained subproblems, e.g.,
       similar to
       \cite{polyakGeneralMethodSolving1967,huangOracleComplexitySingleLoop2023}.
 \item Our oracle complexities match the best-known complexity in unconstrained
       setting; however, we still do not know what are the minimax optimal
       gradient complexities of hidden convex optimization.

 \item While we follow a common strategy of quadratic regularization in our PPM
       framework, an exploration of other natural choices is an interesting
       direction for future work. For example, it is conceptually possible to
       regularize the iterates in the $\UU$-space and leverage hidden convexity.
       While in general such approaches require knowledge of transformation
       $c(\innerEmpty{})$, this direction might be potentially useful in specific
       applications, e.g. convex reinforcement learning
       \eqref{eq:IntroExampleCvx-CMDP}, where an approximation of
       $c(\innerEmpty{})$ can be used to improve regularization and/or act as a
       precondition.

\end{itemize}

%% file: sections/10_appendix.tex
\section{Technical Lemma}
\label{Appendix-sec:ProofsHiddenConvexClass}

\subsection{Proof of \Cref{lemma:HCContractionInequality}}
\label{lemma:HCContractionInequality-PROOF}

\begin{proof}
 By \eqref{eq:HCHCDefinitionConvexityInequality}, using the convexity of $H_i(\innerEmpty{}),$ we have
 \begin{align*}
  F_i(\xx_\alpha) & = F_i(c^{-1}((1-\alpha)c(\xx) + \alpha c(\yy)))   = H_i((1-\alpha)c(\xx) + \alpha c(\yy)) \\
                  & \leq (1-\alpha) H_i(c(\xx)) + \alpha H_i(c(\yy))                                          \\
                  & = (1-\alpha) F_i(\xx) + \alpha F_i(\yy) .
 \end{align*}
 By Lipschitz continuity of the inverse map \eqref{eq:HCDefinitionNormCfunctionInequality}, we have for any $\xx, \yy \in \XX$
 \begin{align*}
  \norm{\xx_\alpha - \xx}
  = \norm{c^{-1}((1-\alpha) c(\xx) + \alpha c(\yy)) - c^{-1}(c(\xx))}
  \leq \frac{1}{\mu_c} \norm{\alpha (c(\xx)- c(\yy))} .
 \end{align*}
\end{proof}

In the case of hidden \emph{strong} convexity, cf. \Cref{def:HC-strong},
above proposition can be strengthened by including the negative term on the
RHS of \eqref{eq:HstrongCHCDefinitionConvexityInequality}.

\section{Additional Material for \Cref{sec:AnalysisPPM}}
\label{Appendix-Sec:ProofsComplexityAnalysisPPM}

\subsection{Inexact PPM under Hidden Strong Convexity}
\label{appendix:InexactPPMunderHiddenStrongConvexity}
Similar to development in \cite{fatkhullinStochasticOptimizationHidden2024}, we can improve the convergence result in the case of hidden strong convexity.

\begin{definition}[Hidden Strong Convexity]\label{def:HC-strong}
 The problem \eqref{eq:MainProblem} is called \emph{hidden strongly convex}
 with modulus $\mu_c > 0, \mu_H > 0$, if its components satisfy the following underlying conditions.
 \begin{enumerate}
  \item The domain $\UU = c(\XX)$ is convex, the functions $H_1, H_2 : \UU \to \RR$
        are convex, i.e. satisfy for $i = 1, 2$ and for all $u, v \in \UU$ and all
        $\lambda \in [0,1]$
        \begin{align}
         H_{i}((1-\lambda)\uu + \lambda \vv) \leq (1-\lambda)H_i(\uu)
         + \lambda H_i(\vv)
         - \frac{(1-\lambda)\lambda \mu_H}{2} \norm{\uu - \vv}^2.
         \tag{HSC-1}
         \label{eq:HstrongCHCDefinitionConvexityInequality}
        \end{align}
        Additionally, we assume \eqref{eq:MainProblem}
        admits a solution $\uu^* \in \UU$ with its corresponding objective function
        value $\Fopt \eqdef H_1(\uu^*) = F_1(c^{-1}(\uu^*))$.
  \item The map $c : \XX \to \UU$ is invertible and there exists a $\mu_c > 0$ such
        that for all $\xx, \yy \in \XX$ it holds
        \begin{align}
         \norm{c(\xx) - c(\yy)} \geq \mu_c \norm{\xx - \yy} \tag{HSC-2}.
         \label{eq:HstrongCDefinitionNormCfunctionInequality}
        \end{align}
 \end{enumerate}
\end{definition}

\begin{theorem}[Hidden Strongly Convex, Inexact PPM]
 \label{thm:MainResultHSC-PPM}
 Assume \eqref{eq:MainProblem} is hidden strongly convex with modulus $\mu_H > 0$
 and \Assref{ass:AssumptionExactPPM} holds. Let us have access to an oracle
 $\Oracle_{\epsin}$ and an initial $\tau$-feasible point $\xxZero$, where
 $\tau > 0$ is the desired constraint violation. Given a lifting parameter
 $\rhoHat > \rho$, and a desired tolerance $\epsilon > 0$ for the optimality gap,
 then after $N \in \NN$ steps of \Cref{algo:PPM}, we have
 \begin{align*}
  F_1(\xxNnext) - \Fopt \leq (1-\alpha)^N \Delta_0 + \frac{\epsin}{\alpha}
  \leq \epsilon,
 \end{align*}
 with $\alpha \define \frac{\mu_c^2 \mu_H}{\rhoHat + \mu_c^2 \mu_H} \in \brac{0,1}$.
 In particular, setting $\rhoHat \define 2 \rho$ and
 $\epsin \define \frac{\alpha \epsilon}{2}$, the last iterate satisfies
 \begin{align*}
  F_1(\xxN) - \Fopt \leq \epsilon \text{ and } F_2(\xxN) \leq \tau
 \end{align*}
 after
 \begin{align*}
  N \geq \frac{\rhoHat + \mu_c^2 \mu_H}{\mu_c^2 \mu_H}
  \cdot \log \paren*{\frac{2\Delta_0}{\epsilon}} = \OO{\log \epsilon^{-1}}
 \end{align*}
 iterations, where $\Delta_0 \define F_1(\xxZero) - \Fopt$.
 The total rate is given by
 \begin{align*}
  \Ttotal \geq N \cdot \Tinner(\epsin) =
  \frac{2\rho + \mu_c^2 \mu_H}{\mu_c^2 \mu_H}
  \cdot \log \paren*{\frac{2\Delta_0}{\epsilon}}
  \cdot \Tinner\paren*{\frac{\mu_c^2 \mu_H}{2\rho + \mu_c^2 \mu_H}
   \frac{\epsilon}{2}}.
 \end{align*}
\end{theorem}
\begin{proof}
 Following similar arguments as in the proof of \Cref{thm:MainResultHCPPM}, the choice of
 $\alpha$ ensures the feasibility, i.e. $\xxK_\alpha$ satisfies
 \eqref{eq:strict_feasibility} with $-\alpha \cdot \tau < 0$.
 We obtain a recursion of the form
 \begin{align*}
  F_1(\xxKnext) - \Fopt \leq (1-\alpha)\paren{F_1(\xxK) - \Fopt} + \epsin,
 \end{align*}
 by the choice of $\alpha$. Thus, unrolling yields
 \begin{align*}
  F_1(\xxNnext) - \Fopt \leq (1-\alpha)^N \Delta_0 + \frac{\epsin}{\alpha}
  \leq \epsilon,
 \end{align*}
 where the last steps holds by the choice of $\alpha, \epsin$ and $N$ as
 in the Theorem statement.
\end{proof}

\subsection{Improving Optimality Criterion in the Unconstrained Case}
\label{subsubsec:OptCritUnconstrainedcase}
Based on the results in \Cref{subsec:InexactPPM}, we can improve the optimality criterion formulated in terms of the optimality gap of the Moreau envelop \cite{fatkhullinStochasticOptimizationHidden2024}.

\begin{corollary}[Hidden Convex, Unconstrained]
 \label{corr:UnconstrainedHiddenConvex}
 We assume \eqref{eq:MainProblem} is an \textbf{unconstrained} hidden convex
 problem, i.e. $F_2 \equiv 0$, \Assref{ass:AssumptionExactPPM} holds,
 and we have access to an approximate solution oracle
 $\Oracle_{\epsin}$.
 Given a target precision $\epsilon > 0$ and a lifting parameter
 $\rhoHat > \rho$, we run \Cref{algo:PPM}
 for $N \in \NN$ steps to obtain the contraction
 \begin{align*}
  F_1(\xxN) - \Fopt \leq (1 - \alpha)^{N} (F_1(x^{(0)}) - \Fopt)
  + \frac{\epsilon^2}{\alpha} + \frac{\alpha \DDUsq \rhoHat}{2 \mu_c^2},
 \end{align*}
 with $\alpha \define \frac{2 \mu_c^2 }{4\mu_c^2 + \DDUsq \rhoHat}\epsilon \in \brac{0,1}$.
 In particular, with $\rhoHat \define 2\rho$, we have
 \begin{align*}
  {F_1(\xxN) - \Fopt} \leq \varepsilon
 \end{align*}
 after
 $N \geq \frac{4\mu_c^2 + 2\DDUsq \rho}{2 \mu_c^2} \frac{1}{\epsilon}
  \log\left(\frac{\Delta_0 (4\mu_c^2 + 2\DDUsq \rho)}{2 \mu_c^2}
  \frac{1}{\epsilon}\right)$ outer loops.
 With $\rhoHat \define 2\rho$ and with the sub-gradient method (SM)
 as an inner solver, we require overall
 \begin{align*}
  \Ttotal & \geq N \cdot \TinnerUp{SM}(\epsilon^2)
  = \frac{4\mu_c^2 + 2\DDUsq \rho}{2 \mu_c^2} \frac{1}{\epsilon}
  \log\left(\frac{\Delta_0 (4\mu_c^2 + 2\DDUsq \rho)}{2 \mu_c^2}
  \frac{1}{\epsilon}\right)
  \cdot \frac{G^2}{\rho}\frac{1}{\epsilon^2}
  = \OOTilde{\epsilon^{-3}}
 \end{align*}
 sub-gradient evaluations.
\end{corollary}
\begin{proof}
 The proof is analogous to the proof of \Cref{thm:MainResultHCPPM},
 disregarding the analysis of the constraint violation.
\end{proof}

\subsection{Proofs of Corollaries~\ref{cor:FinalRateNonSmoothIPPM}, \ref{cor:FinalRateSmoothIPPM_ACGD} and \ref{cor:FinalRateSmoothIPPM_ACGD_Slater}}\label{subsec:ProofofCorollaries}

\begin{proof}[of \Cref{cor:FinalRateNonSmoothIPPM}]
 By the HC--Slater's \Cref{le:Slater_PPM_inner}.\ref{le:Slater_PPM_inner-ITEM1}, we know that for any $k\geq 0$, the problem \eqref{eq:InexactPPMProblemInnerLoopHCHC}
 satisfies $\frac{\alpha\, \tau}{2}$--Slater's condition, thus the ``stricter'' constraint $
  \phi_2^{(k)}(\xx) - \tau + \frac{\alpha\, \tau}{3} \leq 0 ,
 $ is still feasible. We apply the SwSG method to the sub-problem
 $$
  \min_{\xx\in \XX} \phi_1^{(k)}(\xx) , \qquad \text{s.t.} \quad
  \phi_2^{(k)}(\xx) - \tau + \frac{\alpha\, \tau}{3} \leq 0
 $$
 with a switching rule:
 $$
  \phi_2^{(k)}(\xx) - \tau + \frac{\alpha\, \tau}{3} \leq \epsin ,
 $$
 where $\epsin \leq \frac{\alpha\, \tau}{3} $, which ensures the feasibility of the output of the inner solver for sub-problem \eqref{eq:InexactPPMProblemInnerLoopHCHC}. It remains to set $$\alpha = \min\left\{\frac{2\mu_c^2 \epsilon}{3\rhoHat\DDUsq},
  \frac{\mu_c^2 \tau}{\rhoHat \DDUsq}\right\} , $$
 $$\epsin = \frac{\alpha}{3} \min\left\{\varepsilon,
  \tau \right\} \geq \frac{2\mu_c^2 }{9\rhoHat\DDUsq} \min\left\{ \varepsilon^2,
  \tau^2\right\} , $$
 and use \cite[Cor. 3.3]{jiaFirstOrderMethodsNonsmooth2025} with their $\tau$ in place of our $\epsin$ with the stepsize parameters
 $$
  \eta_t = \frac{2}{\rho (t+2) + \frac{144 \rho}{t+1}} .
 $$
 This concludes the proof.
\end{proof}

\begin{proof}[of \Cref{cor:FinalRateSmoothIPPM_ACGD}]
 By the HC--Slater's \Cref{le:Slater_PPM_inner}.\ref{le:Slater_PPM_inner-ITEM1},
 we know that for any $k\geq 0$, the problem
 \eqref{eq:InexactPPMProblemInnerLoopHCHC} satisfies
 $\frac{\alpha\, \tau}{2}$-Slater's condition.
 Consider the following shifted sub-problem
 \begin{equation}\label{eq:acgd_subproblem}
  \min_{\xx\in \XX} \phi_1^{(k)}(\xx) , \qquad \text{s.t.} \quad
  \phi_2^{(k)}(\xx) - \tau + \frac{\alpha\, \tau}{3} \leq 0 ,
 \end{equation}
 which satisfies $\frac{\alpha\, \tau}{6}$--Slater's condition. We know that
 $$
  \phi_1^{(k)}(\xxHatKnext) \geq \FLB{1}, \qquad \phi_1^{(k)}(\xx_{\alpha}^{(k)}) \leq F_1(\xx_{\alpha}^{(k)}) + \frac{\rhoHat\DDUsq \alpha^2}{2 \mu_c^2} ,
 $$
 where the last inequality holds by comparing \eqref{eq:req_main_x_alpha1} and \eqref{eq:req_main_x_alpha2}. Since $F_1(\xx_{\alpha}^{(k)}) \leq \FUB{1}$, and by the requirements on $\alpha$ and $\varepsilon$, we have $\frac{\rhoHat\DDUsq \alpha^2}{2\mu_c^2} \leq \frac{2 \mu_c^2 \varepsilon^2}{2 \rhoHat\DDUsq} \leq \varepsilon$.
 By invoking \cite[Lem. 3.1.21]{nesterovLecturesConvexOptimization2018},
 we can bound the Lagrange multiplier of this problem with
 $$
  \lambdaOpt \leq \frac{\phi_1^{(k)}(\xx_{\alpha}^{(k)}) - \phi_1^{(k)}(\xxHatKnext)}{\frac{\alpha\,\tau}{6}} \leq \frac{\FUB{1} - \FLB{1} + \varepsilon}{\frac{\alpha\,\tau}{6}}   .
 $$

 Since $\lambda^* < +\infty$, we can apply ACGD method
 \cite{zhangSolvingConvexSmooth2022}, to solve this sub-problem with the
 target accuracy $\epsin \leq \frac{\alpha\, \tau}{3} $. This will ensure
 the feasibility of the output of the inner solver for sub-problem
 \eqref{eq:InexactPPMProblemInnerLoopHCHC}.

 To compute the oracle complexity of this method, we need to upper bound the
 Lipschitz constant of the Lagrangian associated with sub-problem
 \eqref{eq:acgd_subproblem} (see (2.6) and (2.10) in
 \cite{zhangSolvingConvexSmooth2022}): $$ L(\Lambda_1) = L\cdot(1 +
  \Lambda_1), \qquad \text{where} \quad \Lambda_1 \define \{0\} \cup
  \{\lambda^* + \lambda : \lambda \in [0, 1] \} . $$ Using the upper bound on
 $\lambda^*$, we have $$ L(\Lambda_1) = \OOTilde{\frac{L (\FUB{1} - \FLB{1}
    + \varepsilon)}{\alpha \tau }} = \OOTilde{\frac{\rho L \DDUsq (\FUB{1} -
    \FLB{1} +\varepsilon)}{ \mu_c^2 \min\left\{\varepsilon, \tau \right\} \tau
   }}, $$ where we set $\alpha = \min\left\{\frac{2\mu_c^2
   \epsilon}{3\rhoHat\DDUsq}, \frac{\mu_c^2 \tau}{\rhoHat \DDUsq}\right\}$,
 and $\rhoHat = 2 \, \rho$. Now we set the accuracy for the inner solver
 ACGD as $$\epsin = \frac{\alpha}{3} \min\left\{\varepsilon, \tau \right\}
  \geq \frac{2\mu_c^2 }{9\rhoHat\DDUsq} \min\left\{ \varepsilon^2,
  \tau^2\right\} $$ and invoke Corollary 2 in
 \cite{zhangSolvingConvexSmooth2022}, yielding
 \begin{align*}
  \TinnerUp{ACGD} & =
  \left[ \sqrt{\frac{L(\Lambda_1)}{\rho}} + 1 \right]
  \log \left[ \frac{ \sqrt{L(\Lambda_1) L} \|\xxK - \xxHatKnext\|^2}{\epsin} + 1 \right] + 4.
  \\
                  & = \OOTilde{\frac{\sqrt{L(\Lambda_1)}}{ \sqrt{\rho} }} = \OOTilde{\frac{ \DDU \sqrt{L (\FUB{1} - \FLB{1} + \varepsilon)}}{ \mu_c \sqrt{\min\left\{\varepsilon,
     \tau \right\} \tau } }} .
 \end{align*}

 Crucially, we observe that $\TinnerUp{ACGD}$ depends on $\epsin$ only under
 the logarithmic term, which we hide in $\OOTilde{\innerEmpty{}}$ above. The
 total complexity is computed with a simple multiplication $$
  \TtotalUp{\smoothLabel} = N \cdot \TinnerUp{ACGD}\paren*{\epsin} . $$
\end{proof}

\begin{proof}[of \Cref{cor:FinalRateSmoothIPPM_ACGD_Slater}]
 By the HC--Slater's \Cref{le:Slater_PPM_inner}.\ref{le:Slater_PPM_inner-ITEM2},
 we know that for any $k\geq 0$, the problem
 \eqref{eq:InexactPPMProblemInnerLoopHCHC} satisfies
 $\frac{\beta\, \theta}{2}$-Slater's condition, where
 $\beta \leq  \frac{\mu_c^2 \theta}{\rhoHat \DDUsq}.$
 Consider the following shifted sub-problem
 \begin{equation}\label{eq:acgd_subproblem_Slater}
  \min_{\xx\in \XX} \phi_1^{(k)}(\xx) , \qquad \text{s.t.} \quad
  \phi_2^{(k)}(\xx) - \tau + \frac{\beta\, \theta}{3} \leq 0 ,
 \end{equation}
 which satisfies $\frac{\beta\, \theta}{6}$--Slater's condition.
 Similarly as in \Cref{cor:FinalRateSmoothIPPM_ACGD}, we can derive
 $$
  \phi_1^{(k)}(\xxHatKnext) \geq \FLB{1}, \qquad \phi_1^{(k)}(\xx_{\beta}^{(k)}) \leq F_1(\xx_{\beta}^{(k)}) + \frac{\rhoHat\DDUsq \beta^2}{2 \mu_c^2} \leq \FUB{1} + \theta,
 $$
 thus, the Lagrange multiplier of this problem is at most
 $$
  \lambda^*  \leq \frac{ \phi_1^{(k)}(\yy_{\beta}^{(k)}) - \phi_1^{(k)}(\xxHatKnext) }{\frac{\beta\,\theta}{6}} \leq \frac{\FUB{1} - \FLB{1} + \theta}{\frac{\beta\,\theta}{6}} .
 $$
 Following similar steps as the proof of \Cref{cor:FinalRateSmoothIPPM_ACGD} we establish
 $$
  L(\Lambda_1) = \OOTilde{\frac{L (\FUB{2} - \FLB{2} + \theta)}{\beta \theta }} = \OOTilde{\frac{\rho L \DDUsq (\FUB{2} - \FLB{2} + \theta)}{ \mu_c^2 \theta^2 }},
 $$
 where we set $\beta = \frac{\mu_c^2 \theta}{\rhoHat \DDUsq} \leq 1 $ and $\rhoHat = 2 \rho $. The oracle complexity of inner solver is at most
 $$
  T_{\text{inner}}^{\text{ACGD}} = \OOTilde{\frac{ \DDU \sqrt{L (\FUB{2} - \FLB{2} + \theta)}}{ \mu_c \theta  }}   .
 $$
 The bound on the total oracle complexity of \Cref{algo:PPM} $\TtotalUp{\smoothSlaterLabel}$ follows immediately form \Cref{thm:MainResultHCPPM}.
\end{proof}